\documentclass[12pt]{amsart} 
\usepackage{amsmath,amssymb,latexsym,amsthm,enumerate,amscd,hyperref,enumitem} 
\usepackage[abbrev]{amsrefs} 
\usepackage{graphicx}
\usepackage{pinlabel}
\usepackage[dvipsnames]{xcolor}
\usepackage{graphicx,tikz, tikz-cd}
\usepackage[margin=1.3in]{geometry}

\usepackage{yhmath}

\usepackage{mathtools}

\newtheoremstyle{dotless}{}{}{\itshape}{}{\bfseries}{}{ }{}

\newtheorem{Theorem}{Theorem}[section]

\newtheorem{Prop}[Theorem]{Proposition} 

\newtheorem{corollary}[Theorem]{Corollary} 
\newtheorem{Question}[Theorem]{Question} 
\newtheorem{lemma}[Theorem]{Lemma}

\theoremstyle{definition} 
\newtheorem{definition}[Theorem]{Definition} 
\newtheorem{remark}[Theorem]{Remark}
\newtheorem{example}[Theorem]{Example}

\newtheorem*{Acknowledgements}{Acknowledgements}

\newtheorem*{tits}{Tits Conjecture} 
 
\newtheorem*{generalizedtits}{Generalized Tits Conjecture}

\DeclareMathOperator{\aster}{\text{\LARGE{\textasteriskcentered}}}

\DeclareMathOperator{\piprod}{\raisebox{-0.1em}{\huge{$\pi$}}\kern -0.2em}

\newcommand{\cc}{{\mathbb C}}

\newcommand{\zz}{{\mathbb Z}}

\newcommand{\Id}{\operatorname{Id}}

\newcommand{\RA}{\operatorname{RA}}
\newcommand{\Mod}{\operatorname{Mod}}

\newcommand{\RAAG}{\operatorname{RAAG}}

\def\clap#1{\hbox to 0pt{\hss#1\hss}}

\newcommand{\comment}[1]{} 
 
\newcommand{\gs}{\sigma} 
 
\newcommand{\gt}{\tau}

\newcommand{\gG}{\Gamma}

\newcommand{\Homeo}{\operatorname{Homeo}}

\newcommand{\Lk}{\operatorname{Lk}}

	{ 
\end{enumerate}
} 

	{ 
\end{enumerate}
} 

\newenvironment{enumeratei'}{ 
\begin{enumerate}[\upshape (i)$'$]}
	{ 
\end{enumerate}
} 

\newenvironment{enumerate1'}{ 
\begin{enumerate}[\upshape (1)$'$]}
	{ 
\end{enumerate}
}

\newenvironment{enumeratea'}{ 
\begin{enumerate}[\upshape (a)$'$]}{ 
\end{enumerate}
}

\usepackage{color}
\usepackage[outerbars,color]{changebar}
\usepackage{ifthen}

\newboolean{usecolor}
\setboolean{usecolor}{true}

\ifthenelse{\boolean{usecolor}}
{
  \definecolor{colore}{cmyk}{0,1,0.6,0}
  \definecolor{coloregen}{cmyk}{0.7,0,1,0}
  \definecolor{coloresimo}{cmyk}{1,0.6,0,0}
}
{
  \definecolor{colore}{cmyk}{0,0,0,1}
  \definecolor{coloregen}{cmyk}{0,0,0,1}
  \definecolor{coloresimo}{cmyk}{0,0,0,1}
}

\numberwithin{equation}{section} 
\begin{document}
\title{Right-angled Artin subgroups of Artin groups}

\author{Kasia Jankiewicz  
\and{Kevin Schreve}
}
\address{Department of Mathematics, University of California, Santa Cruz, CA 95064}
\email{kasia@ucsc.edu}
\address{Department of Mathematics, Louisiana State University, Baton Rouge, LA 70806}
\email{kschreve@lsu.edu}
\date{\today} \maketitle

\begin{abstract}
The Tits Conjecture, proved by Crisp and Paris, states that squares of the standard generators of any Artin group generate an obvious right-angled Artin subgroup.  We consider a larger set of elements consisting of all the centers of the irreducible spherical special subgroups of the Artin group, and conjecture that sufficiently large powers of those elements generate an obvious right-angled Artin subgroup. This alleged right-angled Artin subgroup is in some sense as large as possible; its nerve is homeomorphic to the nerve of the ambient Artin group. We verify this conjecture for the class of locally reducible Artin groups, which includes all $2$-dimensional Artin groups, and for spherical Artin groups of any type other than $E_6, E_7, E_8$. We use our results to conclude that certain Artin groups contain hyperbolic surface subgroups, answering questions of Gordon, Long and Reid.
\end{abstract}






\section{Introduction}

Suppose $(W,S)$ is a Coxeter system (cf. \cite{bourbaki} or \cite{dbook}). 
This means that $W$ is a group, $S$ is a distinguished set of generators, and that $W$ has a presentation 
$$W := \langle s \in S| s^2 = (st)^{m_{st}} = 1 \rangle\ ,$$ where $m_{st} \in \{2,3,\dots\} \cup \{\infty\}$. 
Given a Coxeter system $(W,S)$ there is an associated \emph{Artin group} $A$. 
This group has one generator $x_s$ for each $s\in S$ and the \emph{braid relations}:
	\[
	\underbrace{x_sx_t\ \cdots}_{m_{st} \text{ terms}} = \underbrace{x_tx_s\ \cdots}_{m_{st} \text{ terms}},
	\]
where both sides of the equation are alternating words in $x_s$ and $x_t$, and where $m_{st}$ denotes the order of $st$ in $W$. The Artin group is \emph{right-angled} (and called a $\RAAG$) if $m_{st} \in \{2,\infty\}$.

Consider the subgroup of $A$ generated by the squares $x_s^2$. These elements are all contained in the \emph{pure Artin group} $PA$, which is the kernel of the canonical homomorphism $A \rightarrow W$ which sends $x_s$ to $s$. There are obvious commuting relations between the $x_s^2$, namely if $m_{st} = 2$, then $[x_s^2,x_t^2] = 1$. Crisp and Paris in \cite{cp} proved the remarkable fact that these are the only relations between these elements. This verified a conjecture of Tits, who had previously shown that the elements $\{x_s^2\}$ get sent to linearly independent elements in the abelianization of the pure Artin group. In fact, Crisp-Paris showed that the same is true if we replace to $2$ with any number $N\geq 2$. 

\begin{tits}[{\cite[Thm 1]{cp}}]\label{t:cp}
Let $A$ be an Artin group. For every $N\geq 2$, the subgroup generated by the set $\{x_s^N: s\in S\}$ is a $\RAAG$ with presentation $$\langle x_s^N| [x_s^N,x_t^N] = 1 \text{ if } m_{st} = 2\rangle$$
\end{tits}

This is one of the few theorems known to hold for all Artin groups (e.g.\ it is not known if all Artin groups are torsion-free, have solvable word problem, $\dots$). The Tits Conjecture had earlier been proved by Appel-Schupp for extra-large Artin groups (where $m_{st} > 3$) \cite{as}, by Collins for the braid groups \cite{collins}, and by Charney for the locally reducible Artin groups (where the associated Coxeter groups have each finite special subgroup a direct product of dihedral groups and $\mathbb Z/2$) \cite{charney}, see also \cite{dlser}, \cite{humphries} for more partial results. 

We will come back to Crisp and Paris' method later in the introduction. Very roughly speaking, they construct a representation from the Artin group into the mapping class group of some surface and show that it is faithful on the alleged $\RAAG$ subgroup. 

In this paper, we are interested in a conjectural generalization of the Tits Conjecture, which first appeared in \cite[Conj 4.9]{dls}. This generalization asks for a $\RAAG$ subgroup that is as ``large'' as possible in a certain sense. In particular, it contains the  $\RAAG$ subgroup that Crisp and Paris find and its nerve is homeomorphic to the nerve of the Artin group. We will now explain how there is a natural candidate for this larger $\RAAG$.

Given a Coxeter system $(W,S)$ and $T \subseteq S$, the subgroup $W_T$ generated by $ t \in T$ is called the \emph{special subgroup} corresponding to $T$. 
Then $(W_T,T)$ also is a Coxeter system. The subset $T$ is \emph{spherical} if $W_T$ is finite, in this case $W_T$ is a \emph{spherical special subgroup}. The subset $T$ is called \emph{reducible} if it decomposes as $T_1 \cup T_2$, where $m_{tt'} = 2$ for all $t \in T_1$ and $t' \in T_2$, otherwise $T$ is \emph{irreducible} (and we say $W_T$ is as well). If $T$ is reducible with decomposition $T = \cup_{i=1}^n T_i$, then $W_T = W_{T_1} \times \dots \times W_{T_n}$.

Similarly, the subgroup $A_T$ of $A$ generated by the $\{x_t\}_{t \in T}$ is the Artin group associated to the Coxeter system $(W_T, T)$ \cite{lek}.
If $T$ is irreducible and spherical, then $A_T$ is a \emph{irreducible, spherical special subgroup} of $A$. These spherical Artin groups are better understood than general Artin groups. 
Topologically, the pure Artin group (which in this case is finite index in $A$) is the fundamental group of an aspherical linear hyperplane arrangement in $\mathbb C^n$, which e.g.\ allows one to compute various cohomological invariants of $A$ \cite{deligne}. 
Combinatorially, these groups admit a Garside structure, which e.g.\ gives an easy to compute normal form  \cite{bs}. It is also known that the pure spherical Artin groups have an infinite center which is isomorphic to $\mathbb Z$ if the Coxeter group is irreducible. This center corresponds to the fundamental group of the fiber after projectivizing the hyperplane complement. We denote the generator of this by $\Delta_T^2$ (it is related to the longest element in $W_T$, see Section \ref{s:coxartin}).

We consider the subgroup of an Artin group $A$ generated by (powers of) the centers of the irreducible, spherical, pure Artin subgroups. There are some obvious commuting relations between these elements, namely if $T \subseteq U$ or if $m_{ut} = 2$ for all $u \in U$ and $t \in T$, then the corresponding centers $\Delta_U^2, \Delta_{T}^2$ commute. We write $[U,T]=1$ if  $m_{ut} = 2$ for all $u \in U$ and $t \in T$. As in the Tits Conjecture, we ask whether those are the only relations.

Let $\mathcal S$ be the set of all irreducible, spherical subsets $T\subseteq S$.
Let $\RA$ be a $\RAAG$ generated by the set $\{z_T\}_{T\in\mathcal S}$, with  the presentation
\begin{equation}\RA = \langle z_T \mid [z_T, z_U] = 1 \text{ if }U\subseteq T,\, T\subseteq U, \text{ or }[U,T] = 1\rangle.
\label{eq:raag} 
\end{equation}
By the above, there are homomorphisms $\Phi_N:\RA\to A$ so that $\Phi_N(z_T) = \Delta_T^{2N}$.
The map $\Phi_N$ is injective if and only if the subgroup of $A$ generated by $\{\Delta_{T}^{2N}\}_{T \in \mathcal S}$ is isomorphic to $\RA$. 
For a generator $x_s$, the infinite cyclic subgroup $\langle x_s \rangle$ of $A$ is itself an irreducible, spherical, special Artin subgroup, so the subgroup generated by $\{\Delta_{T}^{2N}\}_{ T \in \mathcal S}$ contains an appropriate $\RAAG$ subgroup found by Crisp and Paris. 
It will turn out that the injectivity (when we can verify it) of $\Phi_N$ will depend on $N$, unlike the original Tits Conjecture. Therefore, we propose the following.

\begin{generalizedtits} 
Let $A$ be an Artin group, $\RA$ the associated $\RAAG$ and $\Phi_N:\RA\to A$ the homomorphism defined above. Then $\Phi_N$ is injective for some $N$. 
\end{generalizedtits}

If we can verify that a specific homomorphism $\Phi_k$ is injective, then we say that $A$ satisfies the Generalized Tits Conjecture for $N = k$. The Generalized Tits Conjecture for $N=1$ was conjectured by Davis, Le, and the second author in {\cite[Conjecture 4.9]{dls}}\label{c:dls}. 
This turns out to be too optimistic in general, though we can show it for a reasonably large class of Artin groups. In Example~\ref{ex:braids} we show that $\Phi_1$ is not injective for braid groups on at least $4$ strands. However, a theorem of Koberda (Theorem~\ref{t:koberda}) implies that braid groups satisfy the Generalized Tits Conjecture (for some $N > > 0$).  

We claim this is a good generalization of the Tits Conjecture. 
As evidence, recall that there is a simplicial complex $L$ ($=L(W,S)$), called the \emph{nerve}.
Its vertex set is $S$ and a subset $T \subseteq S$ spans a simplex of $L$ if and only if $T$ is spherical.
Davis and Huang showed in \cite{dh16} that the nerve $L'$ of the $\RAAG$ $\RA$ defined above is a partial barycentric subdivision of the nerve $L$. 
In particular the nerve of the Artin group is homeomorphic to the nerve of the alleged $\RAAG$ subgroup. 
This is desirable since many topological properties of this nerve are related to  algebraic properties of the Artin group. 
For example, contingent on the $K(\pi,1)$-conjecture, $\RA$ and $A$ will have the same cohomological dimension, their compactly supported/$\ell^2$-cohomology will be nontrivial in the same dimensions, etc. 
In this sense, the alleged subgroup is as ``large" a $\RAAG$ subgroup as one can expect to find in $A$. 
Davis and Huang in \cite{dh16} were interested in determining the minimal dimensional manifold model for a classifying space $BA$ (see also Le's thesis~\cite{le}). These partial subdivisions $L'$ appeared earlier in \cite{djs1}, \cite{djs2}.

\subsection{Results}

We show the Generalized Tits Conjecture holds for Charney's class of locally reducible Artin groups. This includes all 2-dimensional Artin groups and Artin groups with $m_{st} \ne 3$ for all $s,t \in S$. The irreducible spherical Artin subgroups correspond to edges of the nerve, so the only new generators that we are considering come from centers of the Artin subgroups generated by $\langle s,t \rangle$ for $m_{st} \ge 3$. We can show the following:

\begin{Theorem}\label{t:main2}
Artin groups with $m_{st} \ne 3$ for all $s,t \in S$ satisfy the Generalized Tits Conjecture with $N=1$.
Locally reducible Artin groups satisfy the Generalized Tits Conjecture with $N=2$. 
\end{Theorem}

We also prove that the Generalized Tits Conjecture holds for large $N$ for the following spherical Artin groups. 

\begin{Theorem}\label{t:main}
The Generalized Tits Conjecture holds with $N$ sufficiently large for all spherical Artin groups except for those of type $E_n$.
\end{Theorem}

Our methods do not work for these remaining exotic cases. The hardest part of  Theorem \ref{t:main} is confirming the conjecture for the Artin groups of type $D_n$ (the conjecture for Artin groups of type $B_n$ also follows from Koberda's result). For technical reasons we have to assume that $N$ is even, though we suspect it works for general large $N$. 

\subsection{Applications}
Here are some immediate applications of our results. The rough moral here is that if an Artin group satisfies the Generalized Tits Conjecture, then its subgroups are as complicated as the subgroups of the corresponding $\RAAG$. 
In particular, we give a new proof of Wise's result that the spherical Artin group $A$ of type $H_3$ is incoherent, and can show that $A$ (along with many other Artin groups) contains a closed hyperbolic surface subgroup, answering questions of Gordon-Long-Reid \cite{glr}. 
The advantage of our argument is that the same proof works for both questions; the $\RAAG$ subgroup of $A$ has a nerve which is a cone on a pentagon, and it is easy to see this subgroup is incoherent and contains hyperbolic surface subgroups.

\subsection{Outlines of the proofs}

Our methods of proof in the case of spherical Artin group and locally reducible Artin groups are very different. 
For the locally reducible Artin groups, we use similar methods to \cite{charney}. Charney showed that the Deligne complexes of locally reducible Artin groups are CAT(0). She then constructed a cube complex with an action of the predicted RAAG, and showed that it isometrically embeds in the Deligne complex, using arguments from CAT(0) geometry. The crucial case to understand is the dihedral Artin groups $A_{2m}$ (i.e.\ $(W,S)$ is a dihedral group $D_{2m}$), since these appear in the links of vertices in the larger Deligne complex. 
The RAAG that we consider is larger, so we are trying to isometrically embed a larger complex into the Deligne complex. Again, we need to understand the dihedral Artin groups. In this case, our complex looks roughly like a $\mathbb Z's$ worth of Charney's complex, and we need to show that the pieces embed pairwise  orthogonally in the Deligne complex of the dihedral Artin group.

For the spherical Artin groups, we follow Crisp and Paris. We start with a representation from the Artin group into the mapping class group of a surface $\Sigma$. For spherical Artin groups of type $A_n, D_n$ and $E_n$, these representations are classical and due to Perron and Vannier \cite{pv}. The generators of the Artin groups map to Dehn twists around simple closed curves, and powers of the centers of irreducible spherical Artin subgroups map to Dehn twists about the boundary curves of connected subsurfaces of $\Sigma$. Then Koberda's result implies that high powers of Dehn twists around this collection of curves generates a $\RAAG$ subgroup of the mapping class group of $\Sigma$. The reader might suspect that we are done now, and the proof trivially follows from Koberda's theorem (as we initially thought). However, the boundary of these subsurfaces is not necessarily connected, so each center maps to a product of Dehn twists about disjoint simple closed curves in $\Sigma$. Therefore, we have to study the following question:

\begin{Question}
Let $\RA$ be a $\RAAG$. Let $\{w_i\}$ be a collection of elements of $\RA$, where each $w_i$ is a product of commuting generators of $\RA$. Is the subgroup of $\RA$ generated by $w_i$ the (obvious) $\RAAG$?
\end{Question}

This turns out to be subtle, and was also considered in \cite{cp} and \cite{koberda}. Koberda and Crisp-Paris gave different conditions on the $\{w_i\}$ which guaranteed a positive answer to the above question. These conditions fail for the system of curves produced by the Perron-Vannier representation of the Artin groups of type $D_n$. Our main work in this case is to generalize Koberda's condition to a condition that the curves in this system satisfy. This new condition may be of independent interest.

 Unfortunately, our conditions do not work for the RAAG and subwords produced from the Perron-Vannier representation of the Artin groups of type $E_n$. Even worse, for $E_7$ and $E_8$ we can find words in the alleged $\RAAG$ subgroup which are in the kernel of the representation, see Subsection \ref{ss:e7}. One can check by hand that these words correspond to nontrivial elements of the Artin group, so these mapping class group representations are not faithful enough to be used to verify the conjecture, even for spherical Artin groups. 


\subsection{Organization of the paper}
In Section \ref{s:coxartin}, we give some background on Coxeter and Artin groups. Section \ref{s:lr} is devoted to proving Theorems \ref{t:main2} and \ref{t:main3}. In Section \ref{s:raags} we study RAAG subgroups of RAAG's,  and it can be read independently of the rest of the paper. In Sections \ref{s:small} and \ref{s:folding} we use this to prove Theorem \ref{t:main} first in the small type case (i.e.\ where all $m_{st}\leq 3$), and then for the remaining cases. Section \ref{s:applications} discusses some applications.

\begin{Acknowledgements}
The second author wishes to thank Mike Davis and Giang Le for very helpful discussions. We also thank the anonymous referee for their comments and suggestions.
This material is based upon work done while the second author was supported by the National Science Foundation under Award No.\ 1704364. 
The first author was supported by the National Science Foundation
under Grant No.\ DMS-1928930, as a participant in a program hosted
by the Mathematical Sciences Research Institute in Berkeley, California, during
the Fall 2020 semester, and the National Science Foundation Grant No.\ DMS-2105548.
\end{Acknowledgements}

\section{Coxeter groups and Artin groups}\label{s:coxartin}

Let $(W,S)$ be a Coxeter system, and let $A$ be the corresponding Artin group. 
There is a canonical surjection $p:A \rightarrow W$ which sends $x_s$ to $s$. The kernel of $p$ is called the \emph{pure Artin group} and denoted $PA$. It is obviously finite index in $A$ if and only if $W$ is a finite Coxeter group. 
There is also a canonical set-theoretic section $\sigma: W \rightarrow A$ of $p$ which takes a reduced positive word $w$ in $W$ to the same word in $A$. It follows from Tits' solution to the word problem for Coxeter groups that this does not depend on the choice of reduced expression for $w$.

\subsection{Coxeter diagrams}
When there are many commuting generators, Coxeter groups and Artin groups can be efficiently described in terms of their \emph{Coxeter diagrams}. 
Given a Coxeter system $(W,S)$, we consider a graph $\gG$ whose vertex set is the set of standard generators $S$, and whose edges correspond to the pairs of generators that do not commute, i.e.\ there is an edge between $s$ and $t$ if and only if $s \ne t$ and $m_{st} \ne 2$. If $m_{st} \in \{4,5,6,\dots\} \cup \{\infty\}$, then the edge between $s$ and $t$ is labeled with $m_{st}$. Otherwise, i.e.\ when $m_{st} = 3$, the edge has no label. Given such a graph $\gG$, we will denote the corresponding Artin group by $A_\gG$.

Note that if the Coxeter graph $\Gamma$ has multiple connected components $\gG_i$, then the Artin group $A_\gG$ splits as the direct product $\prod_i A_{\gG_i}$. If the Coxeter graph is connected, we say that $A_\gG$ is an \emph{irreducible} Artin group. 
We say an Artin group is \emph{small-type} if $m_{st} \in \{2,3\}$ for all $s,t \in S$. These correspond to Coxeter diagrams which are unlabeled graphs (where there are no loops or multiple edges). The finite Coxeter groups with connected Coxeter diagram were classified by Coxeter, and correspond to the Coxeter graphs in Figure~\ref{fig:fincox}. The small type irreducible spherical Artin groups therefore split into two infinite families; type $A_n$ and $D_n$, and three exotic cases $E_6, E_7$ and $E_8$. Each of the other spherical Artin groups injects into a product of small-type spherical Artin groups, see \cite{crisp} or Section \ref{s:folding}.

\subsection{Fundamental elements and Coxeter elements}

Let $(W,S)$ be a Coxeter system with $W$ finite. A \emph{Coxeter element} of $W$ is a product of all the generators of $S$, in any order, where each generator appears exactly once in the product. Different orderings produce conjugate Coxeter elements \cite[Chap V.6.1]{bourbaki}.
The \emph{Coxeter number} $h$ of $(W,S)$ is the order of a (any) Coxeter element in $W$.

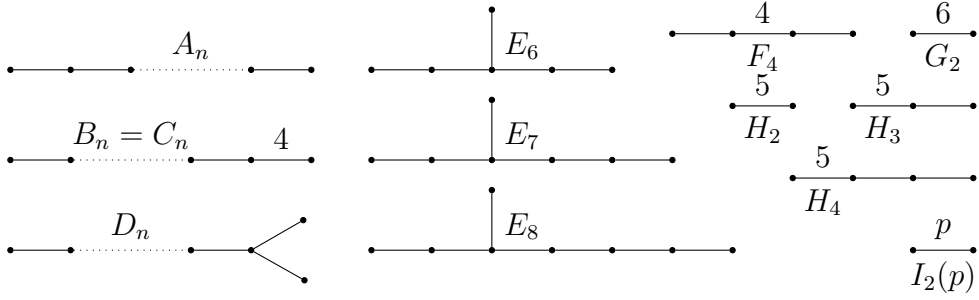
\begin{figure}
\centering
\begin{tikzpicture}[scale = .8]

\begin{scope}
\node[circle, draw, fill, inner sep = 0pt,minimum width = 2pt] (a) at (0,0) {};
\node[circle, draw, fill, inner sep = 0pt,minimum width = 2pt] (b) at (1,0) {};
\node[circle, draw, fill, inner sep = 0pt,minimum width = 2pt] (c) at (2,0) {};
\node[circle, draw, fill, inner sep = 0pt,minimum width = 2pt] (d) at (4,0) {};
\node[circle, draw, fill, inner sep = 0pt,minimum width = 2pt] (e) at (5,0) {};
\draw (a) -- (b) -- (c);
\draw[dotted] (c) -- (d) node[draw=none,fill=none,midway,above] {$A_n$};
\draw (d) -- (e);
\end{scope}
\begin{scope}[shift={(0,-1.5)}]
\node[circle, draw, fill, inner sep = 0pt,minimum width = 2pt] (a) at (0,0) {};
\node[circle, draw, fill, inner sep = 0pt,minimum width = 2pt] (b) at (1,0) {};
\node[circle, draw, fill, inner sep = 0pt,minimum width = 2pt] (c) at (3,0) {};
\node[circle, draw, fill, inner sep = 0pt,minimum width = 2pt] (d) at (4,0) {};
\node[circle, draw, fill, inner sep = 0pt,minimum width = 2pt] (e) at (5,0) {};
\draw (a) -- (b);
\draw[dotted] (b) -- (c) node[draw=none,fill=none,midway,above] {$B_n=C_n$};
\draw (c) -- (d);
\draw (d) -- (e) node[draw=none,fill=none,midway,above] {$4$};
\end{scope}
\begin{scope}[shift={(0, -3)}]
\node[circle, draw, fill, inner sep = 0pt,minimum width = 2pt] (a) at (0,0) {};
\node[circle, draw, fill, inner sep = 0pt,minimum width = 2pt] (b) at (1,0) {};
\node[circle, draw, fill, inner sep = 0pt,minimum width = 2pt] (c) at (3,0) {};
\node[circle, draw, fill, inner sep = 0pt,minimum width = 2pt] (d) at (4,0) {};
\node[circle, draw, fill, inner sep = 0pt,minimum width = 2pt] (e) at (4.866,0.5) {};
\node[circle, draw, fill, inner sep = 0pt,minimum width = 2pt] (f) at (4.886,-0.5) {};
\draw (a) -- (b);
\draw[dotted] (b) -- (c) node[draw=none,fill=none,midway,above] {$D_n$};
\draw (c) --(d) -- (e);
\draw (d) -- (f);
\end{scope}
\begin{scope}[shift={(6,0)}]
\node[circle, draw, fill, inner sep = 0pt,minimum width = 2pt] (a) at (0,0) {};
\node[circle, draw, fill, inner sep = 0pt,minimum width = 2pt] (b) at (1,0) {};
\node[circle, draw, fill, inner sep = 0pt,minimum width = 2pt] (c) at (2,0) {};
\node[circle, draw, fill, inner sep = 0pt,minimum width = 2pt] (d) at (3,0) {};
\node[circle, draw, fill, inner sep = 0pt,minimum width = 2pt] (e) at (4,0) {};
\node[circle, draw, fill, inner sep = 0pt,minimum width = 2pt] (f) at (2,1) {};
\draw (a) -- (b) -- (c) -- (f);
\draw (c) -- (d) node[draw=none,fill=none,midway,above] {$E_6$};
\draw (d) -- (e);
\end{scope}
\begin{scope}[shift={(6,-1.5)}]
\node[circle, draw, fill, inner sep = 0pt,minimum width = 2pt] (a) at (0,0) {};
\node[circle, draw, fill, inner sep = 0pt,minimum width = 2pt] (b) at (1,0) {};
\node[circle, draw, fill, inner sep = 0pt,minimum width = 2pt] (c) at (2,0) {};
\node[circle, draw, fill, inner sep = 0pt,minimum width = 2pt] (d) at (3,0) {};
\node[circle, draw, fill, inner sep = 0pt,minimum width = 2pt] (e) at (4,0) {};
\node[circle, draw, fill, inner sep = 0pt,minimum width = 2pt] (g) at (5,0) {};
\node[circle, draw, fill, inner sep = 0pt,minimum width = 2pt] (f) at (2,1) {};
\draw (a) -- (b) -- (c) -- (f);
\draw (c) -- (d) node[draw=none,fill=none,midway,above] {$E_7$};
\draw (d) -- (e) -- (g);
\end{scope}
\begin{scope}[shift={(6,-3)}]
\node[circle, draw, fill, inner sep = 0pt,minimum width = 2pt] (a) at (0,0) {};
\node[circle, draw, fill, inner sep = 0pt,minimum width = 2pt] (b) at (1,0) {};
\node[circle, draw, fill, inner sep = 0pt,minimum width = 2pt] (c) at (2,0) {};
\node[circle, draw, fill, inner sep = 0pt,minimum width = 2pt] (d) at (3,0) {};
\node[circle, draw, fill, inner sep = 0pt,minimum width = 2pt] (e) at (4,0) {};
\node[circle, draw, fill, inner sep = 0pt,minimum width = 2pt] (g) at (5,0) {};
\node[circle, draw, fill, inner sep = 0pt,minimum width = 2pt] (h) at (6,0) {};
\node[circle, draw, fill, inner sep = 0pt,minimum width = 2pt] (f) at (2,1) {};
\draw (a) -- (b) -- (c) -- (f);
\draw (c) -- (d) node[draw=none,fill=none,midway,above] {$E_8$};
\draw (d) -- (e) -- (g) -- (h) ;
\end{scope}
\begin{scope}[shift={(11,0.6)}]
\node[circle, draw, fill, inner sep = 0pt,minimum width = 2pt] (a) at (0,0) {};
\node[circle, draw, fill, inner sep = 0pt,minimum width = 2pt] (b) at (1,0) {};
\node[circle, draw, fill, inner sep = 0pt,minimum width = 2pt] (c) at (2,0) {};
\node[circle, draw, fill, inner sep = 0pt,minimum width = 2pt] (d) at (3,0) {};
\draw (a) -- (b);
\draw (b) -- (c) node[draw=none,fill=none,midway,above] {$4$} node[draw=none,fill=none,midway,below] {$F_4$};
\draw (c) -- (d);
\end{scope}
\begin{scope}[shift={(15,0.6)}]
\node[circle, draw, fill, inner sep = 0pt,minimum width = 2pt] (a) at (0,0) {};
\node[circle, draw, fill, inner sep = 0pt,minimum width = 2pt] (b) at (1,0) {};
\draw (a) -- (b) node[draw=none,fill=none,midway,above] {$6$} node[draw=none,fill=none,midway,below] {$G_2$};
\end{scope}
\begin{scope}[shift={(12,-0.6)}]
\node[circle, draw, fill, inner sep = 0pt,minimum width = 2pt] (a) at (0,0) {};
\node[circle, draw, fill, inner sep = 0pt,minimum width = 2pt] (b) at (1,0) {};
\draw (a) -- (b) node[draw=none,fill=none,midway,above] {$5$} node[draw=none,fill=none,midway,below] {$H_2$};
\end{scope}
\begin{scope}[shift={(14,-0.6)}]
\node[circle, draw, fill, inner sep = 0pt,minimum width = 2pt] (a) at (0,0) {};
\node[circle, draw, fill, inner sep = 0pt,minimum width = 2pt] (b) at (1,0) {};
\node[circle, draw, fill, inner sep = 0pt,minimum width = 2pt] (c) at (2,0) {};
\draw (a) -- (b) node[draw=none,fill=none,midway,above] {$5$} node[draw=none,fill=none,midway,below] {$H_3$};
\draw (b) -- (c);
\end{scope}
\begin{scope}[shift={(13,-1.8)}]
\node[circle, draw, fill, inner sep = 0pt,minimum width = 2pt] (a) at (0,0) {};
\node[circle, draw, fill, inner sep = 0pt,minimum width = 2pt] (b) at (1,0) {};
\node[circle, draw, fill, inner sep = 0pt,minimum width = 2pt] (c) at (2,0) {};
\node[circle, draw, fill, inner sep = 0pt,minimum width = 2pt] (d) at (3,0) {};
\draw (a) -- (b) node[draw=none,fill=none,midway,above] {$5$} node[draw=none,fill=none,midway,below] {$H_4$};
\draw (b) -- (c) -- (d);
\end{scope}
\begin{scope}[shift={(15,-3)}]
\node[circle, draw, fill, inner sep = 0pt,minimum width = 2pt] (a) at (0,0) {};
\node[circle, draw, fill, inner sep = 0pt,minimum width = 2pt] (b) at (1,0) {};
\draw (a) -- (b) node[draw=none,fill=none,midway,above] {$p$} node[draw=none,fill=none,midway,below] {$I_2(p)$};
\end{scope}

\end{tikzpicture}
\caption{Coxeter graphs of the irreducible finite Coxeter groups.}\label{fig:fincox}

\end{figure}

A \emph{reflection} in $W$ is a conjugate of an element of $S$. 
Each finite Coxeter group has a unique \emph{longest element} $w_S$. This can be characterized as the unique element for which $\ell(sw_S) = \ell(w_Ss) < \ell(w_S)$ for all $s \in S$, where $\ell(w)$ is the minimal length of a representative for $w$. The length of $w_S$ is precisely the number of reflections in $W$. Conjugation by $w_S$ induces an involution of the Coxeter diagram $\gG_S$, in the sense that generators are sent to generators and the relations are preserved. It follows from this that $w_s$ is in the center of $W$ if and only if this involution is trivial. This involution happens to be nontrivial if and only if the Coxeter group is of type $A_n$, $D_n$ with $n$ odd, $E_6$ or $I_2(p)$ for $p$ odd. 

The image of $w_S$ in $A$ under the section $\sigma: W \rightarrow A$ will be called the \emph{fundamental element}  of $A$, and we will denote it by $\Delta$. Each spherical Artin group has an infinite cyclic center which is generated by either $\Delta$ or $\Delta^2$ (depending as above whether the Coxeter group has a nontrivial center). For simplicity, we will only deal with the squares $\Delta^2$, which generate the center of the pure Artin groups, see \cite[Thm 4.7]{ot}. We record the following lemma in \cite{bs}, which will be used in Section~\ref{s:folding}.

\begin{lemma}\label{l:bs}
Let $A$ be a spherical Artin group. We have that $\Delta^2 = \gs(c)^h$, where $c$ is any Coxeter element and $h$ is the Coxeter number.
\end{lemma}

\subsection{Subdivisions}\label{sec:subdivision}

The \emph{nerve} $L$ of a Coxeter system $(W,S)$ is a simplicial complex $L$ with vertex set $S$ where a subset $T\subseteq S$ spans a simplex of $L$ if and only if $T$ is spherical. 
Let $\gs$ be a simplex in $L$ corresponding to a spherical subset $T$, and $A_T$ the corresponding spherical special Artin subgroup of $A$. Davis and Huang described a partial barycentric subdivision $\gs_\oslash$ of $\gs$ where the vertices of $\gs_\oslash$ correspond to irreducible subsets of $T$. For $U \subseteq T$, we think of the vertex corresponding to $A_U$ as the barycenter of the associated simplex in $\gs$. There are edges between two vertices $U$ and $U'$ if and only if $U \subseteq U'$, $U' \subseteq U$, or $[U,U']=1$. See Figure \ref{fig:subdivision} for the subdivision corresponding to the braid group on $4$ strands.

In terms of the Coxeter diagram $\gG$, the irreducible spherical Artin subgroups of $A_\gG$ correspond to connected spherical subdiagrams. In this case, there is an edge between two connected subsets $U,T$ of $\gG_\gs$ if and only if $U \subseteq T$, $T \subseteq U$, or $U$ and $T$ have distance $> 2$ in $\gG$.

\begin{figure}
\centering
\begin{tikzpicture}
\begin{scope}
\node[circle, draw, fill, inner sep = 0pt,minimum width = 2pt, label=above:$t$] (t) at (90:1.5) {};
\node[circle, draw, fill, inner sep = 0pt,minimum width = 2pt, label=left:$s$] (s) at (210:1.5) {};
\node[circle, draw, fill, inner sep = 0pt,minimum width = 2pt, label=right:$u$] (u) at (330:1.5) {};
\draw (t) -- (s) node[draw=none,fill=none,midway,left] {$2$};
\draw (s) -- (u) node[draw=none,fill=none,midway,below] {$3$};
\draw (u) -- (t) node[draw=none,fill=none,midway,right] {$3$};
\end{scope}
\begin{scope}[shift={(5,0)}]
\node[circle, draw, fill, inner sep = 0pt,minimum width = 2pt] (stu) at (0,0) {};
\node[circle, draw, fill, inner sep = 0pt,minimum width = 2pt, label=above:$t$] (t) at (90:1.5) {};
\node[circle, draw, fill, inner sep = 0pt,minimum width = 2pt, label=left:$s$] (s) at (210:1.5) {};
\node[circle, draw, fill, inner sep = 0pt,minimum width = 2pt, label=right:$u$] (u) at (330:1.5) {};
\node[circle, draw, fill, inner sep = 0pt,minimum width = 2pt] (tu) at (30:0.75) {};
\node[circle, draw, fill, inner sep = 0pt,minimum width = 2pt] (su) at (270:0.75) {};
\draw (t) -- (s);
\draw (s) -- (u);
\draw (u) -- (t);
\draw (stu) -- (tu);
\draw (stu) -- (su);
\draw (stu) -- (t);
\draw (stu) -- (u);
\draw (stu) -- (s);
\end{scope}
\end{tikzpicture}

\caption{The subdivision $\gs_{\oslash}$ for the braid group on four strands.}\label{fig:subdivision}
\end{figure}
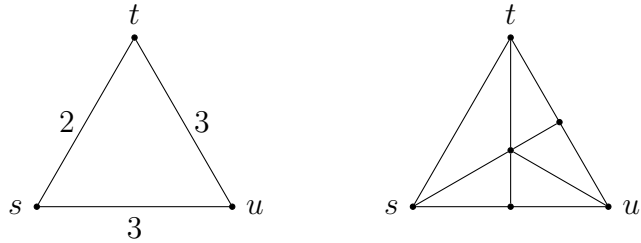
The subdivision $\gs_{\oslash}$ can be defined as the flag completion of this graph. Of course, it is not obvious with this definition that this is a subdivision of $\gs$; Davis and Huang provide an alternative description of $\gs_\oslash$ which obviously produces a subdivision of $\gs$, and show that it is a flag complex with the $1$-skeleton described above.
In either case, these subdivisions $\gs_{\oslash}$ fit together to give a subdivision $L_\oslash$ of $L$. 
The simplicial complex $L_{\oslash}$ is the nerve of the $\RAAG$ $\RA$ described in the introduction. 

In the next subsection, we verify that $\Phi_N: \RA \rightarrow A$ is injective when restricted the free abelian subgroups corresponding to simplices of the nerve $L_{\oslash}$. This can be seen from looking at the abelianization of $PA$. We will also need an explicit description of a natural basis for $H_1(PA,\zz)$ in the locally reducible case.

\subsection{Abelianization of pure spherical Artin groups}

Each finite Coxeter group $W$ acts on $\mathbb R^n$ by linear reflections, where $n$ is the number of elements of $S$. Complexifying this action gives a group action on $\mathbb C^n$ by linear reflections. The complement of the reflecting hyperplanes, denoted by $M(W)$, has $\pi_1(M(W)) = PA$, where $A$ is the associated Artin group.  

The Coxeter group acts freely on $M(W)$, and $\pi_1(M(W)/W) = A$. Deligne showed that $M(W)$ is aspherical \cite{deligne}, so in particular $H_1(PA,\mathbb Z)$ is isomorphic to $H_1(M(W),\mathbb Z)$. It is easy to see that $H_1(M(W),\mathbb Z)$ is isomorphic to the free abelian group $\zz^R$, where $R$ is the set of reflections in $W$, or the set of hyperplanes in this arrangement \cite{ot} (the complement deformation retracts to its intersection with $S^{2n-1}$, which is homeomorphic to $S^{2n-1} - \bigcup_R S^{2n-3}$).  Given a spherical subset $T$ of $S$, let $R_T \subseteq R$ be the set of reflections in $W_T$. 

Let $\{e_r\}$ be the standard basis of $\mathbb{Z}^R$.  
For each element $s \in S$, the element $x_s^2$ in $PA$ corresponds to a loop around the hyperplane corresponding to $s$ in $\mathbb C^n$. It turns out that the class of $x_s^2$ is precisely $e_s$ for each $s \in S$. Any $r \in R$ is a conjugate of an element $s \in S$, i.e.\ $r = wsw^{-1}$ for some $w \in W$. Let $a_w$ be any element of $A$ that projects to $w$ under $p:A \rightarrow W$. Then $a_w x_s^2 a_w^{-1}$ is in $PA$, and its image in $H_1(PA,\zz)$ is precisely $e_r$. 

\begin{lemma}[{\cite[Lem 2.2]{dh16}}]\label{l:homology}
Suppose that $T$ is spherical. Let $e_T$ be the image of $\Delta_T^2$ in $H_1(PA_T,\mathbb Z)$. Then $$e_T = \sum_{t \in R_T} e_t.$$ 
\end{lemma}

It follows from this lemma that given any simplex $\tau$ in $\sigma_{\oslash}$, the image of the elements in $H_1(PA_T,\mathbb Z)$ corresponding to the vertices of $\gt$ are linearly independent, so in particular these form a free abelian subgroup of $PA_T$ (and hence of $PA$) of rank $\dim(\gt)-1$.

Davis and Huang also show that the intersection of these free abelian subgroups is as expected, i.e.\ if $\gs$ and $\gt$ are simplices of $L_{\oslash}$, then the subgroups correspond to $\gs$ and $\gt$ intersect in the subgroup corresponding to $\gs \cap \gt$. This serves as further evidence for the Generalized Tits Conjecture.

\begin{remark}
There are similar configurations of free abelian subgroups for affine hyperplane complements in $\cc^n$, and as far as we know the analogue to the Generalized Tits Conjecture is also open in this case (see \cite[Section 5]{dls}). In this case, the relevant simplicial complex $L$ comes from the intersection poset of the hyperplane arrangement. Each irreducible central subarrangement has an infinite cyclic center, and these combine as above to produce standard free abelian subgroups. Again, it is known that these centers are linearly independent vectors in the first homology group of the arrangement complement. In this case, the relevant partial subdivision of $L$ is the geometric realization of the nested set complex associated to the minimal building set for the arrangement as defined by De Concini and Procesi \cite{dp}. 

\end{remark}

\begin{remark}
The motivation in \cite{dls} behind generalizing the Tits Conjecture was the computation of \emph{action dimension} of $\RAAG$'s in \cite{ados}. This is the minimal dimension of a manifold model of the classifying space $B(\RA)$. This is obviously monotone, in the sense that if $H$ is a subgroup of $G$, then the action dimension of $G$ is greater than the action dimension of $H$. Therefore, if the Generalized Tits Conjecture was true, the action dimension of the Artin group $A$ is larger than the action dimension of $\RA$, and in many cases this would lead to a complete calculation for $A$. On the other hand, combined work of Davis-Huang and Le gave a nearly complete computation for the action dimension of Artin groups (contingent on the $K(\pi,1)$-conjecture) without using the conjecture (it was enough that the standard free abelian subgroups inject into $A$ and intersect as expected) \cite{dh16}, \cite{le}. 
\end{remark}

\subsection{Deligne complex}\label{sec:Deligne complex}

The spherical subsets of $S$ form a poset under inclusion. Let $K$ denote the geometric realization of this poset; $K$ is the cone on the barycentric subdivision $bL$ of the nerve, with the cone vertex corresponding to the empty set $\emptyset$. There is another poset of spherical cosets $aA_T$ of $A$, where $a \in A$ and $A_T$ is a spherical special subgroup, again ordered by inclusion. The geometric realization of this poset is called the \emph{modified Deligne complex} \cite{cd1}, and denoted $D(A)$. The Artin group $A$ acts on $D(A)$ by left multiplication, and $K$ is a strict fundamental domain for this action. The analogous construction in the Coxeter group setting (i.e.\ replace $aA_T$ with $wW_T$) is precisely the Davis complex. 

Here is an alternative description of $D(A)$. 
Given a spherical subset $T$, let $K_{\ge T}$ denote the subcomplex of $K$ spanned by vertices $A_{T'}$ with $T \subseteq T'$. 
Then $$D(A) = A \times K/\sim,$$ where $(a_1,x) \sim (a_2,x)$ if and only if $x$ is in $K_{\ge T}$ and $a_1^{-1}a_2 \in A_T$. We identify $K$ with $1 \times K$. If $A$ is itself a spherical Artin group, the poset of spherical cosets has a maximal element $A$, so the Deligne complex is a cone with cone point $A$. The link of $A$ is a simplicial complex of dimension $|S|-1$, and we denote it by $B(A)$.

Conjecturally, the modified Deligne complex is contractible for all Artin groups. 
It follows from Deligne's work on spherical Artin groups that this would imply the well-known $K(\pi,1)$-conjecture.
Motivated by this, Charney and Davis in \cite{cd1} put two natural piecewise Euclidean metrics on $D(A)$. Note that $D(A)$ has a natural cube complex structure, since the cone on the barycentric subdivision of a simplex is combinatorially isomorphic to a cube. The first metric simply makes each cube isometric to a standard Euclidean cube $[0,1]^n$. It turns out that this metric is CAT(0) if and only if the nerve $L$ is a flag complex (the induced metric on the link of the vertex $\emptyset$ is isometric to $L$ with the all-right spherical metric, so flagness is an obvious necessary condition). 

The second metric on $D(A)$  is called the \emph{Moussong metric}, as it is related to the Moussong metric on the Davis complex, which is always CAT(0). Charney and Davis showed that the Moussong metric on $D(A)$ is CAT(0) for $2$-dimensional Artin groups.
It is still open whether $D(A)$ equipped with the Moussong metric is CAT(0) for all Artin groups.

We now describe this metric for Artin groups corresponding to dihedral groups, which is the only important case for this paper. It is slightly more convenient for us to keep the original cellulation of $D(A)$ as a simplicial complex.
See \cite{charney} or \cite{cd1} for details in the general case. 
The dihedral group $D_{2m}$ acts by linear reflections on $\mathbb R^2$ in the standard way, with strict fundamental domain a simplicial cone. 
Let $x$ be the unique point in this cone whose distance from the two walls of the cone is $1$.  
The convex hull of the orbit of $x$ under $D_{2m}$ is a $2m$-gon with edge lengths $2$. We subdivide the $2m$-gon by coning off the orbit of $x$. The intersection of this convex hull with the simplicial cone is combinatorially isomorphic to $K$, and we give $K$ the inherited Euclidean metric. 
This defines a piecewise Euclidean metric on $D(A)$, which in this case is CAT(0). 
This naturally gives a piecewise spherical metric on $B(A)$, where every edge has length $\pi/2m$. 
This turns out to be a CAT(1) metric, which in this case is equivalent to there being no closed geodesics with length $< 2\pi$. In the $2$-dimensional case, this follows from \cite[Lem 6]{as}.

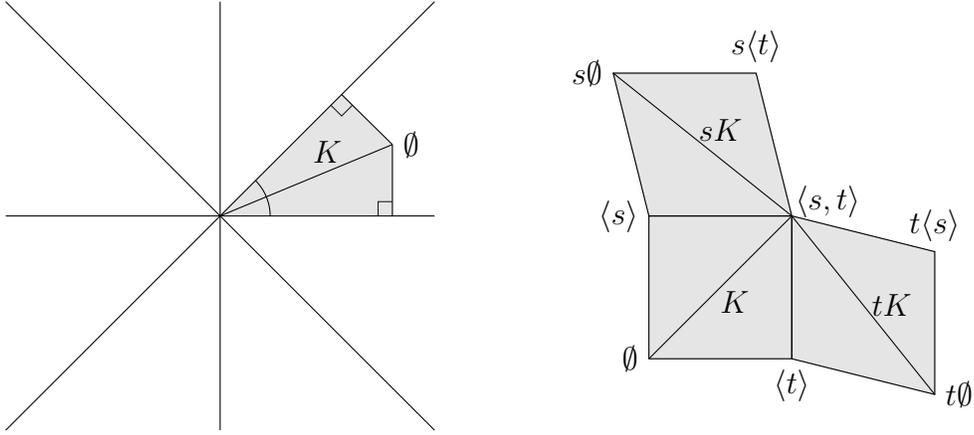
\begin{figure}
\centering
\begin{tikzpicture}[scale = .95]
\draw (-3,-3) -- (3,3);
\draw (3,0) -- (-3,0);
\draw (0,3) -- (0,-3);
\draw (-3,3) -- (3,-3);
\draw (2.41,0) -- (2.41,1);
\draw (2.41,1) -- (1.7,1.7);
\draw (2.41,.2) -- (2.21,.2) -- (2.21,0);
\draw (1.55,1.55) -- (1.7,1.4) -- (1.85,1.55);

\draw[fill = gray, opacity = .2] (2.41,0) -- (2.41,1) -- (1.7, 1.7) -- (0,0) -- (2.41,0);
\node[above, right] at (2.41, 1) {$\emptyset$};
\draw (2.41,1) -- (0,0);
\node at (1.5, .9) {$K$};
\draw (.7,0) arc (0:45:.7);

\begin{scope}[shift = {(6,-2)}];

\draw[fill = gray, opacity = .2] (0,2) -- (2,2) -- (2, 0) -- (0,0) -- (0,2);
\draw[fill = gray, opacity = .2] (2,2) -- (1.5,4) -- (-.5, 4) -- (0,2) -- (2,2);
\draw[fill = gray, opacity = .2] (2,0) -- (4,-.5) -- (4, 1.5) -- (2,2) -- (2,0);

\draw(0,2) -- (2,2) -- (2, 0) -- (0,0) -- (0,2);
\draw (2,2) -- (1.5,4) -- (-.5, 4) -- (0,2) -- (2,2);
\draw (2,0) -- (4,-.5) -- (4, 1.5) -- (2,2) -- (2,0);

\draw (0,0) -- (2,2);
\draw (4,-.5) -- (2,2);
\draw (-.5,4) -- (2,2);

\node at (1.2, .8) {$K$};
\node at (1, 3.2) {$sK$};
\node at (3.4, .75) {$tK$};

\node[below] at (2, 0) {$\langle t \rangle$};
\node[left] at (0, 2) {$\langle s \rangle$};
\node [above, right] at (1.9, 2.2) {$\langle s,t \rangle$ };
\node [below, left] at (0,0) {$\emptyset$};

\node[right] at (4, -.5) {$t\emptyset$};
\node[left] at (-.5, 4) {$s\emptyset$};
\node[above] at (1.5, 4) {$s\langle t \rangle$};
\node[above] at (4, 1.5) {$t\langle s \rangle$};

\end{scope}

\end{tikzpicture}

\caption{The Moussong metric on $K$ for the dihedral Artin group and a small part of the development of the Deligne complex.}
\end{figure}

We now recall the structure of links of vertices in the Deligne complex, again see \cite{cd1} or \cite{charney} for complete details. 
It suffices to consider vertices in $K$ as this is a strict fundamental domain for the action. 
Let $T$ be a spherical subset of $S$, and let $v_T$ be the corresponding vertex in $K$. 
Then we have the following join decomposition $$\Lk_{D(A)}(v_T) = \Lk_{K_{\ge T}}(v_T) \ast B(A_T).$$

The piecewise spherical metric on $\Lk_{D(A)}(v_T)$ is isometric to the spherical join of the metrics on $\Lk_{K_{\ge T}}(v_T)$ and  $B(A_T)$ \cite[Lem 2.2]{charney2}. Recall the spherical join of two piecewise spherical metrics is defined by making each simplex $\gs \ast \gt$ isometric to the simplex in $\mathbb S^{\dim \gs + \dim \gt + 1}$ spanned by $\gs \subseteq S^{\dim \gs}$ and $\gt \subseteq S^{\dim \gt}$, where points in $\gs$ and $\gt$ are all distance $\pi/2$ apart. Charney proves this for the cubical metric on $D(A)$, our links are isometric but have a finer subdivision (which still preserves the join structure).

\subsection{Representations of small type Artin groups inside mapping class groups}\label{sec:reps}

Let $\Sigma$ be an oriented compact surface, possibly with boundary. Let $P = \{P_1,\dots P_n\}$ be a collection of $n$ punctures in the interior of $\Sigma$.
Let $\Homeo^+(\Sigma,P)$ denote the group of orientation-preserving homeomorphisms of $\Sigma$ which fix the boundary pointwise, and which preserve $P$. 
Let $\Homeo^+_0(\Sigma,P)$ denote the connected component of the identity
in $\Homeo^+(\Sigma,P)$. The mapping class group of the pair $(\Sigma,P)$ is defined to be
$$\Mod(\Sigma,P) = \Homeo^+(\Sigma,P)/\Homeo^+_0(\Sigma,P)$$

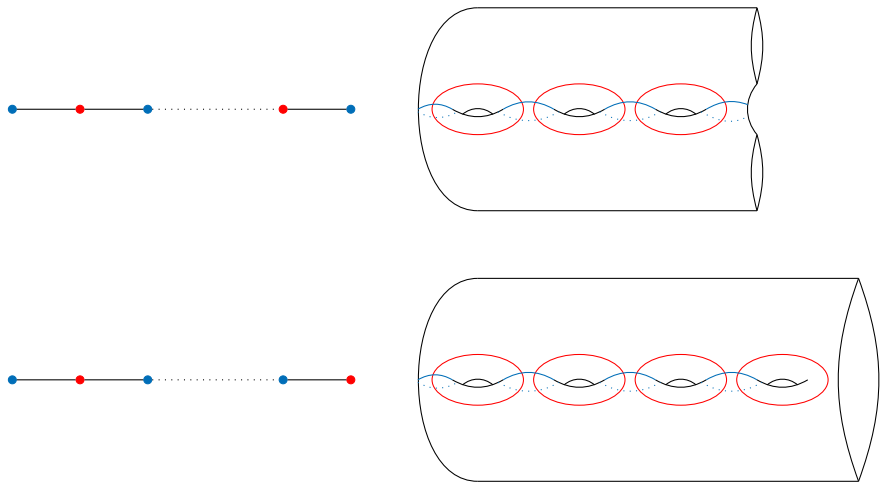
\begin{figure}
\centering
\begin{tikzpicture}[scale = .9]
\begin{scope}[shift={(0,2)}]
\node[circle, draw, fill, inner sep = 0pt,minimum width = 3pt, color=NavyBlue] (a) at (0,0) {};
\node[circle, draw, fill, inner sep = 0pt,minimum width = 3pt, color=red] (b) at (1,0) {};
\node[circle, draw, fill, inner sep = 0pt,minimum width = 3pt, color=NavyBlue] (c) at (2,0) {};
\node[circle, draw, fill, inner sep = 0pt,minimum width = 3pt, color=red] (d) at (4,0) {};
\node[circle, draw, fill, inner sep = 0pt,minimum width = 3pt, color=NavyBlue] (e) at (5,0) {};
\draw (a) -- (b) -- (c);
\draw[dotted] (c) -- (d) node[draw=none,fill=none,midway,above] {};
\draw (d) -- (e);
\end{scope}
\begin{scope}[shift={(0,-2)}]
\node[circle, draw, fill, inner sep = 0pt,minimum width = 3pt, color=NavyBlue] (a) at (0,0) {};
\node[circle, draw, fill, inner sep = 0pt,minimum width = 3pt, color=red] (b) at (1,0) {};
\node[circle, draw, fill, inner sep = 0pt,minimum width = 3pt, color=NavyBlue] (c) at (2,0) {};
\node[circle, draw, fill, inner sep = 0pt,minimum width = 3pt, color=NavyBlue] (d) at (4,0) {};
\node[circle, draw, fill, inner sep = 0pt,minimum width = 3pt, color=red] (e) at (5,0) {};
\draw (a) -- (b) -- (c);
\draw[dotted] (c) -- (d) node[draw=none,fill=none,midway,above] {};
\draw (d) -- (e);
\end{scope}

\begin{scope}[shift = {(5,-2)}, scale = .75];
\draw (2.5,-2) to[out=180, in=180] (2.5,2);
\draw (2.5,-2) to (10, -2);
\draw (2.5,2) to (10, 2);
\draw (10,-2) to[out = 110, in = 250] (10, 2);
\draw (10,-2) to[out = 70, in = -70] (10, 2);
\draw (2, -0) to[out=-30, in=210] (3, 0);
\draw (2.2,-0.1) to[out=40, in=140] (2.8, -0.1);

\draw (4, -0) to[out=-30, in=210] (5, 0);
\draw (4.2,-0.1) to[out=40, in=140] (4.8, -0.1);

\draw (6, -0) to[out=-30, in=210] (7, 0);
\draw (6.2,-0.1) to[out=40, in=140] (6.8, -0.1);

\draw (8, -0) to[out=-30, in=210] (9, 0);
\draw (8.2,-0.1) to[out=40, in=140] (8.8, -0.1);

\draw[red] (2.5, 0) ellipse (0.9 and 0.5);

\draw[red] (4.5, 0) ellipse (0.9 and 0.5);

\draw[red] (6.5, 0) ellipse (0.9 and 0.5);

\draw[red] (8.5, 0) ellipse (0.9 and 0.5);

\draw[NavyBlue] (3,0) to[out=30, in=150] (4,0);
\draw[NavyBlue, dotted] (2.9,-0.05) to[out=330, in=210] (4.1,-0.05);

\draw[NavyBlue] (5,0) to[out=30, in=150] (6,0);
\draw[NavyBlue, dotted] (4.9,-0.05) to[out=330, in=210] (6.1,-0.05);

\draw[NavyBlue] (7,0) to[out=30, in=150] (8,0);
\draw[NavyBlue, dotted] (6.9,-0.05) to[out=330, in=210] (8.1,-0.05);

\draw[NavyBlue] (1.32,0) to[out=30, in=150] (2,0);
\draw[NavyBlue, dotted] (1.32,-0.05) to[out=330, in=210] (2.1,-0.05);

\end{scope}

\begin{scope}[shift = {(5,2)}, scale = .75];
\draw (2.5,-2) to[out=180, in=180] (2.5,2);
\draw (2.5,-2) to (8, -2);
\draw (2.5,2) to (8, 2);
\draw[] (8, -2) to[out =105, in=255] (8, -0.5);
\draw (8, -2) to[out=75, in=285] (8, -0.5);
\draw[] (8, 2) to[out =255, in=105] (8, 0.5);
\draw(8, 2) to[out=285, in=75] (8, 0.5);
\draw (8,0.5) to[out=230, in=130] (8,-0.5);

\draw (2, -0) to[out=-30, in=210] (3, 0);
\draw (2.2,-0.1) to[out=40, in=140] (2.8, -0.1);

\draw (4, -0) to[out=-30, in=210] (5, 0);
\draw (4.2,-0.1) to[out=40, in=140] (4.8, -0.1);

\draw (6, -0) to[out=-30, in=210] (7, 0);
\draw (6.2,-0.1) to[out=40, in=140] (6.8, -0.1);

\draw[red] (2.5, 0) ellipse (0.9 and 0.5);

\draw[red] (4.5, 0) ellipse (0.9 and 0.5);

\draw[red] (6.5, 0) ellipse (0.9 and 0.5);

\draw[NavyBlue] (3,0) to[out=30, in=150] (4,0);
\draw[NavyBlue, dotted] (2.9,-0.05) to[out=330, in=210] (4.1,-0.05);

\draw[NavyBlue] (5,0) to[out=30, in=150] (6,0);
\draw[NavyBlue, dotted] (4.9,-0.05) to[out=330, in=210] (6.1,-0.05);

\draw[NavyBlue] (7,0) to[out=30, in=160] (7.82,0.09);
\draw[NavyBlue, dotted] (6.9,-0.05) to[out=330, in=210] (7.82,-0.14);

\draw[NavyBlue] (1.32,0) to[out=30, in=150] (2,0);
\draw[NavyBlue, dotted] (1.32,-0.05) to[out=330, in=210] (2.1,-0.05);
\end{scope}
\end{tikzpicture}

\caption{The Perron-Vannier representation for the Artin groups of type $A_n$. If $n$ is even, then the element $\Delta_S^4$ goes to a Dehn twists around the boundary curve. If $n$ is odd, then the element $\Delta_S^2$ goes to multitwist that is a product of single Dehn twists about each of the boundary curves.}
\label{f:an}
\end{figure}

A \emph{multicurve} is a disjoint union of a finite number of simple, closed, essential curves in $\Sigma$. A \emph{multitwist} about a multicurve is the composition of (not necessarily the same) powers of Dehn twists about the individual curves. Since the curves are disjoint, the order in which we compose those Dehn twists does not matter. 

For the small-type spherical Artin groups, there are classical representations into mapping class groups. This is due to Birman-Hilden for type $A_n$, and Perron-Vannier in general \cite{pv}. We will refer to all of them as \emph{Perron-Vannier representations}. 
Crisp and Paris defined similar representations for all small-type Artin groups, however we shall not need this generality. 
These representations also naturally arise as geometric monodromies of simple singularities of type $\gG$ \cite{matsumoto}.

For $A$ of type $A_n$, the Perron-Vanier representation $A\to \Mod(\Sigma)$ sends the consecutive generators of $A$ to the Dehn twists around the consecutive curves in Figure~\ref{f:an}. The surface $\Sigma$ has genus $\frac{n-1}{2}$ and two boundary components, when $n$ is odd, and $\Sigma$ has genus $\frac{n}{2}$ and one boundary component, when $n$ is even. 

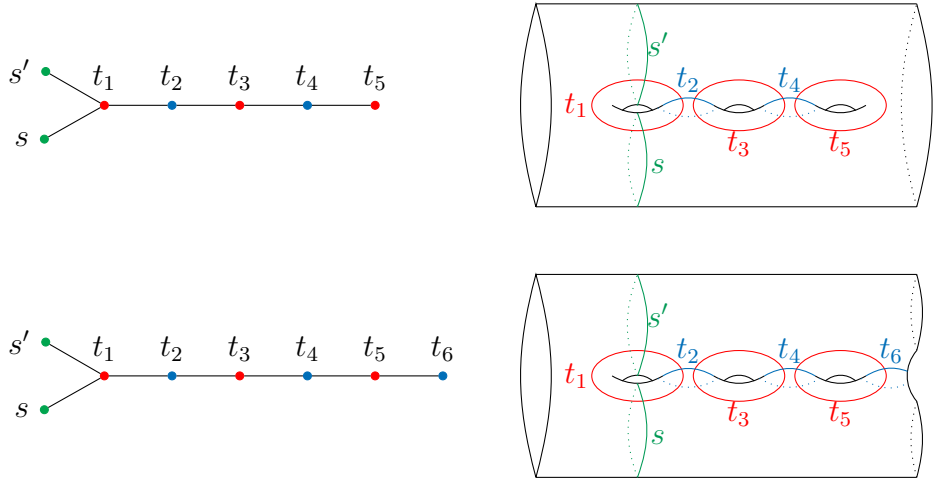
\begin{figure}
\centering

\begin{tikzpicture}[scale = .9]
\begin{scope}
\node[circle, draw, fill, inner sep = 0pt,minimum width = 3pt, label=above:$t_1$, color=red] (a) at (0,0) {};
\node[circle, draw, fill, inner sep = 0pt,minimum width = 3pt, label=above:$t_2$, color=NavyBlue] (b) at (1,0) {};
\node[circle, draw, fill, inner sep = 0pt,minimum width = 3pt, label=above:$t_3$, color=red] (g) at (2,0) {};
\node[circle, draw, fill, inner sep = 0pt,minimum width = 3pt, label=above:$t_4$, color=NavyBlue] (c) at (3,0) {};
\node[circle, draw, fill, inner sep = 0pt,minimum width = 3pt, label=above:$t_5$, color=red] (d) at (4,0) {};
\node[circle, draw, fill, inner sep = 0pt,minimum width = 3pt, label=left:$s'$, color=Green] (e) at (-0.866,0.5) {};
\node[circle, draw, fill, inner sep = 0pt,minimum width = 3pt, label=left:$s$, color=Green] (f) at (-0.886,-0.5) {};
 \draw (e) -- (a) -- (b) -- (g) -- (c) --(d);
\draw (f) -- (a);
\end{scope}

\begin{scope}[shift={(6,0)}, scale = 0.75]
\draw (0.5,-2) to[out=105, in=255] (0.5,2) to[out=285, in=75] (0.5,-2);
\draw (0.5,-2) to (8, -2);
\draw (0.5,2) to (8, 2);
\draw[dotted] (8,-2) to[out=105, in=255] (8,2);
\draw (8,2) to[out=285, in=75] (8,-2);

\draw (2, -0) to[out=-30, in=210] (3, 0);
\draw (2.2,-0.1) to[out=40, in=140] (2.8, -0.1);

\draw (4, -0) to[out=-30, in=210] (5, 0);
\draw (4.2,-0.1) to[out=40, in=140] (4.8, -0.1);

\draw (6, -0) to[out=-30, in=210] (7, 0);
\draw (6.2,-0.1) to[out=40, in=140] (6.8, -0.1);

\draw[red] (2.5, 0) ellipse (0.9 and 0.5);
\node[red, draw=none,fill=none] at (1.3,0) {$t_1$};

\draw[red] (4.5, 0) ellipse (0.9 and 0.5);
\node[red, draw=none,fill=none] at (4.5,-0.75) {$t_3$};

\draw[red] (6.5, 0) ellipse (0.9 and 0.5);
\node[red, draw=none,fill=none] at (6.5,-0.75) {$t_5$};

\draw[NavyBlue] (3,0) to[out=30, in=150] (4,0);
\draw[NavyBlue, dotted] (2.9,-0.05) to[out=330, in=210] (4.1,-0.05);
\node[NavyBlue, draw=none,fill=none] at (3.5,0.5) {$t_2$};

\draw[NavyBlue] (5,0) to[out=30, in=150] (6,0);
\draw[NavyBlue, dotted] (4.9,-0.05) to[out=330, in=210] (6.1,-0.05);
\node[NavyBlue, draw=none,fill=none] at (5.5,0.5) {$t_4$};

\draw[ForestGreen] (2.5, 0) to[out=70 , in=290] (2.5, 2);
\draw[ForestGreen, dotted] (2.5, 0) to[out=110 , in=250] (2.5, 2);
\node[ForestGreen, draw=none,fill=none] at (2.9,1.2) {$s'$};

\draw[ForestGreen] (2.5, -0.15) to[out=290 , in=70] (2.5, -2);
\draw[ForestGreen, dotted] (2.5, -0.15) to[out=250 , in=110] (2.5, -2);
\node[ForestGreen, draw=none,fill=none] at (2.9,-1.2) {$s$};
\end{scope}

\begin{scope}[shift={(0,-4)}]
\node[circle, draw, fill, inner sep = 0pt,minimum width = 3pt, label=above:$t_1$, color=red] (a) at (0,0) {};
\node[circle, draw, fill, inner sep = 0pt,minimum width = 3pt, label=above:$t_2$, color=NavyBlue] (b) at (1,0) {};
\node[circle, draw, fill, inner sep = 0pt,minimum width = 3pt, label=above:$t_3$, color=red] (g) at (2,0) {};
\node[circle, draw, fill, inner sep = 0pt,minimum width = 3pt, label=above:$t_4$, color=NavyBlue] (c) at (3,0) {};
\node[circle, draw, fill, inner sep = 0pt,minimum width = 3pt, label=above:$t_5$, color=red] (d) at (4,0) {};
\node[circle, draw, fill, inner sep = 0pt,minimum width = 3pt, label=above:$t_6$, color=NavyBlue] (h) at (5,0) {};
\node[circle, draw, fill, inner sep = 0pt,minimum width = 3pt, label=left:$s'$, color=Green] (e) at (-0.866,0.5) {};
\node[circle, draw, fill, inner sep = 0pt,minimum width = 3pt, label=left:$s$, color=Green] (f) at (-0.886,-0.5) {};
\draw (e) -- (a) -- (b) -- (g) -- (c) -- (d) -- (h);
\draw (f) -- (a);
\end{scope}
\begin{scope}[shift={(6,-4)}, scale = 0.75]
\draw (0.5,-2) to[out=105, in=255] (0.5,2) to[out=285, in=75] (0.5,-2);
\draw (0.5,-2) to (8, -2);
\draw (0.5,2) to (8, 2);
\draw[dotted] (8, -2) to[out =105, in=255] (8, -0.5);
\draw (8, -2) to[out=75, in=285] (8, -0.5);
\draw[dotted] (8, 2) to[out =255, in=105] (8, 0.5);
\draw(8, 2) to[out=285, in=75] (8, 0.5);
\draw (8,0.5) to[out=230, in=130] (8,-0.5);

\draw (2, -0) to[out=-30, in=210] (3, 0);
\draw (2.2,-0.1) to[out=40, in=140] (2.8, -0.1);

\draw (4, -0) to[out=-30, in=210] (5, 0);
\draw (4.2,-0.1) to[out=40, in=140] (4.8, -0.1);

\draw (6, -0) to[out=-30, in=210] (7, 0);
\draw (6.2,-0.1) to[out=40, in=140] (6.8, -0.1);

\draw[red] (2.5, 0) ellipse (0.9 and 0.5);
\node[red, draw=none,fill=none] at (1.3,0) {$t_1$};

\draw[red] (4.5, 0) ellipse (0.9 and 0.5);
\node[red, draw=none,fill=none] at (4.5,-0.75) {$t_3$};

\draw[red] (6.5, 0) ellipse (0.9 and 0.5);
\node[red, draw=none,fill=none] at (6.5,-0.75) {$t_5$};

\draw[NavyBlue] (3,0) to[out=30, in=150] (4,0);
\draw[NavyBlue, dotted] (2.9,-0.05) to[out=330, in=210] (4.1,-0.05);
\node[NavyBlue, draw=none,fill=none] at (3.5,0.5) {$t_2$};

\draw[NavyBlue] (5,0) to[out=30, in=150] (6,0);
\draw[NavyBlue, dotted] (4.9,-0.05) to[out=330, in=210] (6.1,-0.05);
\node[NavyBlue, draw=none,fill=none] at (5.5,0.5) {$t_4$};

\draw[ForestGreen] (2.5, 0) to[out=70 , in=290] (2.5, 2);
\draw[ForestGreen, dotted] (2.5, 0) to[out=110 , in=250] (2.5, 2);
\node[ForestGreen, draw=none,fill=none] at (2.9,1.2) {$s'$};

\draw[ForestGreen] (2.5, -0.15) to[out=290 , in=70] (2.5, -2);
\draw[ForestGreen, dotted] (2.5, -0.15) to[out=250 , in=110] (2.5, -2);
\node[ForestGreen, draw=none,fill=none] at (2.9,-1.2) {$s$};

\draw[NavyBlue] (7,0) to[out=30, in=160] (7.82,0.09);
\draw[NavyBlue, dotted] (6.9,-0.05) to[out=330, in=210] (7.82,-0.14);
\node[NavyBlue, draw=none,fill=none] at (7.5,0.5) {$t_6$};

\end{scope}
\end{tikzpicture}
\caption{The Perron-Vannier representation $A\to Mod(\Sigma)$ for Artin group $A$ of type $D_n$. Let $\gamma_1$ be the left connected component of $\partial \Sigma$. For $n$ odd, the element $\Delta^2_S$ gets sent to a product of Dehn twists $\gamma_1^{n-2}\gamma_2$. 
For $n$ even, the element $\Delta^2_S$ gets sent to a product of Dehn twists $\gamma_1^{\frac{n}{2}-1} \gamma_2 \gamma_3$.}
\label{fig:Dn}
\end{figure}

If $A$ has type $D_n$, let the standard generators of $A$ be $\{s, s',t_1, \dots, t_{n-2}\}$ where $s$ and $s'$ are both adjacent to $t_1$ in the Coxeter graph, and $t_i$ and $t_{i+1}$ are adjacent for all $i=1,\dots, n-3$. The Perron-Vannier representation $A \to\Mod(\Sigma)$ sends the generators to the Dehn twists around curves, as pictured in Figure~\ref{fig:Dn}. The surface $\Sigma$ has genus $\frac{n-1}2$ and two boundary components when $n$ is odd, and genus $\frac{n-2}2$ and three boundary components when $n$ is even.

The Perron-Vannier representation $A\to \Mod(\Sigma)$ of the Artin group $A$ of type $E_n$, where $n=6,7,8$, is illustrated in Figure~\ref{fig:E6}. The surface $\Sigma$ has 
\begin{itemize}
\item genus three and one boundary component, if $n=6$.
\item genus three and two boundary components, if $n=7$.
\item genus four and one boundary component, if $n=8$.
\end{itemize}

\begin{remark}\label{rem:deltas boundaries}
For any small type spherical Artin group $A$ with the Perron-Vannier representation $\rho:A\to \Mod(\Sigma)$, the image of the element $\Delta^4_S$ under $\rho$ is a multitwist about the boundary components of $\partial \Sigma$. It will never matter for us the exact power of each Dehn twist. 
\end{remark}

\begin{figure}
\centering
\begin{tikzpicture}[scale = .9]
\begin{scope}
\node[circle, draw, fill, inner sep = 0pt,minimum width = 3pt, color=red, label=below:$t_1$] (a) at (0,0) {};
\node[circle, draw, fill, inner sep = 0pt,minimum width = 3pt, color=NavyBlue, label=below:$t_2$] (b) at (1,0) {};
\node[circle, draw, fill, inner sep = 0pt,minimum width = 3pt, color=red, label=below:$t_3$] (c) at (2,0) {};
\node[circle, draw, fill, inner sep = 0pt,minimum width = 3pt, color=NavyBlue, label=below:$t_4$] (d) at (3,0) {};
\node[circle, draw, fill, inner sep = 0pt,minimum width = 3pt, color=red, label=below:$t_5$] (e) at (4,0) {};
\node[circle, draw, fill, inner sep = 0pt,minimum width = 3pt, color=NavyBlue, label=below:$t_6$] (g) at (5,0) {};
\node[circle, draw, fill, inner sep = 0pt,minimum width = 3pt, color=red, label=below:$t_7$] (h) at (6,0) {};
\node[circle, draw, fill, inner sep = 0pt,minimum width = 3pt, color=Green, label=right:$s$] (f) at (2,1) {};
\draw (a) -- (b) -- (c) -- (f);
\draw[dotted] (e) -- (g) -- (h);
\draw (c) -- (d); 
\draw (d) -- (e);
\end{scope}

\begin{scope}[shift={(6,0)}, scale = 0.75]
\draw (2.5,-2) to[out=180, in=180] (2.5,2);
\draw (2.5,-2) to (10, -2);
\draw (2.5,2) to (10, 2);
\draw[dotted] (10,-2) to[out=105, in=255]  (10, 2);
\draw (10,2) to[out=285, in=75] (10, -2);

\draw[Gray] (8.5, 0) to[out=70 , in=290] (8.5, 2);
\draw[Gray, dotted] (8.5, 0) to[out=110 , in=250] (8.5, 2);
\draw[Gray] (8.5, -0.15) to[out=-70 , in=-290] (8.5, -2);
\draw[Gray, dotted] (8.5, -0.15) to[out=-110 , in=-250] (8.5, -2);

\draw[Gray,dotted] (7.5, -2) to[out=105, in=255] (7.5, 2);
\draw[Gray] (7.5,2) to[out=285, in=75] (7.5, -2);

\draw (2, -0) to[out=-30, in=210] (3, 0);
\draw (2.2,-0.1) to[out=40, in=140] (2.8, -0.1);

\draw (4, -0) to[out=-30, in=210] (5, 0);
\draw (4.2,-0.1) to[out=40, in=140] (4.8, -0.1);

\draw (6, -0) to[out=-30, in=210] (7, 0);
\draw (6.2,-0.1) to[out=40, in=140] (6.8, -0.1);

\draw (8, -0) to[out=-30, in=210] (9, 0);
\draw (8.2,-0.1) to[out=40, in=140] (8.8, -0.1);

\draw[red] (2.5, 0) ellipse (0.9 and 0.5);
\node[red, draw=none,fill=none] at (2.5,-0.75) {$t_1$};

\draw[red] (4.5, 0) ellipse (0.9 and 0.5);
\node[red, draw=none,fill=none] at (4.5,-0.75) {$t_3$};

\draw[red] (6.5, 0) ellipse (0.9 and 0.5);
\node[red, draw=none,fill=none] at (6.5,-0.75) {$t_5$};

\draw[red] (8.5, 0) ellipse (0.9 and 0.5);
\node[red, draw=none,fill=none] at (8.5,-0.75) {$t_7$};

\draw[NavyBlue] (3,0) to[out=30, in=150] (4,0);
\draw[NavyBlue, dotted] (2.9,-0.05) to[out=330, in=210] (4.1,-0.05);
\node[NavyBlue, draw=none,fill=none] at (3.5,0.5) {$t_2$};

\draw[NavyBlue] (5,0) to[out=30, in=150] (6,0);
\draw[NavyBlue, dotted] (4.9,-0.05) to[out=330, in=210] (6.1,-0.05);
\node[NavyBlue, draw=none,fill=none] at (5.5,0.5) {$t_4$};

\draw[NavyBlue] (7,0) to[out=30, in=150] (8,0);
\draw[NavyBlue, dotted] (6.9,-0.05) to[out=330, in=210] (8.1,-0.05);
\node[NavyBlue, draw=none,fill=none] at (7.5,0.5) {$t_6$};

\draw[Green] (4.5, 0) to[out=70 , in=290] (4.5, 2);
\draw[Green, dotted] (4.5, 0) to[out=110 , in=250] (4.5, 2);
\node[Green, draw=none,fill=none] at (4.9,1.2) {$s$};

\end{scope}
\end{tikzpicture}
\caption{The Coxeter diagram for Artin group $A$ of type $E_6, E_7, E_8$. A representation $A\to \Mod(\Sigma)$. The subsurfaces with gray boundary curves correspond to $E_6$ and $E_7$.}
\label{fig:E6}

\end{figure}
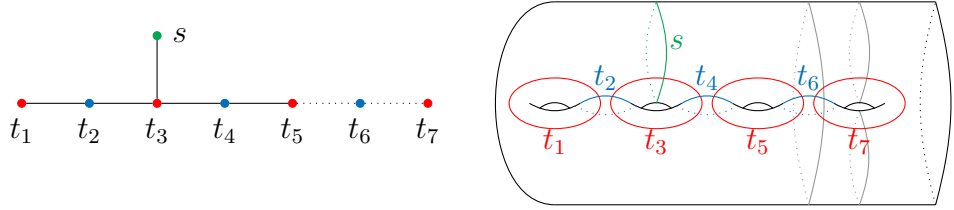

\section{Generalized Tits Conjecture for locally reducible Artin groups}\label{s:lr}

In this section, we show that Artin groups with no edges labeled by $3$ satisfy the Generalized Tits Conjecture for $N = 1$, and locally reducible Artin groups satisfy the conjecture for $N = 2$. 
An Artin group is \emph{totally reducible} if its Coxeter diagram is a disjoint union of vertices and single edges, i.e.\ it is a direct product of Artin groups of rank at most $2$.
Recall from the introduction that an Artin group is \emph{locally reducible} if all spherical special subgroups are totally reducible. 
We first record  an easy characterization of the locally reducible Artin groups in terms of their Coxeter diagram.

\begin{lemma}[{\cite[Lem 3.1]{charney}}]
Let $\gG$ be a Coxeter diagram, and $A$ the associated Artin group. Then $A$ is locally reducible if and only if $\gG$ satisfies the following condition:

If two consecutive edges of $\gG$ are not contained in a triangle, then their labels $a,b$ satisfy $1/a + 1/b \le 1/2$.
\end{lemma}

\subsection{CAT(0) geometry}
We will assume that the reader is comfortable with comparison geometry, particularly in the setting of piecewise Euclidean and piecewise spherical cell complexes, see \cite{charney}, \cite{bh} or \cite[Appendix I]{dbook} for the relevant details. We record some theorems and definitions that we will need. The first is due to Gromov, proofs can be found in \cite{bh}. 

\begin{Theorem}
Let $X$ be a piecewise Euclidean cell complex. Then $X$ is locally CAT(0) if and only if the induced piecewise spherical metric on the link $\Lk(v,X)$ is CAT(1) for all vertices $v \in X$. If $X$ is simply connected and locally CAT(0), then it is CAT(0). 
\end{Theorem}

\begin{lemma}[see Appendix of \cite{cd1}]
The spherical join of two piecewise spherical complexes $L_1$ and $L_2$ is CAT(1) if and only if $L_1$ and $L_2$ are CAT(1). 
\end{lemma}

\begin{definition}
Let $f: L \rightarrow L'$ be a map between piecewise spherical complexes. We say $f$ is \emph{$\pi$-distance preserving} if $$d_L(x_1, x_2) \ge \pi \implies d_{L'}(f(x_1), f(x_2)) \ge \pi.$$
\end{definition}

If $L$ is not connected, then we set $d_L = \infty$ for points in different components. 
Thus $f: L \rightarrow L'$ is $\pi$-distance preserving if it is $\pi$-distance preserving on each component of $L$ and points in different components get mapped $\ge \pi$ apart in $L'$.  

\begin{lemma}[{\cite[Lem 1.4]{charney}}]\label{l:isomembedding}
Suppose $f: X \rightarrow X'$ is a map between two piecewise Euclidean complexes which takes piecewise geodesics to piecewise geodesics. 
Then $f$ is locally an isometric embedding if and only if the induced maps on all links $\Lk(x,X) \rightarrow \Lk(f(x), X')$ are $\pi$-distance preserving. 
Furthermore, if $X'$ is CAT(0), then $f$ is an isometric embedding. 
\end{lemma}

\begin{lemma}[see Appendix of \cite{cd1}]\label{l:ojoins}
If $f: L_1 \rightarrow L_1'$ and $g: L_2 \rightarrow L_2'$ are $\pi$-distance preserving, and $L_1 \ast L_2$, $L_1' \ast L_2'$ are spherical joins, then $$f \ast g: L_1 \ast L_2 \rightarrow L_1' \ast L_2'$$ is $\pi$-distance preserving. 
\end{lemma}

Charney showed that the proof that $D(A)$ is CAT(0) for $2$-dimensional Artin groups extends to locally reducible Artin groups.

\begin{Theorem}{\cite[Thm 3.2]{charney}}
Let $A$ be a locally reducible Artin group. Then $D(A)$ equipped with the Moussong metric is CAT(0). 
\end{Theorem}

\subsection{Orthogonality in the Deligne complex}

\begin{definition}
Let $F_n = \langle x_1, \dots, x_n \rangle$ be a free group of rank $n$. 
Then every element $g$ of $F_n$ can be written uniquely as $$g = x_{i_1}^{n_1}x_{i_2}^{n_2} \dots x_{i_k}^{i_k}$$ where $x_{i_j} \ne x_{i_{j+1}}$ and $n_i \in \mathbb Z - \{0\}$. 
The \emph{syllable length} of $g$ in this case is equal $k$. 
If $G$ is a group admitting a surjection $\phi: F_n \rightarrow G$ is a surjection, then for $g \in G$ we define the \emph{syllable length} of $g$ with respect to the generating set $\{\phi(x_1), \dots, \phi(x_n)\}$, to be the infimum of syllable lengths of words in $\phi^{-1}(g)$. 

\end{definition}
For convenience, we will denote the generators of $A$ by letters $s,t,\dots$ rather than $x_s,x_t \dots$.
We will also denote the coset $a\langle\emptyset\rangle$ just by $a$ for each $a \in A$. 
If $A$ is a dihedral Artin group, then a simplicial path in $B(A)$ between two cosets $g$ and $h$ must contain an even number of vertices. 
In particular, each pair of consecutive edges connects $g$ to $gt^k$ or $gs^k$ for some $k \in \zz$. Therefore, we can associate to a path in $B(A)$ between $g$ and $h$ a word $w$ in the free group $F(s,t)$ so that $gw = h$ in $A$. If this path is embedded, then its length is equal to $\frac{\pi}{m}k$ where $k$ is the syllable length of $w$.

We now record a technical proposition concerning dihedral Artin groups.
In this case, the pure Artin group is isomorphic to the direct product $F_{m-1} \times \mathbb{Z}$, where $\Delta^2$ is the generator of $\mathbb Z$ (in this case, the hyperplane complement is homeomorphic to a $\mathbb C^\ast$-bundle over an $m$-punctured sphere). 
The proposition states a sort of orthogonality for the action of the pure Artin group on the Deligne complex.
 Given two elements $g,h$ in the Artin group, we let $d_{B(A)}(g, h)$ denote the distance in $B(A)$ between the vertices $g$ and $h$.

\begin{Prop}\label{lem:dihedral}
Let $A$ be a dihedral Artin group. Let $g$ be an element of the free subgroup $\langle s^2,t^2\rangle$. Then for any $n \in \zz - \{0\}$ we have that $$d_{B(A)}(\Delta^{2n}, g) \ge \pi.$$  Furthermore, if $m = m_{st} > 3$, then $$d_{B(A)}(\Delta^{2n}, g) \ge \pi + \pi/m.$$ Finally, if $m = m_{st} = 3$ and $g$ is in the free subgroup $\langle s^4, t^4 \rangle$, then  $$d_{B(A)}(\Delta^{2n}, g) \ge \pi + \pi/m.$$
\end{Prop}

\begin{proof}

Suppose that $d_{B(A)}(\Delta^{2n}, g) < \pi$, so there is $a \in A$ of syllable length $<m$ so that $$\Delta^{2n} = ag.$$ 
Note that $a$ is obviously an element of the pure Artin group $PA$, since both $g$ and $\Delta^{2n}$ are. We now compute the image $\overline a$ of $a$ in $H_1(PA, \mathbb Z)$. 
Let $a = bc$, where $b$ has only odd powers of $s$ and $t$, and $c$ begins with $s^{2r}$ or $t^{2r}$. Let $\pi: A \rightarrow W$ be the canonical projection. Since the length of $\pi(b)$ in $W$ is $< 2m$, $c$ must be nontrivial, otherwise $a$ would not be in $PA$.
Without loss of generality, suppose that $c$ starts with $s^{2r}$. Note that $c$ must project to $\pi(b)^{-1}$ under $\pi$. In particular, the syllable length of $b$ is strictly smaller than $\frac{m-1}{2}$.

Thus we can write $a = bs^{2r}b^{-1}bc'$, where $bc'$ is in $PA$ and has strictly smaller syllable length than $a$. Since $bs^{2r}b^{-1}$ and $bc'$ are in $PA$, 
in $H_1(PA,\mathbb Z)$  we have the relation between the images $$\overline a = \overline{bs^{2r}b^{-1}}+ \overline{bc'}.$$ 

Therefore, by induction on the syllable length of $a$, we can assume that $\bar a \in H_1(PA,\mathbb Z)$ is the sum of images of elements $b_is^{2r}b_i^{-1}$ and $b_it^{2r}b_i^{-1}$, where the syllable length of the $b_i$ is $\le \frac{m-2}{2}$ and the $b_i$ only contain odd powers of $s$ and $t$ (or they are trivial). 
Therefore, the image of $\Delta^{2n}$ in $H_1(PA, \mathbb Z)$ could be written as a sum of elements $b_is^{2r}b_i^{-1}$ and $b_it^{2r}b_i^{-1}$ with syllable length of the $b_i \le \frac{m-2}{2}$. It is also easy to see that these $b_i$ can only have syllable length $\frac{m-2}{2}$ for a single $i$.

However,  this contradicts Lemma \ref{l:homology}. 
Recall that $H_1(PA,\mathbb Z)$ has a standard basis given by the hyperplanes in the dihedral arrangement, and the image of $b_is^{2}b_i^{-1}$ is precisely the basis vector $e_r$, where $r$ is the conjugate $\pi(b_i)s\pi(b_i)^{-1}$.

In particular, if $m$ is odd, then any such sum as above misses the hyperplane corresponding to the longest element of $W$, which is a conjugate $wsw^{-1}$ where $w$ has length $\frac{m-1}{2}$. If $m$ is even, there are two hyperplanes corresponding to conjugates $wsw^{-1}$ and $w't(w')^{-1}$ where $w$ and $w'$ have length $\frac{m-2}{2}$, hence the elements in our sum miss one of these. We will refer to these as \emph{longest hyperplanes}.
Since $\Delta^{2n}$ maps in $H_1(PA,\mathbb Z)$ to a vector with nontrivial $e_r$-term for each $r \in R$, this is a contradiction as the image of $g$ obviously misses these longest hyperplanes as well. This completes the proof of the first statement.

Now, suppose that $m_{st} > 3$ and $d_{B(A)}(\Delta^{2n}, g) = \pi$. We first consider the odd case, where $m = 2k+1$. We can write $$\Delta^{2n} = ag$$ where the syllable length of $a$ is equal to $m$. 
In the decomposition $a = bc$ above, we claim that $b$ must have syllable length $k$. Otherwise, we can again write the images in $H_1(PA,\mathbb Z)$ as $\overline a = \overline{bs^{2r}b^{-1}}+ \overline{bc'}$. If $b$ has syllable length $< k$ then both $\overline{bs^{2r}b^{-1}}$ and $\overline{bc'}$ will miss the vector in $H_1$ corresponding to the longest hyperplane. 

Therefore, without loss of generality we have $a = bc$ and the syllable length of $b$ is $k$. Since $a$ is in the pure Artin group, $c$ is of the form $s^{2r}x$ or $t^{2r}x$ where $x$ projects to $\pi(b)^{-1}$ (in particular each term of $x$ has odd exponent). 

Therefore, without loss of generality we can rewrite $a = bs^{2r}b^{-1}a'$. By the above, $a' = b'c'$, where $b'$ has all odd exponents, has syllable length $k-1$, and where $c' = s^{2r}x'$ or $t^{2r}x'$ as above. Repeating this argument gives us that in $H_1(PA,\mathbb Z)$, $$\overline a = e_{r_1} + e_{r_2} + \dots + e_{r_{k+1}}$$ where the length of the reflections $r_i$ is strictly decreasing. If $m_{st} > 3$, then there are two hyperplanes of length $1 < \ell(r) < m$, and hence the image $\overline a$ misses one of those.

If $m = 2k$ is even, the proof essentially extends. 
The point here is that there are two longest hyperplanes and the image $\overline a$ will again miss one of these. More precisely, suppose that $a = bc$ as above. 
In order to hit one of the longest hyperplanes, $b$ must have syllable length $k-1$. This implies that $c = s^{2r}x$ or $c = t^{2r}x$, and $x$ projects to $\pi(b)^{-1}$. 
In particular, we must have that $x = yt^{2m}$ or $x = ys^{2m}$ where $y$ projects to $\pi(b)^{-1}$. Therefore, we can push the last syllable into $g$. 
Since we now have a word $a'$ of length $< m$, the proof in the first case rules out this possibility. 

We now prove the last statement. We will now  assume that $m_{st} = 3$, i.e.\ $A$ is the braid group on $3$-strands. We need the following lemma:

\begin{lemma}
If $n > 0$ we have that 
$$\Delta^{2n} = st^{2n}s\underbrace{t^2s^2\ \cdots}_{2n-1 \text{ terms}} = ts^{2n}t\underbrace{s^2t^2\ \cdots}_{2n-1 \text{ terms}}$$ 
$$ \Delta^{-2n} = s^{-1}t^{-2n}s^{-1}\underbrace{t^{-2}s^{-2}\ \cdots}_{2n-1 \text{ terms}} = t^{-1}s^{-2n}t^{-1}\underbrace{s^{-2}t^{-2}\ \cdots}_{2n-1 \text{ terms}}$$
\end{lemma}

\begin{proof}
We only prove the first equalities for both $\Delta^{2n}$ and $\Delta^{-2n}$; the same argument with $s$ and $t$ switched will give the second.
Since $\Delta^2 = st^2st^2 = t^2st^2s$ we have by induction that for $n > 1$ $$\Delta^{2n}=st^{2n-2}s\underbrace{t^2s^2\ \cdots}_{2n -3 \text{ terms}}(t^2st^2s) = st^{2n-2}(t^2st^2s)s\underbrace{t^2s^2\ \cdots}_{2n -3 \text{ terms}}$$
$$ = st^{2n}st^2s^2\underbrace{t^2s^2\ \cdots}_{2n -1 \text{ terms}}$$

Similarly, $\Delta^{-2} = s^{-1}t^{-2}s^{-1}t^{-2} = t^{-2}s^{-1}t^{-2}s^{-1}$, so by induction for $n > 1$ we have $$\Delta^{-2n} = s^{-1}t^{-2n+2}s^{-1}\underbrace{t^{-2}s^{-2}\ \cdots}_{2n-3 \text{ terms}} t^{-2}s^{-1}t^{-2}s^{-1} = s^{-1}t^{-2n}s^{-1}\underbrace{t^{-2}s^{-2}\ \cdots}_{2n-1 \text{ terms}}$$
\end{proof}

Now, suppose that $\Delta^{2n} = ag$ where $a$ has syllable length $3$ and $g \in \langle s^4,t^4\rangle$. Then by Lemma~\ref{l:homology} without loss of generality we can assume that $a = s^kt^{2n}s^l$ where $k$ and $l$ are odd integers. We will assume that $n > 0$, a similar argument works for $n < 0$.

By conjugating with an even power of $s$, we can assume that $k = 1$. 
Then we have 
$$st^{4n}s\underbrace{t^2s^2\ \cdots}_{2n -1 \text{ terms}} = st^{4n}s^lgs^{k-1}$$ so in particular 
$\underbrace{t^2s^2\ \cdots}_{2n -1 \text{ terms}} = s^{l-1}gs^{k-1}$, where $l-1$ and $k-1$ are even. 

Since the original Tits Conjecture holds for $A$, we must have that $\underbrace{t^2s^2\ \cdots}_{2n +1 \text{ terms}}$ is equivalent to $s^{l-1}gs^{k-1}$ in the free group on $s$ and $t$. But the powers of $s$ and $t$ in $g$ are powers of $4$, which is a contradiction. 
\end{proof}


\subsection{Proof of Generalized Tits Conjecture in the locally reducible case}
Let $A$ be a locally reducible Artin group and $\RA$ the associated RAAG.
Let $K$ be the fundamental domain for the action of $A$ on its Deligne complex $D(A)$ with the metric as described in Section~\ref{sec:Deligne complex}. Let $\hat A$ denote the RAAG subgroup that is in the original Tits Conjecture (i.e.\ the generators of $\hat A$ correspond to generators of $A$).
In \cite{charney} Charney defined a complex $$\widehat D(A) = \widehat A \times K/\sim,$$ where $(a_1,x) \sim (a_2,x)$ if and only if $x\in K_{\ge T}$ and $a_1^{-1}a_2 \in \widehat A_T$. 
This is not the Deligne complex for $\widehat A$, since $\widehat A_T$ may not be spherical. 
We define $$\widetriangle D(A) =RA \times K/\sim,$$ where $(a_1,x) \sim (a_2,x)$ if and only if $x\in K_{\ge T}$ and $a_1^{-1}a_2 \in RA_T$. 

The Moussong metric on $K$ induces a piecewise Euclidean metric on $\widetriangle D(A)$, and the homomorphisms $\Phi_{N}: \RA \rightarrow A$ define an induced map $\widetriangle \Phi_{N}: 
\widetriangle D(A) \rightarrow D(A)$ which sends $s \times K$ isomorphically onto $\Phi_{N}(s) \times K$. 

If $A$ is spherical, then both $\widehat D(A)$ and 
$\widetriangle D(A)$ have a cone point, labeled by $A$. We let $\widetriangle B(A)$ be the link of this cone point in $\widetriangle D(A)$, and $\widehat B(A)$ the link of the cone point in $\widehat D(A)$. Note that $\widehat B(A)$ is naturally a subcomplex of $\widetriangle B(A)$. 
If $A$ is dihedral, then we denote the generators of $\RA$ by $\{z_s,z_t,z_{\{s,t\}}\}$. In particular, $\Phi_N(z_{\{s, t\}}) = \Delta^{2N}$.

\begin{lemma}\label{l:disjoint}
If $A$ is a dihedral Artin group, then $\widetriangle B(A)$ is the disjoint union $$\widetriangle B(A) =  \bigsqcup_{i \in \zz} z_{\{s,t\}}^i\widehat B(A).$$
\end{lemma}

\begin{proof}
The adjacent vertices in $\widetriangle B(A)$ to $w\langle s \rangle$ are $w\emptyset$ and $wz_s\emptyset$, (and similarly for $w\langle t \rangle$). Since $w = gz_{\{s,t\}}^i$ and $wz_s = gz_sz_{\{s,t\}}^i$ for some $g \in F_2$, the power $z_{\{s,t\}}^i$ is the same for $w$ and $ws$. Therefore, any vertices that can be connected by a path in $\widetriangle B(A)$ have the same power of $z_{\{s,t\}}$. 
The $z_{\{s,t\}}$-action on $\widetriangle B(A)$ identifies each connected component of $\widetriangle B(A)$ with $\widehat B(A)$. 
\end{proof}

\begin{Prop}\label{prop:pi-distance preserving}
Let $A$ be a totally reducible, finite type Artin group with $m_{st} \ne 3$ for each factor. Then the induced map $\Phi_1: \Lk(A, \widetriangle D(A))\to \Lk(A, D(A))$ is $\pi$-distance preserving. For any totally reducible finite type Artin group, the induced map $\Phi_2: \Lk(A, \widetriangle D(A))\to \Lk(A, D(A))$ is $\pi$-distance preserving.
\end{Prop}

\begin{proof}
In this proof $N=1$ or $2$, depending on whether there are $s,t$ with $m_{st}=3$.  
We have that $A$ and $\RA$ decompose as $$A = A_1 \times A_2 \times \dots \times A_k;\quad \RA = \RA_1 \times \RA_2 \times \dots \times \RA_k$$ where each $A_i$ is an irreducible spherical subgroup of rank $2$ or $\zz$. 
Therefore, both $B(A)$ and $\widetriangle B(A)$ decompose as spherical joins $$B(A) = B(A_1)\ast B(A_2) \ast \dots \ast B(A_k);\quad \widetriangle B(A) = \widetriangle B(A_1)\ast \widetriangle B(A_2) \ast \dots \ast \widetriangle B(A_k),$$ so by Lemma \ref{l:ojoins} it suffices to check $\pi$-distance preserving for each $A_i$.
If $A = \mathbb Z$ this is obvious, so suppose that $A$ is a dihedral Artin group. 

By \cite[Lem 4.1]{charney}, the induced map $\widehat \Phi_{N}:\widehat B(A) \to B(A)$ is $\pi$-distance preserving. Since the map $\Phi_N$ is equivariant, this implies that the induced map is $\pi$-distance preserving on each component $(\Delta^{2N})^i\widehat B(A)$. 
By Lemma \ref{l:disjoint}, it suffices to verify that for $x,y$ in two different copies of $\widehat D(A)$ in $\widetriangle D(A)$, 
their images $\widetriangle \Phi_N(x), \widetriangle \Phi_N(y)$ have distance at least $\pi$ in $D(A)$.
Let $x$ lie in an edge of $\Delta^{2n_1}\widehat B(A)$ and $y$ lie in an edge  of $\Delta^{2n_2}\widehat B(A)$, where $g_1, g_2\in \widehat A$, and $n_1\neq n_2$. Then $x$ and $y$ are within distance $\pi/2m$ from vertices $\Delta^{2n_1}g_1$ and $\Delta^{2n_2}g_2$.
By Proposition~\ref{lem:dihedral}, $$d_{B(A)}(\Delta^{2n_2}g_2, \Delta^{2n_1}g_1) = d_{B(A)}(\Delta^{2(n_2-n_1)}, g_1g_2^{-1})\geq \pi + \pi/m.$$ 
This implies that the images of $x$ and $y$ are $\ge \pi$ apart.
\end{proof}

\begin{Theorem}\label{thm:locally reducible}
The map $\Phi_2:RA\to A$ is injective for every locally reducible Artin group. If $m_{st} \ne 3$ for all $s,t \in S$, then $\Phi_1:RA\to A$ is injective.
\end{Theorem}

\begin{proof}
In this proof $N=1$ or $2$, depending on whether there are $s,t$ with $m_{st}=3$.  
We prove that the induced map $\widetriangle \Phi_N:\widetriangle D(A)\to D(A)$ is an isometric embedding. 
By Lemma \ref{l:isomembedding}, $\widetriangle \Phi_N:\widetriangle D(A)\to D(A)$ is an isometric embedding provided that the map $\Lk(x,\widetriangle D(A))\to \Lk(\widetriangle \Phi_N(x), D(A))$ induced by $\widetriangle \Phi_N$ is $\pi$-distance preserving for every $x\in\widetriangle D(A)$. We only check this where $x$ is a vertex, essentially the same argument works for all $x$.
Since $\widetriangle \Phi_N$ is equivariant, it suffices to check vertices of $K$. For $T$ a spherical subset and $v_T \in \widetriangle D(A) \cap K$, the link of $v_T$ decomposes as $\Lk(x,K_{\ge T}) \times \widetriangle B(A_T)$. 
Since the link of $v_T$ in $D(A)$ decomposes as $\Lk(x,K_{\ge T}) \times B(A_T)$ and the map between links decomposes as $\Id \times \Phi_N$ the result follows from Proposition~\ref{prop:pi-distance preserving}. 
\end{proof}

\subsection{Intersections with special subgroups}

Finally, we use a coning trick to show that for any locally reducible Artin group $A$ and special subgroup $A_T$, we have that our $\RAAG$ subgroup for $A$ intersects $A_T$ in the $\RAAG$ subgroup for $T$. 
Charney used CAT(0) geometry to prove this for the $\RAAG$ provided by the original Tits Conjecture \cite[Thm 5.2]{charney}. 
Since $\Phi_2: \RA \rightarrow A$ injective, by Theorem~\ref{thm:locally reducible}, we write $\RA_T$ for the image of $\RA_T$ under $\Phi_2$.
\begin{Prop}\label{p:cone} Let $A$ be a locally reducible Artin group. Then $\RA\cap A_T = \RA_T$.
\end{Prop}

\begin{proof}
Suppose that there was a reduced word $w$ in $\RA$ so that $w \notin \RA_T$ but $w \in A_T$. Define a larger Artin group $\widetilde A$ by ``coning" off $T$, i.e.\ introduce a new generator $s$ so that $m_{st} = 2$ for all $t \in T$ and $m_{su} = \infty$ otherwise. 
Then $\widetilde A$ is a locally reducible Artin group, so we know the Generalized Tits Conjecture for $\widetilde A$. 
Note that the $\RAAG$ for $\widetilde A$ is just the $\RAAG$ for $A$ with the $\RAAG$ for $A_T$ coned off. Now, by assumption we have that $[s,w] = 1$, which contradicts the $\RAAG$ for $\widetilde A$ injecting into $\widetilde A$, as the centralizer of $s$ in that $\RAAG$ is the $\RAAG$ subgroup $\RA_T \times \langle s \rangle$. 
\end{proof}

\begin{remark}
The same argument shows that the $\RAAG$ subgroup that Crisp and Paris find intersects each special Artin group in the expected way. 
If we knew the Generalized Tits Conjecture for all Artin groups, then we would know this intersection property as well. 
\end{remark}

\section{$\RAAG$ subgroups of $\RAAG$s}\label{s:raags}

In this section, we study whether a subgroup of a $\RAAG$ generated by words which are powers of commuting elements is the ``obvious" $\RAAG$. In Sections \ref{s:small} and \ref{s:folding} we apply these ideas to the Generalized Tits Conjecture, but we hope that this will be of independent interest.

In general, this is a delicate question; the subgroup may be a $\RAAG$ but a different one than expected, or may not be a $\RAAG$ at all. See \cite{kim-koberda} and \cite{koberda} for a detailed analysis and many (positive and negative) examples.
Our main goal in this section is Theorem~\ref{thm:generalized PP}, which generalizes a condition on the words given in \cite{koberda} called \emph{Property PP} (short for ping-pong). 

\subsection{Koberda's Property PP}

Let $L$ be a flag complex and $\RA_L$ the corresponding $\RAAG$.  
For each simplex $\gs \in L$, let $w_\sigma$ be a (possibly trivial) word with all positive powers (or all negative powers) of the generators corresponding to vertices in $\gs$ (in particular, if $w_\gs$ is nontrivial it has a nontrivial power of each generator in $\gs^{(0)}$). 

The collection $\{w_\gs\}_{\gs \in L}$ determines a flag complex $L'$. The vertex set of $L'$ is $\{\gs\in L : w_\gs \text{ is nontrivial}\}$ (the reader can imagine the vertex at the barycenter of $\gs$), and the simplices correspond to collections $\gs_1,\dots, \gs_k$ where $w_{\gs_1},\dots w_{\gs_k}$ pairwise commute in $\RA_L$. Of course, we have that $w_\gs$ and $w_\gt$ commute if any only if $\gs$ and $\gt$ span a simplex in $L$.

\begin{definition}\label{def:PP}
We say that the collection $\{w_\gs\}_{\gs \in L}$ satisfies \emph{Property $PP$} if there is an injective simplicial map $p: L' \rightarrow L$, so that
\begin{itemize} 
\item $p(L')$ is a full subcomplex of $L$, 
\item for every vertex $\gs$ of $L'$ we have $p(\gs)\in \gs$, and
\item if vertices $\gs, \gt$ are joined by an edge in $L'$, 
then $p(\gs)\notin \gt$. 
\end{itemize}
\end{definition}
One can think of this map $p$ as choosing a representative vertex in $L$ for each $w_\sigma$. 
The requirement that $p(L')$ is a full subcomplex ensures that if two words $w_\gs$ and $w_{\gs'}$ do not commute, then their representative vertices do not commute. 
The following proposition of Koberda  follows quickly from the normal form for $\RAAG$'s.


\begin{Prop}[{\cite[Lem 4.2]{koberda}}]\label{prop:koberda subgroup}
Let $w_\gs$ have Property $PP$. 
Then the map $f:\RA_{L'} \to \RA_L$ sending the generators $\sigma$ of $\RA_{L'}$ to $w_\gs$, is injective. Therefore, the subgroup generated by $\langle w_\gs \rangle$ is isomorphic to $\RA_{L'}$. 
\end{Prop}

We note that the statement of property PP in \cite{koberda} does not include the third condition of Definition~\ref{def:PP}.
However without that assumption Proposition~\ref{prop:koberda subgroup} fails to hold:
\begin{example}\label{e:PP failure}
Consider a graph on four vertices $s,t,u,x$ where $s,t,u$ are pairwise adjacent, and $x$ is adjacent to $u$. 
The associated Artin group $A$ is $(\mathbb Z^2*\mathbb Z)\times \mathbb Z$.
Let ${st, stu, x}$ be a collection of words and let $L'$ be the corresponding flag complex, which is the disjoint union of an interval joining $st$ and $stu$ and a point $x$. The map $p:L'\to L$ sending $st\mapsto s$, $stu\mapsto t$ and $x\mapsto x$ satisfies the first two conditions in Definition~\ref{def:PP}. However, the subgroup $\langle st, stu, x\rangle$ is not isomorphic to the associated RAAG $\mathbb Z^2*\mathbb Z$ since $[st(stu)^{-1}, x] = 1$.
\end{example}
The proof of Lemma 4.2 in \cite{koberda} is correct for the property PP as stated in Definition~\ref{def:PP}. 
The last sentence of the second paragraph of the proof is not true without the third condition of Definition~\ref{def:PP}.

\begin{example}\label{e:badpp}
The following example is taken from \cite{koberda}, where it is attributed to M. Casals. Let $\RA_L = F_2 \times F_2$ where the first $F_2$ is generated by $a,c$ and the second generated by $b,d$. Consider the subgroup $H_n < \RA_L$ generated by $a^n, d^n$ and $(bc)^n$ for $n \in \mathbb N$. It turns out that $H_n$ is not isomorphic to any right-angled Artin group. We will not provide the full proof of this, the key point is that in this group there is the relation $$[a^n, (bc)^nd^n(bc)^{-n}] = 1.$$
Note that this collection of words does not satisfy property PP, in this case $L'$ is the disjoint union of an edge and a point, and hence there is no injective map from $L' \rightarrow L$ where the image is a full subcomplex. 
\end{example}

\subsection{Generalization of Property PP}
We start with a motivating example. 
\begin{example}
Let $L$ be a path with $4$ vertices $a,b,c,d$, and let $\RA_L$ be the corresponding $\RAAG$. Consider the subgroup $H$ generated by $\{a,d,bc\}$. By the same reasoning as Example \ref{e:badpp}, this collection of words do not satisfy property $PP$ (in this case $L'$ is $3$ points and there is no injective map from $3$ points to $L$ with image a full subcomplex). However, we claim that the subgroup $H$ generated by these words is still isomorphic to the free group $F_3$.

To see this, note that $\RA_L$ splits as the amalgamated product $$\RA_L = \langle a,b,c \rangle \ast_{\langle b,c \rangle} \langle b,c,d \rangle$$ and $F_3$ decomposes as $F_2 \ast_\zz F_2$. 
We can use Property $PP$ to say that the subgroups generated by $\langle a,bc \rangle$ and $\langle bc,d \rangle$ inside $\langle a,b,c \rangle$ and $\langle b,c,d \rangle$ respectively are both $F_2$. 
Furthermore, each of these subgroups intersects $\langle b,c \rangle$ in $\langle bc \rangle$ (this is not completely obvious and generalizing this is the majority of our work below). Therefore, we can apply the following lemma of Serre:
\end{example}

\begin{Prop}[{\cite[Chap 1 Prop 3]{serre}}]\label{p:serre} 
Let $G_i$ be a collection of groups with common subgroup $A$ and let
 $\aster_A, G_i$ denote the amalgamated product. Let $H_i\subseteq G_i$ be subgroups and suppose the intersection $B=H_i\cap A$ is independent of $i$. Then the natural homomorphism $\aster_B\ H_i \to \aster_A\, G_i$ is injective.
\end{Prop}

So, in the above example, we get an injection from $F_3 \cong F_2 \ast_\zz F_2 \rightarrow \langle a,bc,d \rangle$, which is obviously an isomorphism.

For general $\RAAG$'s, we suppose the following: our nerve $L$ decomposes as $L = L_1 \cup_{L_0} L_2$, where each $L_i$ is a full subcomplex. 
We consider a collection of words $\{w_\gs\}_{\gs\in L}$ where, as before, $w_\sigma$ is a (possibly trivial) word with all positive powers (or all negative powers) of the generators corresponding to vertices in $\gs$. We assume that each of the collections $\{w_\gs\}_{\gs\in L_1}$, $\{w_\gs\}_{\gs\in L_2}$ satisfies Property PP in $\RA_{L_1}$, $\RA_{L_2}$ respectively. Note that the functions in property $PP$ for each $L_i$ do not have to agree on the words for $L_0$ (if they do it is easy to see that the words already satisfy property $PP$). However, we have to make some additional assumptions, to conclude that $\RA_{L'}$ embeds in $\RA_L$.

\begin{definition}\label{d:avoids}
Let $L$ be a flag complex and $L_0$ a full subcomplex. Suppose we have a collection of words $\{w_\gs\}_{\gs \in L}$ satisfying property $PP$, and $L'$ is the associated flag complex and function $p:L'\to L$. We say that the collection $\{w_\gs\}_{\gs \in L}$ \emph{avoids $L_0$} if $\gs \notin L_0$ implies that $p(\gs) \notin L_0$. 

\end{definition}

The next lemma guarantees that if $\{w_\gs\}_{\gs \in L}$ satisfies property $PP$ and avoid $L_0$, then the intersection of the $\RAAG$ subgroup (which is guaranteed by the Koberda's Property PP) generated by the $\{w_\gs\}$ intersects the special subgroup $A_{L_0}$ as expected. 

\begin{lemma}\label{l:avoids}
Let $L$ be a flag complex, $\RA_L$ the $\RAAG$ on $L$, and let $L_0$ be a full subcomplex. Let $\{w_\gs\}_{\gs \in L}$ be words satisfying Property $PP$ avoiding $L_0$. Let $\RA_{L'}$ be the $\RAAG$ subgroup of $\RA_L$ generated by the $\{w_\gs\}_{\gs \in L}$, and $\RA_{L_0'}$ the $\RAAG$ subgroup generated by $\{w_\gs\}_{\gs \in L_0}$. 
Then $\RA_{L'} \cap \RA_{L_0} = \RA_{L_0'}$. 
\end{lemma} 

\begin{proof}

Obviously, $\RA_{L_0'}$ is contained in $\RA_{L'} \cap \RA_{L_0}$. 
The proof of the other containment is similar to the coning trick in Proposition \ref{p:cone}. Make a new flag complex $K$ by coning off $L_0$. On the group level we are adding a new generator $x$ which commutes only with the generators in $L_0$.  Since $\{w_\gs\}$ satisfies Property $PP$ and avoids $L_0$, the collection $\{w_\gs\} \cup \{x\}$ satisfies Property $PP$. Therefore, by Proposition \ref{prop:koberda subgroup} this collection generates the obvious RAAG subgroup $RA_{K'}$ of $RA_{K}$. If $\RA_{L'} \cap \RA_{L_0}$ contained an element $w$ not in $\RA_{L_0'}$, then $w$ would commute with $x$, and this relation does not appear in $RA_{K'}$. 
\end{proof}

Lemma \ref{l:avoids} combined with Proposition \ref{p:serre} implies the following theorem:

\begin{Theorem}\label{thm:generalized PP}
Let $L$ be a flag complex and suppose that $L  = L_1 \cup_{L_0} L_2$ where each $L_i$ is a full subcomplex. Let $\{w_\gs\}_{\gs \in L}$ be words so that each subcollection $\{w_\gs\}_{\gs \in L_i}$ satisfies property $PP$ and avoids $L_0$. Then the subgroup generated by the $\{w_\gs\}_{\gs \in L}$ is the $\RAAG$ based on $L'$. 
\end{Theorem} 

\begin{proof} 

By Property $PP$, we know that the subgroup $\RA_{L_i'}$ of $\RA_{L_i}$ generated by $\{w_\gs\}_{\gs \in L_i}$ is a $\RAAG$. 
By Lemma~\ref{l:avoids}, we know that $$\RA_{L_1'} \cap \RA_{L_0} = \RA_{L_2'} \cap \RA_{L_0} = \RA_{L_0'}.$$
By Proposition \ref{p:serre}, $\RA_{L'} = \RA_{L_1'} \ast_{\RA_{L_0'}} \RA_{L_2'}$ injects into $\RA_L = \RA_{L_1} \ast_{\RA_{L_0}} \RA_{L_2}$. 
\end{proof}

\begin{remark}
The words $b,d,ac$ in Example \ref{e:badpp} do not satisfy generalized PP for any decomposition of $L$. For example, if we take $L_1$ the full subcomplex containing $a,b,c$ and $L_2$ the full subcomplex containing $b,c,d$, then for the edge $ac$ we have to choose the vertex $c$ in property PP, which is in $L_1 \cap L_2$. Note that the subgroup generated by $ac$ and $b$ contains $cbc^{-1}$, and hence intersects the subgroup $\langle b,c \rangle$ outside of $\langle b \rangle$.

\end{remark}

\section{Generalized Tits Conjecture and small type spherical Artin groups}\label{s:small}
In this section, we describe how the Generalized Tits Conjecture holds for small-type spherical Artin groups of type $A_n$ and $D_n$. The main work is type $D_n$, where we require the generalized Property PP of the previous section.

Let $A$ be any small-type spherical Artin group with the standard generating set $S$, and the Perron-Vannier representation $A\to \Mod(\Sigma)$. Let $\RA$ be the associated $\RAAG$ with the presentation \ref{eq:raag}. Let us record the following observation.
\begin{lemma}
If the composition $\RA\xrightarrow{\Phi_k} A\to \Mod(\Sigma)$ is injective, then $A$ satisfies the Generalized Tits Conjecture for $N = k$.
\end{lemma}

For every irreducible subset $T\subseteq S$, let  $A_T\to \Mod(\Sigma_T)$ be its Perron-Vannier representation of the special subgroup $A_T$. If $T=\{s\}$, by $\Sigma_T$ as well as $\partial \Sigma_T$ we mean the single curve in $\Sigma$ corresponding to $s$.
The surface $\Sigma_T$ can be embedded in $\Sigma$ and that embedding induces a homomorphism $\Mod(\Sigma_T)\to \Mod(\Sigma)$ which makes the following diagram commute:
\begin{equation*}
\begin{tikzcd}[row sep=huge]
A_T \arrow[r] \arrow[hookrightarrow, d] &
\Mod(\Sigma_T) \arrow[d] \\
A\arrow[r] &
\Mod(\Sigma)
\end{tikzcd}
\end{equation*}
Note that $\partial \Sigma_T$ is a multicurve in $\Sigma$.
By Remark~\ref{rem:deltas boundaries}, $\Delta^4_T$ is sent to a multitwist around the boundary components of $\partial \Sigma_T$ in $\Mod(\Sigma)$.
We summarize the above discussion in the following lemma.
\begin{lemma}\label{lem:generators multitwists}
Let $A$ be a small type spherical Artin group and let $N\geq 2$. Then each of the generators $\{z_T\}_{T}$ of the $\RAAG$ $\RA$ is mapped to a multitwist in $\Mod(\Sigma)$ via $\rho\circ \Phi_N$.
\end{lemma}

The following theorem of Koberda is our main tool for proving the generalized Tits conjecture for the spherical Artin groups. 
\begin{Theorem}[{\cite[Thm 1.1]{koberda}}]\label{t:koberda}
Let $f_i$ be a finite collection of nontrivial powers of Dehn twists around simple closed curves so that the subgroup $\langle f_i,f_j \rangle$ is not cyclic for all $i,j$. Then there is an $M > 0$ so that the powers $f_i^m$ generate a $\RAAG$ subgroup of $\Mod(\Sigma)$ for all $m \ge M$. 
\end{Theorem}

The flag complex that is the nerve of the RAAG generated by those powers of Dehn twists has the collection of the curves as its vertex set, where a subcollection of curves spans a simplex if and only if they can be realized as pairwise disjoint curves. The analogous result does not hold for multitwists instead of Dehn twists about single curves. Indeed, it is easy to find surfaces and multicurves so that the RAAG generated by Dehn twists about the individual curves is $F_2 \times F_2$, and the words given by the multicurves are as in Example~\ref{e:badpp}.

Coming back to the surface $\Sigma$ from the Perron-Vannier representation of $A$, let $\mathcal C$ denote the collection of all the curves in the support of the multitwists $\{\rho\circ \Phi_N(z_T)\}_T$. By Theorem~\ref{t:koberda}, high powers of Dehn twists about the curves in $\mathcal C$ generate a $\RAAG$ $\RA_{\mathcal C}$. 
Therefore, we have an induced homomorphism $\RA \rightarrow \RA_{\mathcal C} \subseteq \Mod(\Sigma)$, where each generator of $\RA$ goes to a product of commuting generators of $\RA_{\mathcal C}$. In order to prove that $A$ satisfies the Generalized Tits Conjecture with $N$, it suffices to show that $\RA \rightarrow \RA_{\mathcal C}$ is injective. In the case of $A$ of type $A_n$ and $D_n$, we show it using Property PP and generalized Property PP from Section~\ref{s:raags}, respectively.
To verify Property PP, we pick one of the boundary components $p(T)$ in $\partial \Sigma_T$ for each irreducible $T$ such that the curves $p(T), p(T')$ intersect if and only if $[\Delta_T^2, \Delta_T^2] = 1$ (if and only if $T\subseteq T'$, $T'\subseteq T$, or $[T, T']=1$). For details on how we verify generalized Property PP, see subsection on type $D_n$.

We now rephrase Property PP in terms of curves on surfaces in the context of $\RAAG$ subgroups of mapping class groups.
\begin{lemma}\label{lem:PP for surfaces}
Let $\hat{\mathcal C}$ be a collection of multicurves in a surface $\Sigma$ and let $\mathcal C$ be the union of all the connected component of elements of $\hat{\mathcal C}$. Suppose there exists a function $p:\hat{\mathcal C}\to \mathcal C$ such that 
\begin{itemize}
\item $p(\alpha)\in \alpha$ and
\item $p(\alpha), p(\beta)$ intersect if and only if some connected component of $\alpha$ intersect some connected component of $\beta$.
\item If $\alpha \cup \beta$ is a multicurve, then $p(\alpha)$ is not contained in $\beta$. 
\end{itemize}
Let $\{T_{\alpha}:\alpha\in \hat{\mathcal C}\}$ be a collection of multitwists about multicurves in $\hat{\mathcal C}$.
Then the exists $M>0$ such that $\langle T_{\alpha}^M:\alpha\in \hat{\mathcal C}\ \rangle$ is a RAAG. 
\end{lemma}
\begin{proof}
By Theorem~\ref{t:koberda} the group generated by sufficiently large powers of Dehn twists around $\mathcal C$ is a RAAG. The assumption about existence of function $p$ is just reformulation of Property PP from Definition~\ref{def:PP}. The conclusion follows from Proposition~\ref{prop:koberda subgroup}.
\end{proof}

Similarly, the generalized Property PP could also be rephrased in terms of curves in a surface.

\subsection{Intersections of boundary curves}
Let $A$ be a small-type spherical Artin group, and let $A \rightarrow \Mod(\Sigma)$ be the Perron-Vannier representation.
Suppose that $T$ and $T'$ are two irreducible spherical subsets of $S$, and let $\Sigma_T$ and $\Sigma_{T'}$ be the associated subsurfaces of $\Sigma$. If $\Delta_T^2$ and $\Delta^2_{T'}$ commute, then each of the boundary curves of $\Sigma_T$ and $\Sigma_{T'}$ are necessarily disjoint. If they do not commute, then there are different possibilities for the possible intersections between the curves. We shall record a few lemmas that we will need later.

\begin{lemma}\label{l:curves1}
Let $A$ be of type $A_n$, and suppose that $T$ and $T'$ are irreducible subsets such that $\Delta_T^2$ and $\Delta_{T'}^2$ do not commute. If $\Sigma_T$ has one boundary component up to homotopy (so $|T| = 1$ or $|T|$ is even), then the boundary curve of $\Sigma_T$ has nontrivial intersection number with every boundary curve of $A_{T'}$.
\end{lemma}

\begin{proof}

Note that $\partial \Sigma_T$ is fixed up to homotopy by the hyperelliptic involution. If $A_{T'}$ has two boundary components, then these are permuted up to homotopy by the hyperelliptic involution. Therefore, we only have to show that $\partial \Sigma_T$ has nontrivial intersection number with a single boundary component of $\partial \Sigma_{T'}$. Since the Perron-Vannier representation is injective in this case \cite{pv}, this follows from the fact that $[\Delta_T^m, \Delta_{T'}^n] \ne 1$ in $A_n$.

\end{proof}

\begin{lemma}\label{l:curves2}
Suppose $\Sigma_{T'}$ and $\Sigma_T$ both have two boundary components, i.e.\ both $|T|$ and $|T'|$ are odd and $\neq 1$. Then
\begin{itemize}

\item Each boundary component of $\Sigma_T$ intersects at least one boundary component of $\Sigma_T'$, and vice versa.

\item If $|T - T'|$ is odd, then each boundary component of $\Sigma_T$ intersects each boundary component of $\Sigma_{T'}$.
\end{itemize}

\end{lemma}

\begin{proof}
For the first statement, again by injectivity of the Perron-Vannier representation we must have that at least one component of $\partial \Sigma_T$ intersects a component of $\partial \Sigma_{T'}$. Since both components are permuted by the hyperelliptic involution, we get the other intersection. 

We will prove the second statement in the Alternative proof of Proposition \ref{prop:conjecture for An} below, once we have developed more notation.

\end{proof}

There is a similar lemma for the images of curves in the Perron-Vannier representation $A\to \Mod(\Sigma)$ for Artin groups $A$ of type $D_n$. Note that $\Sigma$ has a unique boundary curve which is contained in the boundary of $\Sigma_T$ for any irreducible subset $T\subseteq S$ of type $D_m$ for $m < n$. We will call this the \emph{central component}, and the other boundary components \emph{non-central components}. We shall see that the Dehn twist around the central component never factors into any of our calculations. 

\begin{lemma}\label{l:curves2}
Let $A$ be of type $D_n$, and suppose that $T$ and $T'$ are irreducible subsets such that $\Delta_T^2$ and $\Delta_{T'}^2$ do not commute.

\begin{enumerate}
\item If $A_T$ has type $D_m$ and $m = 2k+1$ for $k > 0$, then the non-central boundary component of $\partial \Sigma_T$ intersects each curve in $\partial \Sigma_{T'}$ if $s, s' \notin T'$, and intersects one boundary curve of $\Sigma_{T'}$ if one of $s$ or $s' \in T'$. 

\item If $A_T$ has type $D_m$ and $m = 2k$ for $k > 1$, then if $s,s' \notin T'$  each non-central boundary component either intersects all curves in $\partial\Sigma_{T'}$ or there are two pairs of components that intersect. If $s$ or $s' \in T'$, there is one non-central component of $\partial \Sigma_T$ which has nontrivial intersection with $\partial \Sigma_{T'}$. Furthermore, one component  intersects all $\partial \Sigma_{T'}$ for $T'$ that contain $s$, and one component intersects all $\partial \Sigma_{T'}$ for $T'$ that contain $s'$.

\item If $s \in T$ and $s' \in T'$, then the components of $\partial\Sigma_T$ and $\partial \Sigma_{T'}$ that are not contained in $\partial \Sigma_{D_m}$ for any $m$ always intersect. 

\end{enumerate}

\end{lemma}

\begin{proof}
The proof of the first two items is similar to Lemma \ref{l:curves1}. 
The non-central component(s) of $\partial\Sigma$ is fixed by the hyperelliptic involution.
 If $\partial \Sigma_{T'}$ has two components then if $s,s' \notin T'$ then these are permuted by the hyperelliptic involution. 
If $s$ or $s' \in T'$, then one of the curves of $\partial \Sigma_{T'}$ is contained in $\partial \Sigma_{D_m}$ for $m > n$, and hence misses the non-central component(s) of $\partial \Sigma_{D_n}$. The other curve therefore intersects $\partial \Sigma_{D_n}$
For the last statement of the second item, note that each of the boundary components of $\partial \Sigma_{D_n}$ is a boundary component of $\partial \Sigma_{A_n}$ and $\partial \Sigma_{A_n'}$. Therefore, they are disjoint from $\partial\Sigma_{A_m}$ and $\partial\Sigma_{A_m'}$ respectively for $m > n$. 

For the last item, note that both these curves have nontrivial intersection number with $\gamma_s$ and $\gamma_s'$ respectively (where $\gamma_s$ and $\gamma_s'$ are the curves corresponding to the generators $s$ and $s'$). The curves $\gamma_s$ and $\gamma_{s'}$ and the outer boundary component are the boundary of a pair of pants. If our curves had trivial intersection number, then we could homotope one of them to have trivial intersection number with both $\gamma_s$ and $\gamma_s'$.
\end{proof}

\subsection{Type $A_n$}

Suppose that $A$ is a spherical Artin group of type $A_n$, i.e.,\ $A$ is the braid group on $n+1$-strands. Then $A$ is the mapping class group of the $n+1$-punctured disc. We give the punctures an arbitrary labeling $\{1,2, \dots, n+1\}$.  The squares of the standard generators of $A$ correspond to the Dehn twists about the simple closed curves around two consecutive punctures $\{i, i+1\}$. 
Any irreducible subset $T$ of $S$ corresponds to a subset $I\subseteq\{1, \dots, n+1\}$ of consecutive numbers, where $|I| = |T|+1$. The center of the pure braid group $PB_T$ corresponds to a Dehn twist $T_I$
around punctures in $I$, see Figure \ref{f:braids}. Therefore, the Generalized Tits Conjecture for $N = 1$ asks if this collection of Dehn twists about these simple closed curves generates a $\RAAG$ subgroup of the braid group. For simplicity, we denote the generators $z_T$ of $\RA$ by $z_I$. 

\begin{example}\label{ex:braids}
The braid groups on $\ge 4$ strands do not satisfy the Generalized Tits Conjecture with $N=1$.
The map $\Phi_1$ is not injective, as there are deeper relations between the $\Delta_T^2$. See Figure \ref{f:braids}.
By the lantern relation (see e.g.\ \cite[Prop 5.1]{MCGbook}), the Dehn twist $T_r$ about the red curve in Figure~\ref{f:braids} satisfies 
$$T_r=T_{123}^{-1}T_{12}T_{23}.$$
Similarly, the Dehn twist $T_b$ about the blue curve satisfies  
$$T_b=T_{234}^{-1}T_{23}T_{34}.$$
However, the elements $z_{123}^{-1}z_{12}z_{23}$ and $z_{234}^{-1}z_{23}z_{34}$ do not commute in $\RA$. Indeed, it suffices to consider their images under the retraction $\RA\to \langle z_{123}, z_{234}\rangle\simeq F_2$.

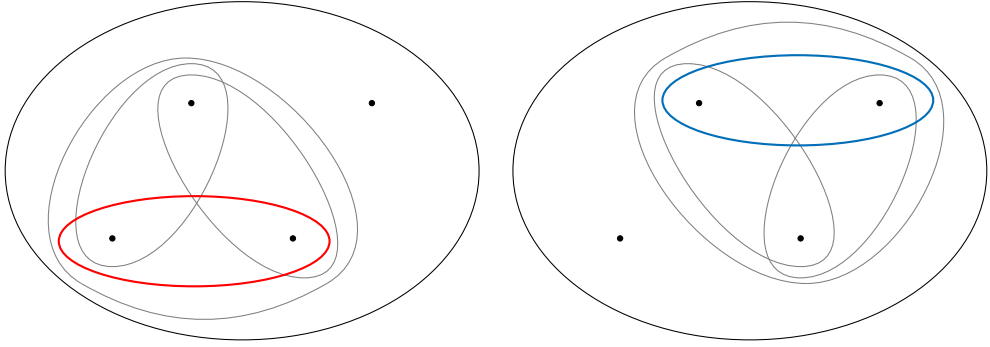
\begin{figure}
\centering
\begin{tikzpicture}[scale = .75]
\begin{scope}[xshift=-4.5cm]
\draw (0,0) ellipse (4.2cm and 3 cm);
\node[circle, draw, fill, inner sep = 0pt,minimum width = 2pt] (a) at (-2.3,-1.2) {};
\node[circle, draw, fill, inner sep = 0pt,minimum width = 2pt] (d) at (-.9,1.2) {};
\node[circle, draw, fill, inner sep = 0pt,minimum width = 2pt] (e) at (2.3,1.2) {};
\node[circle, draw, fill, inner sep = 0pt,minimum width = 2pt] (h) at (.9,-1.2) {};


\draw[gray] (-2.3, -1.7) to[out =180, in = 180] (-.9, 1.9) to[out=0, in=0] (-2.3, -1.7);
\draw[gray] (1.1,-1.9) to[out =180, in = 180] (-.9, 1.7) to[out=0, in=0] (1.1,-1.9);

\draw[red, thick ] (-.85,-1.25) ellipse (2.4cm and .8cm);

\draw[gray] (-2.85, -2) to[out=150, in=180] (-1, 2);
\draw[gray] (-1, 2) to[out=0, in=30] (1.5, -2);
\draw[gray] (-2.85, -2) to[out=-30, in=-150] (1.5, -2);

\end{scope}
\begin{scope}[xshift=4.5cm]

\draw (0,0) ellipse (4.2cm and 3 cm);
\node[circle, draw, fill, inner sep = 0pt,minimum width = 2pt] (a) at (-2.3,-1.2) {};
\node[circle, draw, fill, inner sep = 0pt,minimum width = 2pt] (d) at (-.9,1.2) {};
\node[circle, draw, fill, inner sep = 0pt,minimum width = 2pt] (e) at (2.3,1.2) {};
\node[circle, draw, fill, inner sep = 0pt,minimum width = 2pt] (h) at (.9,-1.2) {};

\draw[gray] (2.3, 1.7) to[out =180, in = 180] (.9, -1.9) to[out=0, in=0] (2.3, 1.7);
\draw[gray] (-1.1,1.9) to[out =180, in = 180] (.9, -1.7) to[out=0, in=0] (-1.1,1.9);

\draw[NavyBlue, thick] (.85,1.25) ellipse (2.4cm and .8cm);

\draw[gray] (2.85,2) to[out=-30, in=0] (1, -2);
\draw[gray] (1, -2) to[out=180, in=-150] (-1.5, 2);
\draw[gray] (2.85, 2) to[out=150, in=30] (-1.5, 2);
\end{scope}
\end{tikzpicture}
\caption{Dehn twists around the grey curves in this diagram are images of $\Delta_T^2$ for certain braid subgroups of the braid group on $8$ strands. The red and blue curves arise via the lantern relation, and obviously commute in the braid group. It is easy to check that the words in the lantern relation which produce these curves do not commute in the $\RAAG$ on the $\Delta_T^2$, and hence the induced homomorphism from this $\RAAG$ to the braid group is not injective.}
\label{f:braids}

\end{figure}

\end{example}

As we have noted, Theorem \ref{t:koberda} implies that high powers of these elements generate a $\RAAG$, i.e.\ we get the following
\begin{Prop}\label{prop:conjecture for An}
The Generalized Tits Conjecture holds for all spherical Artin groups of type $A_n$ with $N$ sufficiently large.
\end{Prop}

We shall also show how the conjecture follows from using the Perron-Vannier representation and property PP (this will serve as a warmup for the other cases). 

\begin{proof}[Alternative proof of Proposition~\ref{prop:conjecture for An}]
By Lemma~\ref{lem:PP for surfaces}, it suffices to show that the multicurves that arise in the Perron-Vannier representation satisfy the condition in the statement of the lemma. Let $S = \{t_1,\dots t_n\}$ be the standard set of generators. 

Note that if $n$ is odd, then the $\Sigma - \{t_1, t_3, \dots, t_{2k+1},\dots t_n\}$ has two connected components. Similarly, if $n$ is even, then $\Sigma - \{t_1, t_3, \dots t_{2k+1}, \dots, t_{n-1}, k\}$ has two connected components, where $k$ is an arc with both endpoints in $\partial \Sigma$ and which intersect only $t_n$ among the curves in $S$. In either case, we pick a connected component, and denote it by $\Sigma_+$. 
Similarly, let $\Sigma_{\#}$ be the connected component that does not contain $\partial \Sigma$ of respectively $\Sigma-\bigcup\{t_2, t_4, \dots t_{n-1}, \ell, \ell'\}$ in case of odd $n$, and of $\Sigma-\bigcup\{t_2, t_4, \dots t_{n}, k'\}$ in case of even $n$. Here, $\ell,\ell'$ denote arcs with the endpoint in two connected components of $\partial
\Sigma$ where $\ell$ intersects only $t_1$, and $\ell'$ intersects only $t_n$ among curves in $S$. The arc $k'$ has both endpoint in $\partial \Sigma$ and intersects only $t_1$ among curves in $S$.

Let $T  = \{t_i, \dots, t_j\}$. If $|T|$ is even, then $\Sigma_T$ has a unique boundary component, denoted by $t_{i:j}$. Otherwise, if $|T|$ is odd, then we denote the two boundary component of $\Sigma_T$ by $\{t_{i:j}, t_{i:j}'\}$. See Figure~\ref{fig:curves in An}. 
If additionally, $i$ is odd, then again exactly one of these curves is contained in $\Sigma_+$. We assume it is $t_{i:j}$.
Similarly, if $i$ is even, then we assume that $t_{i:j}$ is contained in $\Sigma_\#$. 
For each irreducible set $T=\{t_i,\dots, t_j\}$, we set $p(T) = t_{i:j}$, if $i>1$.

We claim that our choice function $p$ satisfies Lemma~\ref{lem:PP for surfaces}.
By the first part of Lemma~\ref{l:curves1}, we only need to worry about irreducible subsets of odd cardinality. Suppose we have two subsets $T  = \{t_i, \dots, t_j\}$ and $T'  = \{t_{i'}, \dots, t_{j'}\}$ of odd cardinality which do not commute, i.e.\ $[z_T, z_{T'}]\neq 1$ in $\RA$. 

We claim that if $i$ is odd and $i'$ is even or vice versa, then each curve in $\partial \Sigma_T$ intersects each curve in $\partial \Sigma_{T'}$, (this is the second part of Lemma~\ref{l:curves2}). To see this, note that it suffices to assume that $T = \{1,2,\dots, j\}$ and $T' = \{k, \dots, l\}$. Since $j$ is odd, $\partial \Sigma_T$ has two components, and both components of $\partial \Sigma_{T'}$ intersect the curve $t_{k-1}$ exactly once (and miss all other $t_i$ for $i \le j$). Therefore, each component of $\partial\Sigma_{T'}$ intersects both components of $\Sigma_T - \{t_1, t_3, \dots, t_{k-1}, \dots t_j\}$, and therefore must intersect both components of $\partial \Sigma_T$.

If $i$ and $i'$ are both odd, then $p(T)$ cannot intersect the curve of $\partial \Sigma_{T'}$ which is contained in $\Sigma-\Sigma_+$, and by the second part of Lemma \ref{l:curves1}, $p(T)$ intersects $p(T')$. Similarly, if $i$ and $i'$ are both even, then $p(T)$ and $p(T')$ intersect.

Finally, note that the third condition in Lemma \ref{lem:PP for surfaces} is trivially satisfied, as no curve is contained in $\partial \Sigma_T$ and $\partial \Sigma_{T'}$ for $T \ne T'$. 
\end{proof}

\begin{figure}
\centering
\begin{tikzpicture}[scale=0.8]
\begin{scope}
\draw (2,-2) to[out=180, in=180] (2,2);
\draw (2,-2) to (8, -2);
\draw (2,2) to (8, 2);
\draw (2, -0) to[out=-30, in=210] (3, 0);
\draw (2.2,-0.1) to[out=40, in=140] (2.8, -0.1);

\draw (4, -0) to[out=-30, in=210] (5, 0);
\draw (4.2,-0.1) to[out=40, in=140] (4.8, -0.1);

\draw (6, -0) to[out=-30, in=210] (7, 0);
\draw (6.2,-0.1) to[out=40, in=140] (6.8, -0.1);

\draw[dotted, thick] (7.5,0) to (7.75,0);

\draw[ForestGreen] (6.5, 0) to[out=70 , in=290] (6.5, 2);
\draw[ForestGreen, dotted] (6.5, 0) to[out=110 , in=250] (6.5, 2);
\node[ForestGreen, draw=none,fill=none] at (7.05,1.4) {$t_{1:5}'$};

\draw[ForestGreen, very thick] (6.5, -0.15) to[out=290 , in=70] (6.5, -2);
\draw[ForestGreen, dotted, very thick] (6.5, -0.15) to[out=250 , in=110] (6.5, -2);
\node[ForestGreen, draw=none,fill=none] at (7.05,-1.4) {$t_{1:5}$};

\draw[ForestGreen] (4.5, 0) to[out=70 , in=290] (4.5, 2);
\draw[ForestGreen, dotted] (4.5, 0) to[out=110 , in=250] (4.5, 2);
\node[ForestGreen, draw=none,fill=none] at (5.05,1.4) {$t_{1:3}'$};

\draw[ForestGreen, very thick] (4.5, -0.15) to[out=290 , in=70] (4.5, -2);
\draw[ForestGreen, dotted, very thick] (4.5, -0.15) to[out=250 , in=110] (4.5, -2);
\node[ForestGreen, draw=none,fill=none] at (5.05,-1.4) {$t_{1:3}$};

\draw[ForestGreen,dotted,very thick] (5.5, -2) to[out=105, in=255] (5.5, 2);
\draw[ForestGreen, very thick] (5.5,2) to[out=285, in=75] (5.5, -2);
\node[ForestGreen, draw=none,fill=none] at (6.1,1.2) {$t_{1:4}$};

\draw[ForestGreen,dotted, very thick] (3.5, -2) to[out=105, in=255] (3.5, 2);
\draw[ForestGreen,very thick] (3.5,2) to[out=285, in=75] (3.5, -2);
\node[ForestGreen, draw=none,fill=none] at (4.1,1.2) {$t_{1:2}$};
\end{scope}
\begin{scope}[shift={(8,0)}]
\draw (2,-2) to[out=180, in=180] (2,2);
\draw (2,-2) to (9, -2);
\draw (2,2) to (9, 2);

\draw (2, -0) to[out=-30, in=210] (3, 0);
\draw (2.2,-0.1) to[out=40, in=140] (2.8, -0.1);

\draw (4, -0) to[out=-30, in=210] (5, 0);
\draw (4.2,-0.1) to[out=40, in=140] (4.8, -0.1);

\draw (6, -0) to[out=-30, in=210] (7, 0);
\draw (6.2,-0.1) to[out=40, in=140] (6.8, -0.1);

\draw (8, -0) to[out=-30, in=210] (9, 0);
\draw (8.2,-0.1) to[out=40, in=140] (8.8, -0.1);

\draw[dotted, thick] (9.5,0) to (9.75,0);

\draw[purple,very thick] (5.5, 0) ellipse (1.8 and 0.5);
\node[purple, draw=none,fill=none] at (5.5,-0.25) {$t_{4:6}$};

\draw[purple, very thick] (2.7, -0.03) to[out= 30, in = 90] (7.5,0) to[out =270, in=340] (2.7, -0.13);
\draw[purple, dotted, very thick] (2.7, -0.03) to[out= 35, in = 90] (7.6,0) to[out =270, in=335] (2.7, -0.13);
\node[purple, draw=none,fill=none] at (4.05,-0.9) {$t_{3:6}$};

\draw[purple] (2.5, 0.01) to[out= 60, in = 120] (8.5,0.01);
\draw[purple, dotted] (2.5, 0.01) to[out= 65, in = 115] (8.5,0.01);
\draw[purple,very thick] (2.5, -0.15) to[out= -60, in = -120] (8.5,-0.15);
\draw[purple, dotted, very thick] (2.5, -0.15) to[out= -65, in = -115] (8.5,-0.15);
\node[purple, draw=none,fill=none] at (7.6,1.35) {$t_{3:7}'$};
\node[purple, draw=none,fill=none] at (7.9,-1.35) {$t_{3:7}$};

\end{scope}
\end{tikzpicture}
\caption{Curves $s_2, s_3, s_3',\dots s_{2k}, s_{2k+1}, s_{2k+1}'\dots$ and curves $t_{i:j}, t_{i:j}'$ in the surface $\Sigma$ of type $A_n$. The thick curves are an example of a choice of the representatives for property PP.}
\label{fig:curves in An}
\end{figure}
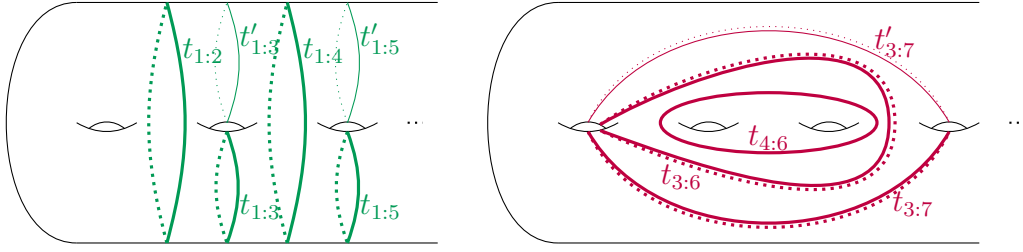

We do not know if $N = 2$ suffices for the braid group. In fact, Runnels and Seo have independently and recently shown an effective version of Koberda's result \cite{runnels}, \cite{seo}.  For the collections of Dehn twists that arise in the braid group case, it turns out that $N = 17$ works, and we suspect there is a not very large $N$ that works for all Artin groups.

\subsection{Type $D_n$}
We now consider an Artin group $A$ of type $D_n$ with the standard generating set $S=\{s,s', t_1, \dots t_n\}$, as in Figure~\ref{fig:Dn}.
\begin{Theorem}\label{thm:conjecture for Dn}
Generalized Tits Conjecture holds for all spherical Artin groups of type $D_n$.
\end{Theorem}
\begin{proof}
This proof will use the generalized property PP from the last section. 
There are four families of irreducible subsets of $S$ with at least two elements. 
\begin{enumerate}[label=(\arabic*)]
\item If $T=\{s,s', t_1,\dots, t_j\}$ where $j\leq n-2$, then $A_T$ has type $D_{j+2}$.
\item If $T=\{t_i, \dots, t_j\}$ where $1\leq i<j\leq n-2$, then $A_T$ has type $A_{j-i+1}$.
\item If $T=\{s', t_1,\dots, t_j\}$ where $1\leq j\leq n-2$, then $A_T$ has type $A_{j+1}$.
\item If $T=\{s, t_1,\dots, t_j\}$ where $1\leq j\leq n-2$, then $A_T$ has type $A_{j+1}$.
\end{enumerate}

We consider the Perron-Vannier representation $\rho:A \to\Mod(\Sigma)$ (see Section~\ref{sec:reps} and Figure~\ref{fig:Dn}). By Lemma~\ref{lem:generators multitwists} the generator $z_T$ of $\RA$ is sent to a multitwist about the boundary of a subsurface $\Sigma_T$. 

We pick a connected component $\Sigma_+$ of, respectively, $\Sigma-\{t_2, t_4, \dots, t_{n-2}, \ell\}$ when $n$ is even, and of $\Sigma-\{t_2, t_4, \dots, t_{n-3}, \ell, \ell'\}$ when $n$ is odd.
In both cases $\ell$ denotes an arc with both endpoints in the central boundary component of $\Sigma$ that intersects $t_1$ and no other curves in $S$. When $n$ is odd, then $\ell'$ denotes an arc with both endpoints in the unique non-central boundary component of $\Sigma$, that intersects $t_{n-2}$ and no other curves in $S$. 
We can also pick a connected component $\Sigma_\#$ of, respectively, $\Sigma-\{t_1, t_3, \dots, t_{n-3}, k_1, k_2, k_3\}$ when $n$ is even, and of $\Sigma-\{t_1, t_3, \dots, t_{n-2}, k_1, k_2\}$ when $n$ is odd. By $k_1, k_2, k_3$ we denote arc with endpoints in distinct connected components of $\partial \Sigma$ that intersect only respectively $s,s', t_{n-2}$ (the last one only when $n$ is even) among curves of $S$.

Let us now analyze what the multicurves $\partial \Sigma_T$ are for $T$ in the four families of irreducible subsets of $S$. Let $s_0$ denote the central component of $\partial \Sigma$.

\begin{enumerate}
\item For each irreducible set $T = \{s,s', t_1, \dots, t_j\}$, the multicurve $\partial \Sigma_T$ is of the form $\{s_0, s_{j+1}, s_{j+1}'\}$ when $|T|$ is even, and of the form $\{s_0, s_{j}\}$ when $|T|$ is odd. We also set $s_1 = s$ and $s_1' = s'$.
See the left side of Figure~\ref{fig:curves in Dn}. Without loss of generality, we can assume that all the curves $s_j$ are contained in $\Sigma_+$.

\begin{figure}\label{fig:curves in Dn}
\centering
\begin{tikzpicture}[scale=0.8]
\begin{scope}
\draw (1,-2) to[out=105, in=255] (1,2);
\draw (1,2) to[out=285, in=75] (1,-2);
\draw (1,-2) to (8, -2);
\draw (1,2) to (8, 2);

\draw (2, -0) to[out=-30, in=210] (3, 0);
\draw (2.2,-0.1) to[out=40, in=140] (2.8, -0.1);

\draw (4, -0) to[out=-30, in=210] (5, 0);
\draw (4.2,-0.1) to[out=40, in=140] (4.8, -0.1);

\draw (6, -0) to[out=-30, in=210] (7, 0);
\draw (6.2,-0.1) to[out=40, in=140] (6.8, -0.1);

\draw[dotted, thick] (7.5,0) to (7.75,0);

\draw[ForestGreen] (6.5, 0) to[out=70 , in=290] (6.5, 2);
\draw[ForestGreen, dotted] (6.5, 0) to[out=110 , in=250] (6.5, 2);
\node[ForestGreen, draw=none,fill=none] at (7,1.2) {$s_5'$};

\draw[ForestGreen] (6.5, -0.15) to[out=290 , in=70] (6.5, -2);
\draw[ForestGreen, dotted] (6.5, -0.15) to[out=250 , in=110] (6.5, -2);
\node[ForestGreen, draw=none,fill=none] at (7,-1.2) {$s_5$};

\draw[ForestGreen] (4.5, 0) to[out=70 , in=290] (4.5, 2);
\draw[ForestGreen, dotted] (4.5, 0) to[out=110 , in=250] (4.5, 2);
\node[ForestGreen, draw=none,fill=none] at (5,1.2) {$s_3'$};

\draw[ForestGreen] (4.5, -0.15) to[out=290 , in=70] (4.5, -2);
\draw[ForestGreen, dotted] (4.5, -0.15) to[out=250 , in=110] (4.5, -2);
\node[ForestGreen, draw=none,fill=none] at (5,-1.2) {$s_3$};

\draw[ForestGreen] (2.5, 0) to[out=70 , in=290] (2.5, 2);
\draw[ForestGreen, dotted] (2.5, 0) to[out=110 , in=250] (2.5, 2);
\node[ForestGreen, draw=none,fill=none] at (3,1.2) {$s_1'$};

\draw[ForestGreen] (2.5, -0.15) to[out=290 , in=70] (2.5, -2);
\draw[ForestGreen, dotted] (2.5, -0.15) to[out=250 , in=110] (2.5, -2);
\node[ForestGreen, draw=none,fill=none] at (3,-1.2) {$s_1$};

\draw[ForestGreen,dotted] (5.5, -2) to[out=105, in=255] (5.5, 2);
\draw[ForestGreen] (5.5,2) to[out=285, in=75] (5.5, -2);
\node[ForestGreen, draw=none,fill=none] at (6.1,1) {$s_4$};

\draw[ForestGreen,dotted] (3.5, -2) to[out=105, in=255] (3.5, 2);
\draw[ForestGreen] (3.5,2) to[out=285, in=75] (3.5, -2);
\node[ForestGreen, draw=none,fill=none] at (4.1,1) {$s_2$};

\draw[ForestGreen,dotted] (1.5, -2) to[out=105, in=255] (1.5, 2);
\draw[ForestGreen] (1.5,2) to[out=285, in=75] (1.5, -2);
\node[ForestGreen, draw=none,fill=none] at (2.1,1) {$s_0$};
\end{scope}
\begin{scope}[shift={(7.5,0)}]
\draw (1.5,-2) to[out=105, in=255] (1.5,2);
\draw (1.5,2) to[out=285, in=75] (1.5,-2);
\draw (1.5,-2) to (9, -2);
\draw (1.5,2) to (9, 2);

\draw (2, -0) to[out=-30, in=210] (3, 0);
\draw (2.2,-0.1) to[out=40, in=140] (2.8, -0.1);

\draw (4, -0) to[out=-30, in=210] (5, 0);
\draw (4.2,-0.1) to[out=40, in=140] (4.8, -0.1);

\draw (6, -0) to[out=-30, in=210] (7, 0);
\draw (6.2,-0.1) to[out=40, in=140] (6.8, -0.1);

\draw (8, -0) to[out=-30, in=210] (9, 0);
\draw (8.2,-0.1) to[out=40, in=140] (8.8, -0.1);

\draw[dotted, thick] (9.5,0) to (9.75,0);

\draw[purple] (5.5, 0) ellipse (1.8 and 0.5);
\node[purple, draw=none,fill=none] at (5.5,-0.25) {$t_{3:5}$};

\draw[purple] (2.7, -0.03) to[out= 30, in = 90] (7.5,0) to[out =270, in=340] (2.7, -0.13);
\draw[purple, dotted] (2.7, -0.03) to[out= 35, in = 90] (7.6,0) to[out =270, in=335] (2.7, -0.13);
\node[purple, draw=none,fill=none] at (4.05,-0.9) {$t_{2:5}$};

\draw[purple] (2.5, 0.01) to[out= 60, in = 120] (8.5,0.01);
\draw[purple, dotted] (2.5, 0.01) to[out= 65, in = 115] (8.5,0.01);
\draw[purple] (2.5, -0.15) to[out= -60, in = -120] (8.5,-0.15);
\draw[purple, dotted] (2.5, -0.15) to[out= -65, in = -115] (8.5,-0.15);
\node[purple, draw=none,fill=none] at (7.6,1.35) {$t_{2:6}'$};
\node[purple, draw=none,fill=none] at (7.9,-1.35) {$t_{2:6}$};

\end{scope}
\end{tikzpicture}
\caption{Curves $s_0, s_1=2, s_1' = s',\dots s_{2k}, s_{2k+1}, s_{2k+1}'\dots$ and curves $t_{i:j}, t_{i:j}'$ in the surface $\Sigma$ corresponding to the Artin group of type $A_n$ or $D_n$. In the case of $D_n$, the boundary of $\Sigma$ includes the gray curve.}
\label{fig:curves in Dn}
\end{figure}
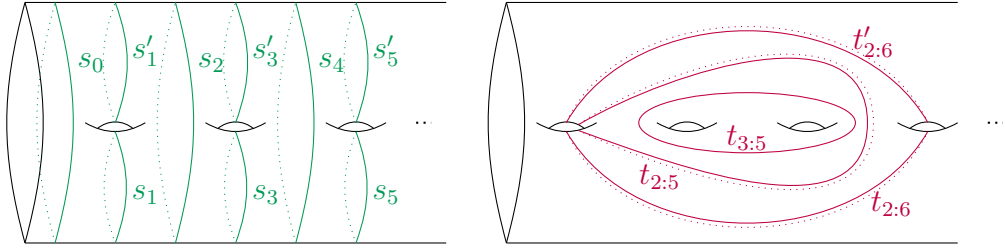
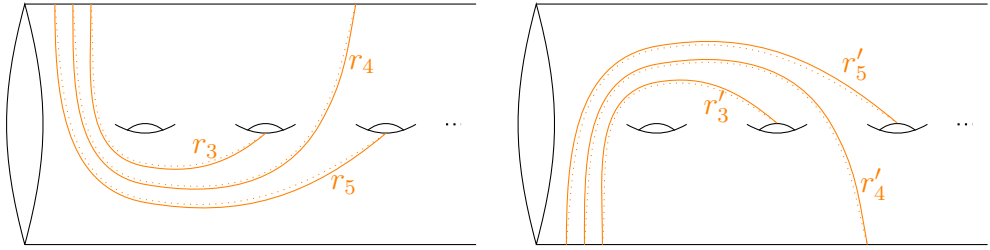
\begin{figure}
\centering
\begin{tikzpicture}[scale=0.8]
\begin{scope}
\draw (0.5,-2) to[out=105, in=255] (0.5,2) to[out=285, in=75] (0.5,-2);
\draw (0.5,-2) to (8, -2);
\draw (0.5,2) to (8, 2);

\draw (2, -0) to[out=-30, in=210] (3, 0);
\draw (2.2,-0.1) to[out=40, in=140] (2.8, -0.1);

\draw (4, -0) to[out=-30, in=210] (5, 0);
\draw (4.2,-0.1) to[out=40, in=140] (4.8, -0.1);

\draw (6, -0) to[out=-30, in=210] (7, 0);
\draw (6.2,-0.1) to[out=40, in=140] (6.8, -0.1);

\draw[dotted, thick] (7.5,0) to (7.75,0);

\draw[orange] (1.6, 2) to[out = 270, in=170] (2.5, -0.7) to[out = -10, in= 220] (4.5, -0.15);
\draw[orange, dotted] (1.6, 2) to[out = 275, in=170] (2.5, -0.65) to[out = -10, in= 215] (4.5, -0.15);
\node[orange, draw=none,fill=none] at (3.5,-0.4) {$r_3$};

\draw[orange] (1.3, 2) to[out = 270, in=170] (2.5, -1) to[out = -10, in= 240] (5.4, 0) to[out=60, in=260] (6,2);
\draw[orange, dotted] (1.3, 2) to[out = 275, in=170] (2.5, -0.95) to[out = -10, in= 240] (5.35, 0) to[out=60, in=255] (6,2);
\node[orange, draw=none,fill=none] at (6.1,1) {$r_4$};

\draw[orange] (1, 2) to[out = 270, in=170] (2.5, -1.3) to[out = -10, in= 220] (6.5, -0.15);
\draw[orange, dotted] (1, 2) to[out = 275, in=170] (2.5, -1.25) to[out = -10, in= 215] (6.5, -0.15);
\node[orange, draw=none,fill=none] at (5.8,-1) {$r_5$};

\end{scope}
\begin{scope}[shift={(8.5,0)}]
\draw (0.5,-2) to[out=105, in=255] (0.5,2) to[out=285, in=75] (0.5,-2);
\draw (0.5,-2) to (8, -2);
\draw (0.5,2) to (8, 2);

\draw (2, -0) to[out=-30, in=210] (3, 0);
\draw (2.2,-0.1) to[out=40, in=140] (2.8, -0.1);

\draw (4, -0) to[out=-30, in=210] (5, 0);
\draw (4.2,-0.1) to[out=40, in=140] (4.8, -0.1);

\draw (6, -0) to[out=-30, in=210] (7, 0);
\draw (6.2,-0.1) to[out=40, in=140] (6.8, -0.1);

\draw[dotted, thick] (7.5,0) to (7.75,0);

\draw[orange] (1.6, -2) to[out = 90, in=-170] (2.5, 0.7) to[out = 10, in= -220] (4.5, 0.02);
\draw[orange, dotted] (1.6, -2) to[out = -275, in=-170] (2.5, 0.65) to[out = 10, in= -215] (4.5, 0.02);
\node[orange, draw=none,fill=none] at (3.5,0.3) {$r_3'$};

\draw[orange] (1.3, -2) to[out = -270, in=-170] (2.5, 1) to[out = 10, in= -240] (5.4, 0) to[out=-60, in=-260] (6,-2);\draw[orange, dotted] (1.3, -2) to[out = -275, in=-170] (2.5, 0.95) to[out = 10, in= -240] (5.35, 0) to[out=-60, in=-255] (6,-2);
\node[orange, draw=none,fill=none] at (6.1,-1) {$r_4'$};

\draw[orange] (1, -2) to[out = -270, in=-170] (2.5, 1.3) to[out = 10, in= -220] (6.5, 0.02);
\draw[orange, dotted] (1, -2) to[out = -275, in=-170] (2.5, 1.25) to[out = 10, in=- 215] (6.5, 0.02);
\node[orange, draw=none,fill=none] at (5.8,1) {$r_5'$};
\end{scope}
\end{tikzpicture}
\caption{Curves $r_2,r_3, \dots$ and $r_2',r_3', \dots$ .}
\label{fig:more curves in Dn}
\end{figure}

\item Now consider an irreducible subset $T=\{t_i,\dots, t_j\}$ where $1\leq i<j\leq n-2$. If $|T|$ is even, let $t_{i:j}$ be the unique boundary curve of the subsurface $\Sigma_T$. 
If $|T|$ is odd, let $t_{i:j}, t_{i:j}'$ be the two boundary curves of the subsurface containing $t_i, \dots, t_j$. 
See the right side of Figure~\ref{fig:curves in Dn}. 
If $i$ is even, we assume that $t_{i:j}$ lies in $\Sigma_+$, and if $i$ is odd, we assume that $t_{i:j}$ is contained in $\Sigma_\#$. 

\item[(3)/(4)]Finally consider an irreducible subset $T = \{s', t_1, \dots t_j\}$. When $j$ is odd, we denote the unique boundary component of $\Sigma_T$ by $r_{j+1}$. When $j$ is even, the boundary of the subsurface $\partial \Sigma_T$ has two connected components, one is $s_{j+1}'$, and the other is denoted by $r_{j+1}$. See Figure~\ref{fig:more curves in Dn}. We define curves $r_2', r_3',\dots $ analogously, so that $\partial \Sigma_T = \{r_{j+1}'\}$ or $\{s_{j+1}, r_{j+1}'\}$ for $T=\{s', t_1, \dots t_j\}$.
\end{enumerate}

Let $\mathcal C = \{s_{j}, s'_{j}\}\cup \{t_j\}\cup \{t_{i:j}, t_{i,j}'\} \cup \{r_j, r_j'\}$ be the collection of all the curves that arise as boundary curves of subsurfaces $\Sigma_T$ for irreducible subsets $T\subseteq S$. By Theorem \ref{t:koberda} the subgroup generated by sufficiently large powers of Dehn twists around curves in $\mathcal C$ is a RAAG, which we denote by $\RA_{\mathcal C}$.
Note that the words $w_\gs$ corresponding to centers of irreducible spherical Artin subgroups, regarded as elements of $\RA_{\mathcal C}$ do not satisfy property PP, see Figure \ref{f:noPP}.

We now show that generalized property PP implies the conjecture for Artin groups of type $D_n$.  
Let $\mathcal C_1 = \mathcal C-\{r_k'\}_{k\geq 2}$, $\mathcal C_2 = \mathcal C- \{r_k\}_{k\geq 2}$, and $\mathcal C_0 = \mathcal C_1\cap \mathcal C_2 = \mathcal C - \{r_k, r_k'\}_{k\geq 2} $. 

Let $L$ be the flag complex on the vertex set $\mathcal C$ defining $\RA_{\mathcal C}$, and let $L_0, L_1, L_2$ be the full subcomplexes of $L$ on the sets $\mathcal C_0,\mathcal C_1, \mathcal C_2$ respectively. Since the curves $r_i$ and $r_j'$ intersect for every $i, j\geq 2$, the complex $L$ decomposes as $L_1\cup_{L_0}L_2$. Therefore we have a splitting of the $\RA_{\mathcal C}$ as
$A_{L_1} \ast_{A_{L_0}} A_{L_2}$ 
where $A_{L_i}$ is a RAAG with nerve $L_i$. Also, the $\RAAG$ $\RA$ associated with $A$ splits as $\RA = \RA_{L_1} \ast_{\RA_{L_0}} \RA_{L_2}$ where 
\begin{itemize}
\item $\RA_{L_1}$ omits all the generators corresponding to centers of the irreducible special subgroups from family (4), i.e.\ $A_{T}$ with $T=\{s, t_1,\dots, t_k\}$ and $k\geq1$,
\item $\RA_{L_2}$ omits generators corresponding to centers of the irreducible special subgroups from family (3), i.e.\ $A_{T}$ with $T=\{s', t_1,\dots, t_k\}$ and $k\geq1$, and
\item $\RA_{L_0}$  omits all the generators corresponding to centers of the irreducible special subgroups from families (3) and (4).
\end{itemize}

We claim that each $L_i$ satisfies Property PP avoiding $L_0$.
We only verify it for $L_1$, the proof for $L_2$ is analogous. For each irreducible subset $T$ (which corresponds to a simplex in $L$ spanned by the boundary curves of its corresponding subsurface $\Sigma_T$) from families (1), (2) and (3), we need to make a choice of a boundary curve of $\Sigma_T$. For a subgroup from family (3) of type $A_{k}$ we choose the curve $r_{k}$. Note that these subsets are exactly the ones corresponding to simplices $\sigma$ of $L$ with nontrivial $w_{\sigma}$ such that $\sigma\notin L_0$. We thus see that $p(\sigma)\notin L_0$ for such simplices $\sigma$. Also, $w_{\tau}$ is trivial for any other simplex $\tau$ containing $p(\sigma)$, because $\Sigma_T$ is the unique subsurface corresponding to an irreducible subset of $S$ whose boundary contains $r_k$. 

By Lemma \ref{l:curves2}, for each irreducible subset $T$ from family (1), there is a unique boundary curve of $\Sigma_T$ which intersects $\partial \Sigma_U$ for $U$ in family (3). Indeed, it is the curve $s_{k-1}$, see Figure~\ref{fig:curves in Dn}.

Finally, for an irreducible subset in family (2), we pick the curve $t_{i:j}$, i.e.\ a curve in $\partial\Sigma_T$ that intersects $r_k$ for every $i\leq k\leq j$, which is unique unless $i,j$ are both odd (in which case either choice works).

For each simplex in $L_1 - L_0$, we have chosen curves not contained in $L_0$. Also, for each simplex in $L_1 - L_0$, we have chosen a simple closed curve that is not a boundary curve of any subsurface corresponding to a different irreducible Artin subgroup. 
Apart from the curve $s_0$ (which is not in the image of $p$), the only curves contained in two multicurves are the $s'_{2k+1}$. Since $p(T) = s_{2k+1}$ for $A_T$ of type $D_{2k+2}$, and $p(T) = r_{2k+1}$ for $A_T$ of type $A_{2k+1}$ from family $(3)$ this guarantees the third condition in Lemma \ref{lem:PP for surfaces} is satisfied. 
By Theorem~\ref{thm:generalized PP}, the conclusion follows.
\end{proof}
\begin{figure}
\centering
\begin{tikzpicture}[scale=0.8]
\draw (0.5,-2) to[out=105, in=255] (0.5,2) to[out=285, in=75] (0.5,-2);
\draw (0.5,-2) to (8, -2);
\draw (0.5,2) to (8, 2);

\draw (2, -0) to[out=-30, in=210] (3, 0);
\draw (2.2,-0.1) to[out=40, in=140] (2.8, -0.1);

\draw (4, -0) to[out=-30, in=210] (5, 0);
\draw (4.2,-0.1) to[out=40, in=140] (4.8, -0.1);

\draw (6, -0) to[out=-30, in=210] (7, 0);
\draw (6.2,-0.1) to[out=40, in=140] (6.8, -0.1);

\draw[dotted, thick] (7.5,0) to (7.75,0);

\draw[black] (2.5, 0) ellipse (0.9 and 0.5);

\draw[black] (4.5, 0) ellipse (0.9 and 0.5);

\draw[black] (6.5, 0) ellipse (0.9 and 0.5);

\draw[black] (3,0) to[out=30, in=150] (4,0);
\draw[black, dotted] (2.9,-0.05) to[out=330, in=210] (4.1,-0.05);

\draw[black] (5,0) to[out=30, in=150] (6,0);
\draw[black, dotted] (4.9,-0.05) to[out=330, in=210] (6.1,-0.05);

\draw[black] (2.5, 0) to[out=70 , in=290] (2.5, 2);
\draw[black, dotted] (2.5, 0) to[out=110 , in=250] (2.5, 2);

\draw[black] (2.5, -0.15) to[out=290 , in=70] (2.5, -2);
\draw[black, dotted] (2.5, -0.15) to[out=250 , in=110] (2.5, -2);
\draw[orange] (1, 2) to[out = 275, in=170] (2.5, -1.25) to[out = -10, in= 215] (6.5, -0.15);
\draw[red] (1, -2) to[out = 85, in=190] (2.5, 1.1) to[out = 10, in= 155] (6.5, 0);

\draw[ForestGreen] (4.5, 0) to[out=70 , in=290] (4.5, 2);
\draw[ForestGreen, dotted] (4.5, 0) to[out=110 , in=250] (4.5, 2);

\draw[ForestGreen] (4.5, -0.15) to[out=290 , in=70] (4.5, -2);
\draw[ForestGreen, dotted] (4.5, -0.15) to[out=250 , in=110] (4.5, -2);

\end{tikzpicture}
\caption{These curves show the failure of property PP for the images of centers of $D_n$. The red and orange curves are boundary components of subsurfaces $\Sigma_{A_5}$ and $\Sigma_{A_5'}$. The green curves are the boundary components of the subsurface $\Sigma_{D_4}$. Since the other boundary component misses $\gamma$ and $\gamma'$ respectively, we are forced to choose these. However, this implies that we cannot correctly choose a boundary component of $\partial \Sigma_{D_4}$.}
\label{f:noPP}
\end{figure}
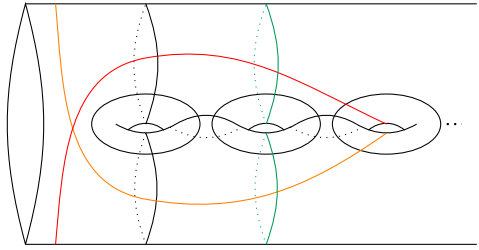

\subsection{Types $E_6$, $E_7$ and $E_8$}\label{ss:e7}

We now show that any homomorphism $A_{E_7} \rightarrow \Mod(\Sigma)$ which sends generators to Dehn twists has nontrivial kernel which intersects the image of $\Phi_m$. Wajnryb has previously shown that there is no injective homomorphism from the Artin groups of type $E_6,E_7, E_8$ to any mapping class group which maps generators to Dehn twists \cite{wajnryb}. It is still open whether these Artin groups admit other faithful representations into mapping class groups.

\begin{figure}
\centering
\begin{tikzpicture}[scale = .85]
\draw (2.5,-2) to[out=180, in=180] (2.5,2);
\draw (2.5,-2) to (8, -2);
\draw (2.5,2) to (8, 2);
\draw[dotted] (8, -2) to[out =105, in=255] (8, -0.5);
\draw (8, -2) to[out=75, in=285] (8, -0.5);
\draw[dotted] (8, 2) to[out =255, in=105] (8, 0.5);
\draw(8, 2) to[out=285, in=75] (8, 0.5);
\draw (8,0.5) to[out=230, in=130] (8,-0.5);

\draw (2, -0) to[out=-30, in=210] (3, 0);
\draw (2.2,-0.1) to[out=40, in=140] (2.8, -0.1);

\draw (4, -0) to[out=-30, in=210] (5, 0);
\draw (4.2,-0.1) to[out=40, in=140] (4.8, -0.1);

\draw (6, -0) to[out=-30, in=210] (7, 0);
\draw (6.2,-0.1) to[out=40, in=140] (6.8, -0.1);

\draw[black] (2.5, 0) ellipse (0.9 and 0.5);

\draw[black] (4.5, 0) ellipse (0.9 and 0.5);

\draw[black] (6.5, 0) ellipse (0.9 and 0.5);

\draw[black] (3,0) to[out=30, in=150] (4,0);
\draw[black, dotted] (2.9,-0.05) to[out=330, in=210] (4.1,-0.05);

\draw[black] (5,0) to[out=30, in=150] (6,0);
\draw[black, dotted] (4.9,-0.05) to[out=330, in=210] (6.1,-0.05);
\node[right] at (4.6,1.2) {$s$};

\draw[black] (7,0) to[out=30, in=160] (7.82,0.09);
\draw[black, dotted] (6.9,-0.05) to[out=330, in=210] (7.82,-0.14);

\draw[black] (4.5, 0) to[out=70 , in=290] (4.5, 2);
\draw[black, dotted] (4.5, 0) to[out=110 , in=250] (4.5, 2);

\draw[ForestGreen] (6.5, -0.15) to[out=290 , in=70] (6.5, -2);
\draw[ForestGreen, dotted] (6.5, -0.15) to[out=250 , in=110] (6.5, -2);

\draw[ForestGreen] (6.5, 0) to[out=70 , in=290] (6.5, 2);
\draw[ForestGreen, dotted] (6.5, 0) to[out=110 , in=250] (6.5, 2);

\draw[red] (2.5, 0.01) to[out=30, in=160] (7.9, .35);
\draw[red, dotted] (2.5, 0.01) to[out=25 , in=165] (7.9, 0.35);

\draw[red] (2.5, -0.15) to[out=-30, in=-160] (7.9, -.35);
\draw[red, dotted] (2.5, -0.15) to[out=-25 , in=-165] (7.9, -.35);

\draw[NavyBlue] (1.32,0) to[out=30, in=150] (2,0);
\draw[NavyBlue, dotted] (1.32,-0.05) to[out=330, in=210] (2.1,-0.05);

\draw[NavyBlue] (2.4, -0.15) to[out=-45, in=-150] (7.95, -.45);
\draw[NavyBlue, dotted] (2.4, -0.15) to[out=-35 , in=-155] (7.95, -.45);

\begin{scope}[shift={(9,0)}, scale = 1.2]
\node[circle, draw, fill, inner sep = 0pt,minimum width = 3pt] (a) at (0,0) {};
\node[circle, draw, fill, inner sep = 0pt,minimum width = 3pt] (b) at (1,0) {};
\node[circle, draw, fill, inner sep = 0pt,minimum width = 3pt] (c) at (2,0) {};
\node[circle, draw, fill, inner sep = 0pt,minimum width = 3pt] (d) at (3,0) {};
\node[circle, draw, fill, inner sep = 0pt,minimum width = 3pt] (e) at (4,0) {};
\node[circle, draw, fill, inner sep = 0pt,minimum width = 3pt] (g) at (5,0) {};
\node[circle, draw, fill, inner sep = 0pt,minimum width = 3pt] (f) at (2,1) {};
\draw (a) -- (b) -- (c) -- (f);
\draw (c) -- (d);
\draw (d) -- (e) -- (g);

\draw[red] (3, 0) ellipse (2.5 and 0.3);
\draw[NavyBlue] (.4, 0) to[out=90, in=225] (1.5, 1);
\draw[NavyBlue] (1.5, 1) to[out=45, in=135] (2.5, 1);
\draw[NavyBlue] (2.5, 1) to[out=-45, in=135] (6, 0);
\draw[NavyBlue] (.4,0) arc(-180:0:2.8cm and .4cm);

\draw[ForestGreen] (2,1.1) to[out = 0, in = 135] (3.2,0);

\draw[ForestGreen] (2,1.1) to[out = 180, in = 0] (1,.3);
\draw[ForestGreen] (1,.3) to[out = 180, in = 45] (-.2,0);
\draw[ForestGreen] (-.2,0) arc(-180:0:1.7cm and .3cm);
\node[right,ForestGreen] at (0,0.5) {$T$};
\node[right,NavyBlue] at (3.5,0.75) {$V$};
\node[below,red] at (3,-0.4) {$U$};

\end{scope}
\end{tikzpicture}

\caption{High powers of Dehn twists around the multicurves above do not generate the obvious $\RAAG$ subgroup. The figure on the left shows the corresponding centers of irreducible Artin subgroups. To see this, start with the Dehn twist around $s$, then conjugate by the red multitwist, then conjugate by the green multitwist. This element commutes with the blue multitwist.
}
\label{f:e7}
\end{figure}
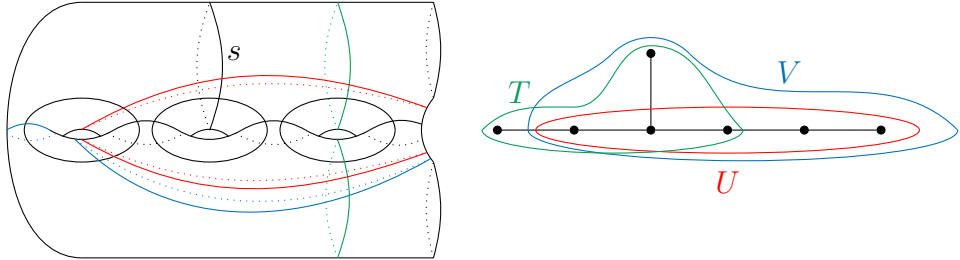

The Dehn twists in Figure \ref{f:e7} show an element in the kernel of the Perron-Vannier representation $\rho: A \rightarrow \Mod(\Sigma)$ where $A$ is of type $E_7$. A specific element in the kernel of $\rho$ is the commutator 
$$[\Delta^{2N}_{T} \Delta^{2N}_{U}s^{2N} \Delta^{-2N}_{U} \Delta^{-2N}_{T}  , \Delta^{2N}_{V}]$$
where $T=\{t_1,t_2, t_3, t_4,s\}$, $U=\{t_2, t_3, t_4, t_5, t_6\}$ and $V=\{t_2, t_3, t_4, t_5, t_6,s\}$. The nerve of the subgroup of $\RA$ on the generators corresponding to $T, \{s\}, V, U$ is a path on $4$ vertices.
It is easy to verify that the corresponding commutator in $\RA$ is nontrivial.
The commutator above is also nontrivial in $A$. This can be verified by computing the Deligne's normal form for its positive representatives. We have performed the computation in GAP 3 with the package CHEVIE \cite{gap3} \cite{chevie}. Note that the above does not imply that the Generalized Tits Conjecture does not hold for Artin group of type $E_7$ (or $E_8$). It only shows that our strategy, using the Perron-Vannier representation, does not work in that case. 

\begin{remark}
The multicurves produced by the Perron-Vannier representation of the Artin group of type $E_6$ do not satisfy generalized property PP, but we cannot find a word in $\RA$ as above in the kernel of the representation. 
\end{remark}

\section{Generalized Tits Conjecture for other spherical Artin groups}\label{s:folding}
In this section, we show the Generalized Tits Conjecture holds for all spherical Artin groups which are not small type. The main tool is a folding trick due to Crisp (see \cite{crisp} or \cite[Section 6]{cp}), which embeds any of these Artin groups into small type spherical Artin groups. 

\subsection{Folding homomorphisms}
In this section it is more convenient to work with Coxeter diagrams rather than nerves.
So, suppose that $A_\gG$ is an Artin group with a connected Coxeter diagram $\gG$ that has no $\infty$-labels. 
Crisp and Paris define a \emph{folding homomorphism} $\Psi: A_\gG \rightarrow A_{\widetilde{\gG}}$ where $A_{\widetilde{\gG}}$ is small-type. Here is the construction: let $N=\text{lcm}\{m_{st} - 1| s \ne t \in S\}$. For each vertex $s\in S$, let $I(s)$ be a set with $N$ elements. For $m\geq3$, let $\gG(m)$ be the Coxeter diagram for $A_{m-1} \times A_{m-1}$, i.e.\ $\gG(m)$ is a disjoint union of two copies of the Coxeter graph of type $A_{m-1}$. See Figure~\ref{fig:Gamma(m)}. 

Let $\widetilde \gG$ be a Coxeter diagram so that:

\begin{itemize}
\item The vertex set of $\widetilde \gG$ is the disjoint union of the sets $I(s)$. 
\item If there is no edge in $\gG$ between $s$ and $t$, there is no edge between vertices in $I(s)$ and the vertices of $I(t)$. In particular, there are no edges between any two vertices in $I(s)$.
\item If $m_{st} \ge 3$, then the subgraph of $\widetilde \gG$ spanned by $I(s) \cup I(t)$ is isomorphic to $\frac{N}{m_{st}-1}$ copies of $\gG(m_{st})$. 
\end{itemize}

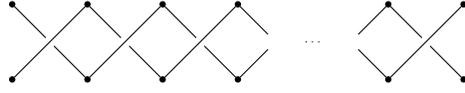
\begin{figure}
\centering
\begin{tikzpicture}
\node[circle, draw, fill, inner sep = 0pt,minimum width = 2pt] (a0) at (0,0) {};
\node[circle, draw, fill, inner sep = 0pt,minimum width = 2pt] (a1) at (1,0) {};
\node[circle, draw, fill, inner sep = 0pt,minimum width = 2pt] (a2) at (2,0) {};
\node[circle, draw, fill, inner sep = 0pt,minimum width = 2pt] (a3) at (3,0) {};
\node[circle, draw, fill, inner sep = 0pt,minimum width = 2pt] (a5) at (5,0) {};
\node[circle, draw, fill, inner sep = 0pt,minimum width = 2pt] (a6) at (6,0) {};
\node[circle, draw, fill, inner sep = 0pt,minimum width = 2pt] (b0) at (0,1) {};
\node[circle, draw, fill, inner sep = 0pt,minimum width = 2pt] (b1) at (1,1) {};
\node[circle, draw, fill, inner sep = 0pt,minimum width = 2pt] (b2) at (2,1) {};
\node[circle, draw, fill, inner sep = 0pt,minimum width = 2pt] (b3) at (3,1) {};
\node[circle, draw, fill, inner sep = 0pt,minimum width = 2pt] (b5) at (5,1) {};
\node[circle, draw, fill, inner sep = 0pt,minimum width = 2pt] (b6) at (6,1) {};
\draw (b0) -- (a1);
\node[circle, draw, fill, inner sep=0pt, minimum width = 4pt, white] at (0.5,0.5) {};
\draw (a0) -- (b1) -- (a2);
\node[circle, draw, fill, inner sep=0pt, minimum width = 4pt, white] at (1.5,0.5) {};
\draw (a1) -- (b2) -- (a3) -- (3.4, 0.4);
\node[circle, draw, fill, inner sep=0pt, minimum width = 4pt, white] at (2.5,0.5) {};
\draw (a2) -- (b3) -- (3.4, 0.6);
\draw[dotted] (3.9, 0.5) -- (4.1, 0.5);
\draw (4.6, 0.6) -- (b5) -- (a6);
\node[circle, draw, fill, inner sep=0pt, minimum width = 4pt, white] at (5.5,0.5) {};
\draw (4.6, 0.4) --  (a5) -- (b6);
\end{tikzpicture}
\caption{The graph $\Gamma(m)$ on $2(m-1)$ vertices.}\label{fig:Gamma(m)}
\end{figure}

It is easy to see that such a $\widetilde \gG$ can always be constructed, though it will not be unique. There is a map of graphs $\widetilde \gG\to \gG$ sending every vertex in $I(s)$ to $s$.
Crisp and Paris show that the map $\Psi: A_{\gG} \rightarrow A_{\widetilde \gG}$ which sends a generator $s$ to the product $\prod_{s_i \in I(s)} s_i$ extends to a homomorphism \cite[Prop 13]{cp}. 
Crisp showed that this homomorphism is injective when restricted to the Artin monoid \cite{crisp}, and it follows that it is injective when $A_\gG$ is spherical. 
 It is still open whether $\Psi$ is injective in general.

We now verify that $\Psi$ has some additional properties.  For each $T \subseteq S$, let $\widetilde T$ denote the preimage of $T$ under $\widetilde \gG\to\gG$, i.e.\ $\widetilde T$ contains the vertices in $I(s)$ for each $s \in T$.

\begin{lemma}\label{lemma:folding homomorphism} Let $A_{\gG}$ be an Artin group with connected Coxeter diagram $\gG$ and no $\infty$-labels. Let $\widetilde \gG$ be as above, and let $\Psi:A_{\gG} \rightarrow A_{\widetilde \gG}$ be the above homomorphism. Then $\Psi$ satisfies the following:

\begin{itemize}
\item For every spherical Artin subgroup $A_T \subseteq A_{\gG}$, $A_{\widetilde{T}}$ is spherical. 
\item If $A_{T}$ is an irreducible spherical subgroup of $A_{\gG}$ , then $$\Psi(\Delta_T^2) = \prod_{T_i \subseteq \widetilde T} \Delta_{T_i}^2$$ where the $\{T_i\}$ are the irreducible components of $\widetilde T$.
\end{itemize}
\end{lemma}

\begin{proof}
If $A_T$ is small type, then the Coxeter diagram of $A_{\widetilde T}$ is a disjoint union of finitely many copies of Coxeter diagram for $A_T$, and hence $A_{\widetilde T}$ is a direct product of copies of $A_{T}$.
In this case, both statements follow easily. 
Suppose $\gs$ is not small type. 
If $A_T$ is irreducible, the Coxeter subdiagram $\gG_T$ for $A_T$ is a line, with exactly one edge $e$ labeled with number greater than $3$. 
Let $\widetilde \gG_T^i$ be a connected component of the Coxeter diagram for $\widetilde \gG_T$. Then $\widetilde \gG_{\widetilde T^i}$ is obtained from $\gG_T$ by replacing $e$ with a copy of $\gG(m)$ and then attaching copies of $T - e$. 
In each case we check directly that this produces a spherical Artin group.

For the second statement, by the definition of $\Psi$, a Coxeter element for $W_\gG$ maps to a Coxeter element for $W_{\widetilde \gG}$.  The Coxeter number is preserved in the above identifications (see Lemma \ref{l:cases} below), in that if $\widetilde{\gG}_{\widetilde T^i}$ is connected component of $\widetilde \gG_{T}$, then the Coxeter number of $A_{T}$ is the same as the number for $A_{\widetilde T^i})$. This immediately implies the second statement by Lemma \ref{l:bs}.
\end{proof}

We now record the following specific cases of Lemma~\ref{lemma:folding homomorphism} (these were previously tabulated in \cite{crisp}). 

\begin{lemma}\label{l:cases}
Let $A_\gG$ be a spherical Artin group as above, and let $A_{\widetilde \gG_c}$ be an Artin group in the image of $\Psi$
where $\widetilde \gG_c$ is a connected component of $\widetilde{\gG}$.
Let $h$ be the Coxeter number of the Coxeter system of $A_\gG$. Then the Coxeter number of the Coxeter system of $A_{\widetilde \gG_c}$ is $h$. Moreover:
\begin{enumerate}
\item If $A_\gG= B_n$, then $A_{\widetilde \gG_c}$ has type $D_{n+1}$ or $A_{2n-1}$, and $h$ = 2n.
\item If $A_{\gG} = I_2(p)$, then $A_{\widetilde \gG_c}$ has type $A_{p-1}$ and $h = p$.
\item If $A_{\gG} = H_3$, then $A_{\widetilde \gG_c}$ has type $D_6$ and $h = 10$.
\item If $A_{\gG} = H_4$, then $A_{\widetilde \gG_c}$ has type $E_8$ and $h = 30$. 
\item If $A_{\gG} = F_4$, then $A_{\widetilde \gG_c}$ has type $E_6$ and $h = 12$. 
\end{enumerate}
\end{lemma}

Let $A_{\Gamma}$ be an Artin group  with no $\infty$-labels and connected $\gG$.
The $\RAAG$ with presentation~\ref{eq:raag} associated with $A_{\Gamma}$ is denoted by $\RA_{\Gamma}$, 
and the $\RAAG$ associated with $A_{\widetilde \Gamma}$ is denoted by $\RA_{\widetilde\Gamma}$. 
We denote the homomorphism $\RA_{\widetilde\Gamma}\to A_{\widetilde\Gamma}$ from the statement of the Generalized Tits Conjecture by $\widetilde\Phi_N$.
 
We consider the composition $\Psi \circ \Phi_N: RA_{\Gamma}  \rightarrow A_{\widetilde \gG}$. Our aim is to verify the Generalized Tits Conjecture for $A_{\Gamma}$, i.e.\ to show that $\Phi_N$ is injective for some $N$. 
Let $F: \RA_{\Gamma} \rightarrow \RA_{\widetilde\Gamma} $. If $T$ is an irreducible, special subset, then $\widetilde T$ will generally not be irreducible. 
For a generator $z_T$ of $\RA_{\Gamma}$ its image $F(z_T)$ is the product $\prod_{\widetilde T_i \subseteq \widetilde T} z_{\widetilde T_i}$
of generators in $\RA_{\widetilde \Gamma}$ corresponding to irreducible subsets $\widetilde T_i\subseteq \widetilde T$. 
For each $N$, we have the following equality $\widetilde \Phi_N \circ F = \Psi \circ \Phi_N$.

Crisp and Paris show the following lemma for the $\RAAG$ generated by $\{s^2:s\in S\}$. The same proof works for our $\RAAG$s. For the benefit of the reader, we provide the proof. 

\begin{lemma}\label{l:raaginj}
The homomorphism $F$ is injective. Therefore, if $\widetilde \Phi_N$ is injective, then $\Phi_N$ is injective.\end{lemma}

\begin{proof}
 Note that if $T,U\subseteq S$ are distinct irreducible subsets, then $z_{\widetilde T^i} \ne z_{\widetilde U^j}$ for all $i$ and $j$. 
 Furthermore, if $z_T$ and $z_U$ do not commute in $\RA_{\Gamma}$, then for every $z_{\widetilde T^i}$ there exists a $z_{\widetilde U^i}$ that does not commute with $s_{\widetilde T^i}$ in $\RA_{\widetilde \Gamma}$.
To see this, note that if two spherical subsets $T$ and $U$ do not commute, there is $t \in T - U$ and $u \in U - T$ with $m_{tu} \ne 2$. Any component $\widetilde T^i$ of $\widetilde T$ contains a vertex $\widetilde t$ of $I(t)$. There exists a vertex $\widetilde u \in I(u)$ with $m_{\widetilde t \widetilde u} \ne 2$. Therefore, the component of $\widetilde U$ containing $\widetilde u$ does not commute with $\widetilde T_i$. 
By the normal form for $\RAAG$s this implies $F$ is injective, as it takes a reduced word in  $\RA_{\Gamma}$ to a reduced word in $\RA_{\widetilde \Gamma}$. 
\end{proof}

\subsection{Generalized Tits Conjecture for $B_n, F_4, H_3, H_4, I_2(p)$.}

We now finish the remaining spherical cases. By Lemma~\ref{l:cases} Artin groups of types $B_n$, $H_3$, and $I_2(p)$ embed via the folding homomorphisms in Artin groups of types $A_m$ and $D_m$, for which we already know that the conjecture holds by Proposition~\ref{prop:conjecture for An} and Theorem~\ref{thm:conjecture for Dn}. By Lemma~\ref{lemma:folding homomorphism} and Lemma~\ref{l:raaginj} we get the following.

\begin{corollary}
The Generalized Tits Conjecture holds for spherical Artin group of type $B_n$, $H_3$, and $I_2(p)$.
\end{corollary}

In the case of $F_4$ and $H_4$, we do not know the conjecture for $E_6$ and $E_8$ respectively. 
However, in each of those cases, we can still show that the generalized Tits conjecture holds by considering 
$$\RA_{\Gamma}\to A_{\Gamma}\hookrightarrow A_{\widetilde \Gamma}\to \prod_{\Gamma^c} \Mod(\Sigma_{\Gamma^c})$$
 where $A_{\Gamma}$ is of type $F_4$ (respectively $H_4$), $\RA_{\Gamma}$ its associated $\RAAG$, and each component $\gG^c$ of $\widetilde \gG$ is of type $E_6$ (respectively $E_8$) with its Perron-Vannier representation in a mapping class group (see Section~\ref{sec:reps}). In both cases, we will only concentrate on one component of $\widetilde \gG$. An identical argument will work for all components, and so we get a faithful representation of $RA_\gG$ into a direct product of RAAG's.

\begin{Theorem}
The Generalized Tits Conjecture holds for the spherical Artin group of type $F_4$.
\end{Theorem}
\begin{proof} Let $S=\{s, t,u,v\}$ be the standard generators of $A_{\gG}$ of type $F_4$ where $m_{st} = m_{uv} = 3$, $m_{tu} = 4$. Consider the homomorphism $\Psi_i:A_{\gG} \to A_{\widetilde \gG_i}$, where $A_{\widetilde \gG}$ has type $E_6$, and where $\Psi_i$ is the composition of the folding homomorphism $\Psi$ with the projection $A_{\widetilde\gG}\to A_{\widetilde \gG_i}$ where $\widetilde \gG_i$ is a connected component of $\widetilde \gG$. See Figure~\ref{fig:F4}.
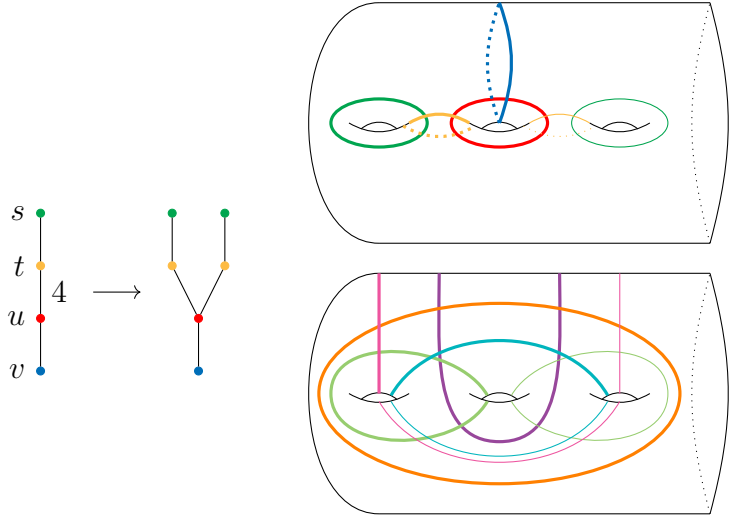
\begin{figure}
\centering
\begin{tikzpicture}
\begin{scope}[shift={(0,-1.5)}, scale = 0.7]
\node[circle, draw, fill, inner sep = 0pt,minimum width = 3pt, NavyBlue, label=left:$v$] (a0) at (0,0) {};
\node[circle, draw, fill, inner sep = 0pt,minimum width = 3pt, red,label=left:$u$] (a1) at (0,1) {};
\node[circle, draw, fill, inner sep = 0pt,minimum width = 3pt, Dandelion,label=left:$t$] (a2) at (0,2) {};
\node[circle, draw, fill, inner sep = 0pt,minimum width = 3pt,Green,label=left:$s$] (a3) at (0,3) {};
\draw (a0) -- (a1) -- (a2) node[draw=none,fill=none,midway,right] {$4$};
\draw (a2) -- (a3);

\draw[->] (1, 1.5) -- (1.9, 1.5);
\node[circle, draw, fill, inner sep = 0pt,minimum width = 3pt, NavyBlue] (b0) at (3,0) {};
\node[circle, draw, fill, inner sep = 0pt,minimum width = 3pt, red] (b1) at (3,1) {};
\node[circle, draw, fill, inner sep = 0pt,minimum width = 3pt,Dandelion] (b2) at (2.5,2) {};
\node[circle, draw, fill, inner sep = 0pt,minimum width = 3pt,Dandelion] (b2') at (3.5,2) {};
\node[circle, draw, fill, inner sep = 0pt,minimum width = 3pt,Green] (b3) at (2.5,3) {};
\node[circle, draw, fill, inner sep = 0pt,minimum width = 3pt,Green] (b3') at (3.5,3) {};
\draw (b0) -- (b1) -- (b2) -- (b3);
\draw (b1) -- (b2') -- (b3');
\end{scope}
\begin{scope}[shift={(2.5,1.8)}, scale=0.8]
\draw (2.5,-2) to[out=180, in=180] (2.5,2);
\draw (2.5,-2) to (8, -2);
\draw (2.5,2) to (8, 2);
\draw[dotted] (8,-2) to[out=105, in=255]  (8, 2);
\draw (8,2) to[out=285, in=75] (8, -2);
\draw (2, -0) to[out=-30, in=210] (3, 0);
\draw (2.2,-0.1) to[out=40, in=140] (2.8, -0.1);

\draw (4, -0) to[out=-30, in=210] (5, 0);
\draw (4.2,-0.1) to[out=40, in=140] (4.8, -0.1);

\draw (6, -0) to[out=-30, in=210] (7, 0);
\draw (6.2,-0.1) to[out=40, in=140] (6.8, -0.1);

\draw[Green, very thick] (2.5, 0) ellipse (0.8 and 0.4);

\draw[red, very thick] (4.5, 0) ellipse (0.8 and 0.4);

\draw[Green] (6.5, 0) ellipse (0.8 and 0.4);

\draw[Dandelion, very thick] (3,0) to[out=30, in=150] (4,0);
\draw[Dandelion, dotted, very thick] (2.9,-0.05) to[out=330, in=210] (4.1,-0.05);

\draw[Dandelion] (5,0) to[out=30, in=150] (6,0);
\draw[Dandelion, dotted] (4.9,-0.05) to[out=330, in=210] (6.1,-0.05);

\draw[NavyBlue, very thick] (4.5, 0) to[out=70 , in=290] (4.5, 2);
\draw[NavyBlue, dotted, very thick] (4.5, 0) to[out=110 , in=250] (4.5, 2);

\end{scope}

\begin{scope}[shift={(2.5,-1.8)}, scale=0.8]
\draw (2.5,-2) to[out=180, in=180] (2.5,2);
\draw (2.5,-2) to (8, -2);
\draw (2.5,2) to (8, 2);
\draw[dotted] (8,-2) to[out=105, in=255]  (8, 2);
\draw (8,2) to[out=285, in=75] (8, -2);
\draw (2, -0) to[out=-30, in=210] (3, 0);
\draw (2.2,-0.1) to[out=40, in=140] (2.8, -0.1);

\draw (4, -0) to[out=-30, in=210] (5, 0);
\draw (4.2,-0.1) to[out=40, in=140] (4.8, -0.1);

\draw (6, -0) to[out=-30, in=210] (7, 0);
\draw (6.2,-0.1) to[out=40, in=140] (6.8, -0.1);

\draw[Purple, very thick] (3.5, 2) to[out=270 , in=180] (4.5, -0.8) to[out=0 , in=270] (5.5, 2);

\draw[YellowGreen, very thick] (4.3, -0.13) to[out=240 , in=270]  (1.7, 0) to[out=90 , in=120]  (4.3, -0.03) ;
\draw[YellowGreen] (4.7, -0.13) to[out=300 , in=270]  (7.3, 0) to[out=90 , in=60]  (4.7, -0.03) ;

\draw[Aquamarine] (2.7, -0.13) to[out=300 , in=240]  (6.3, -0.13);
\draw[Aquamarine, very thick] (2.7, -0.03) to[out=60 , in=120]  (6.3, -0.03);

\draw[orange, very thick] (4.5,0) ellipse (3 and 1.5);

\draw[Rhodamine] (2.5, -0.13) to[out=300 , in=240]  (6.5, -0.13);
\draw[Rhodamine, very thick] (2.5, 0) to
(2.5, 2);
\draw[Rhodamine] (6.5, 0) to
(6.5, 2);

\end{scope}
\end{tikzpicture}
\caption{The restriction of the folding homomorphism for $A_{\Gamma}$ of type $F_4$ to one connected component. The generators of $A_{\Gamma}$ are mapped to the Dehn twists and multitwists in the left surface. The right surface has all the multicurves that arise as subsurfaces corresponding to irreducible special subgroups of $A_{\Gamma}$. Thick curves represent an example of a choice of curves satisfying Property PP.}
\label{fig:F4}
\end{figure}
We also consider the Perron-Vannier representation $\rho:A_{\widetilde \Gamma_i}\to \Mod(\Sigma)$, as discussed in Section~\ref{sec:reps}. The images of the elements $\Delta_{T}^4$ for irreducible subsets $T\subseteq S$ are powers of Dehn twists around curves in $\Sigma$, pictured in Figure~\ref{fig:F4}. Let $\mathcal C$ denote the collection of all these curves. 
Again, by Theorem~\ref{t:koberda} sufficiently large powers $k$ of Dehn twists around single curves in $\mathcal C$ generate a RAAG. 
By Lemma~\ref{lem:PP for surfaces} to show that $A$ satisfies the Generalized Tits Conjecture, we need to make a choice $p(T)$ of a curve in $\partial \Sigma_T$ for each of the irreducible subsets $T$ of $S$. For $T=\{u\}$ or $\{v\}$, the choice is unique. It remains to make choices for $\{s\},\{ t\}, \{s,t\}, \{t,u\}, \{s,t,u\}, \{t,u,v\}$. There is a unique curve in $\Sigma_{t,u}$ that intersects the curve $p(\{v\})$, we choose that curve for $p(\{t,u\})$. Note that that curve intersects both boundary components of $\Sigma_{\{s\}}$, both boundary components of $\Sigma_{\{s,t\}}$ and the unique boundary component of $\Sigma_{\{u,v\}}$.
For each choice of a curve in $\partial \Sigma_{\{s\}}$, there is a unique intersecting curve in $\partial \Sigma_{\{t\}}$, and there is a unique curve in $\partial \Sigma_{\{s,t\}}$ that is the boundary of a subsurface containing those two curves. We make such a consistent choice of $p(\{s\}), p(\{t\}), p(\{s,t\})$. Then there is a unique choice of $p(\{t,u,v\})$ that intersects $p(\{s\})$ and both components of $\Sigma_{\{s,t,u\}}$. Finally, any choice of $p(\{s,t,u\})$ works. 
The only multicurves which share a curve correspond to the collections $\{t,u\}$ and $\{t,u,v\}$, and we never choose this curve (the thin pink and teal curve in Figure \ref{fig:F4}), therefore the third condition of property PP is satisfied. 
\end{proof}

\begin{Theorem}
The Generalized Tits Conjecture holds for the spherical Artin group of type $H_4$.
\end{Theorem}

\begin{proof}
Let $S=\{s, t,u,v\}$ be the standard generators of $A_{\Gamma}$ of type $H_4$ where $m_{st} = 5$ and $m_{tu} = m_{uv} = 3$. Consider the restriction of a folding homomorphism to a single component of Coxeter diagram. We get a homomorphism $\Phi:A_{H_4} \to A_{E_8}$ as pictured in the left of Figure~\ref{fig:H4}. We also consider the Perron-Vannier representation $A_{E_6}\to \Mod(\Sigma)$, as discussed in Section~\ref{sec:reps}. The images of the squares of the original generators, and of the elements $\Delta_{T}^2$ for other irreducible subsets $T\subseteq S$, are Dehn twists around curves in two copies of $\Sigma$ in Figure~\ref{fig:F4}. Let $\mathcal C$ denote the collection of all these curves. As in previously considered cases, by Theorem~\ref{t:koberda} sufficiently large powers $k$ of Dehn twists around single curves in $\mathcal C$ generate a RAAG. We need to verify that Property PP is satisfied for $\{\Phi(\Delta_T^{2k})\}$. Specifically, we need to make choices for the following irreducible subsets of $S$: $\{s\},\{ t\}, \{s,t\}, \{t,u\}, \{s,t,u\}, \{t,u,v\}$. There is a unique curve in $\Sigma_{t,u}$ that intersects the curve corresponding to $v$, we choose that curve for $p(\{t,u\})$. Note that that curve intersects both boundary components of $\Sigma_{\{s\}}$, both boundary components of $\Sigma_{\{s,t\}}$ and the unique boundary component of $\Sigma_{\{u,v\}}$.
For each choice of a curve in $\partial \Sigma_{\{s\}}$, there is a unique intersecting curve in $\partial \Sigma_{\{t\}}$, and there is a unique curve in $\partial \Sigma_{\{s,t\}}$ that is the composition of those two curves. We make such a consistent choice of $p(\{s\}), p(\{t\}), p(\{s,t\})$. That also forces a choice of $p(\{t,u,v\})$ and $p(\{s,t,u\})$. Since no multicurves share a curve, the third condition of property PP is trivially satisfied. 
\end{proof}

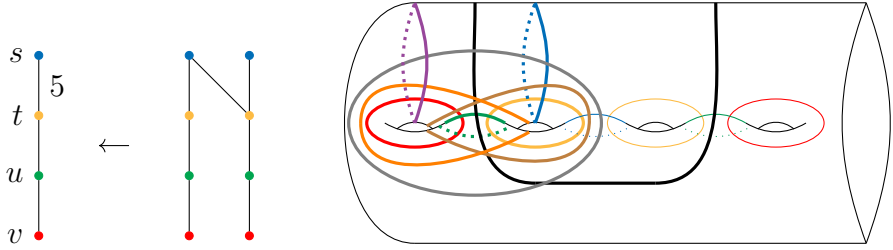
\begin{figure}
\centering
\begin{tikzpicture}
\begin{scope}[shift={(0,-1.5)}, scale = 0.8]
\node[circle, draw, fill, inner sep = 0pt,minimum width = 3pt, red, label=left:$v$] (a0) at (0,0) {};
\node[circle, draw, fill, inner sep = 0pt,minimum width = 3pt, Green,label=left:$u$] (a1) at (0,1) {};
\node[circle, draw, fill, inner sep = 0pt,minimum width = 3pt, Dandelion,label=left:$t$] (a2) at (0,2) {};
\node[circle, draw, fill, inner sep = 0pt,minimum width = 3pt,NavyBlue,label=left:$s$] (a3) at (0,3) {};
\draw (a0) -- (a1) -- (a2);
\draw (a2) -- (a3) node[draw=none,fill=none,midway,right] {$5$};

\draw[<-] (1, 1.5) -- (1.5, 1.5);
\node[circle, draw, fill, inner sep = 0pt,minimum width = 3pt, red] (b0) at (2.5,0) {};
\node[circle, draw, fill, inner sep = 0pt,minimum width = 3pt, red] (b0') at (3.5,0) {};
\node[circle, draw, fill, inner sep = 0pt,minimum width = 3pt, Green] (b1) at (2.5,1) {};
\node[circle, draw, fill, inner sep = 0pt,minimum width = 3pt, Green] (b1') at (3.5,1) {};
\node[circle, draw, fill, inner sep = 0pt,minimum width = 3pt,Dandelion] (b2) at (2.5,2) {};
\node[circle, draw, fill, inner sep = 0pt,minimum width = 3pt,Dandelion] (b2') at (3.5,2) {};
\node[circle, draw, fill, inner sep = 0pt,minimum width = 3pt,NavyBlue] (b3) at (2.5,3) {};
\node[circle, draw, fill, inner sep = 0pt,minimum width = 3pt,NavyBlue] (b3') at (3.5,3) {};
\draw (b0) -- (b1) -- (b2) -- (b3) -- (b2') -- (b3');
\draw (b0') -- (b1') -- (b2');
\end{scope}

\begin{scope}[shift={(3,0)}, scale = .8]

\draw (2.5,-2) to[out=180, in=180] (2.5,2);
\draw (2.5,-2) to (10, -2);
\draw (2.5,2) to (10, 2);
\draw (10,-2) to[out = 110, in = 250] (10, 2);
\draw (10,-2) to[out = 70, in = -70] (10, 2);
\draw (2, -0) to[out=-30, in=210] (3, 0);
\draw (2.2,-0.1) to[out=40, in=140] (2.8, -0.1);

\draw (4, -0) to[out=-30, in=210] (5, 0);
\draw (4.2,-0.1) to[out=40, in=140] (4.8, -0.1);

\draw (6, -0) to[out=-30, in=210] (7, 0);
\draw (6.2,-0.1) to[out=40, in=140] (6.8, -0.1);

\draw (8, -0) to[out=-30, in=210] (9, 0);
\draw (8.2,-0.1) to[out=40, in=140] (8.8, -0.1);

\draw[red, very thick] (2.5, 0) ellipse (0.8 and 0.4);

\draw[Dandelion, very thick] (4.5, 0) ellipse (0.8 and 0.4);

\draw[Dandelion] (6.5, 0) ellipse (0.8 and 0.4);

\draw[red] (8.5, 0) ellipse (0.8 and 0.4);

\draw[ForestGreen, very thick] (3,0) to[out=30, in=150] (4,0);
\draw[ForestGreen, dotted, very thick] (2.9,-0.05) to[out=330, in=210] (4.1,-0.05);

\draw[NavyBlue] (5,0) to[out=30, in=150] (6,0);
\draw[NavyBlue, dotted] (4.9,-0.05) to[out=330, in=210] (6.1,-0.05);

\draw[ForestGreen] (7,0) to[out=30, in=150] (8,0);
\draw[ForestGreen, dotted] (6.9,-0.05) to[out=330, in=210] (8.1,-0.05);

\draw[NavyBlue, very thick] (4.5, 0) to[out=70 , in=290] (4.5, 2);
\draw[NavyBlue, dotted, very thick] (4.5, 0) to[out=110 , in=250] (4.5, 2);

\draw[black, very thick] (3.5, 2) to[out=-90 , in=180] (4.5, -1);
\draw[black, very thick] (7.5, 2) to[out=-90 , in=0] (6.5, -1);
\draw[black, very thick] (4.5, -1) to (6.5, -1);

\draw[brown, very thick] (2.7, -.05) to[out=30 , in=90] (5.4, 0.05);

\draw[brown, very thick] (2.7, -.125) to[out=-30 , in=270] (5.4, .05);

\draw[orange, very thick] (4.4, 0) to[out=150 , in=90] (1.6, 0);

\draw[orange, very thick] (4.4, -.15) to[out=-150 , in=-90] (1.6, 0);

\draw[gray, very thick] (3.5,0) ellipse (2.1 and 1.2);

\draw[Purple, very thick] (2.5, 0) to[out=70 , in=290] (2.5, 2);
\draw[Purple, dotted, very thick] (2.5, 0) to[out=110 , in=250] (2.5, 2);
\end{scope}
\end{tikzpicture}
\caption{The restriction of folding homomorphism for $A_{\Gamma}$ of type $E_8$ to one connected component. The thick curves are the ones we choose for the standard generators in verifying Property PP. The gray and purple curves are the curves we choose for $\{t,u,v\}$ and $\{s, t,u\}$ respectively. The black, brown, and orange curves are the unique curves for the irreducible rank $2$ subsets.}
\label{fig:H4}
\end{figure}

\section{Applications}\label{s:applications}

If the Generalized Tits Conjecture holds in full generality, one immediate application is that the subgroups of Artin groups are as complicated as subgroups of right-angled Artin groups. The latter are currently more well understood. In this section, we give a few applications of this idea. 

\subsection{Incoherence}

Recall that a group $G$ is \emph{coherent} if every finitely generated subgroup of $G$ is finitely presented. Droms had showed that the right-angled Artin group $\RA_L$ was coherent if and only if $L$ was a chordal graph \cite{droms}. Gordon showed that if the Artin group of type $H_3$ was incoherent, there was a similar classification of coherent Artin groups. Wise answered this in the affirmative in 2013.

\begin{Theorem}[\cite{wise}]
The Artin group of type $H_3$ is incoherent.
\end{Theorem}

Since $A$ of type $H_3$ satisfies the generalized Tits conjecture, we can give an alternative proof of Wise's theorem. In this case, the nerve $L_{\oslash}$ of the $\RAAG$ subgroup $\RA$ in $A$ is the cone on a pentagon. The $\RAAG$ based on a pentagon is well known to be incoherent, for example the Bestvina-Brady subgroup of the $\RAAG$ (the kernel of the map to $\zz$ which sends every generator to $1$) is finitely generated and not finitely presented.

\subsection{Hyperbolic surface subgroups}

In \cite{glr}, Gordon, Long, and Reid studied which Coxeter groups and Artin groups contained hyperbolic surface subgroups. They showed that all finite type Artin groups except types $A_1, I_2(m)$, and $H_3$ contained these subgroups. The first two classes do not contain such subgroups (More generally, any Artin group where the nerve $L$ is a tree does not contain a hyperbolic surface subgroup), and type $H_3$ was left as an open question. It follows from the Generalized Tits Conjecture that the answer to their question is yes.

\begin{Theorem}
Every Artin group with $s,t,u\in S$ so that $m_{st}, m_{tu}, m_{us}<\infty$ and at most one of them equals $2$, contains a surface subgroup.
\end{Theorem}
\begin{proof}
The subgroup $A_{\{s,t,u\}}$ satisfies the Generalized Tits Conjecture and the nerve of $\RA_{\{s,t,u\}}$ is a pentagon, hexagon or a cone on a pentagon. The $\RAAG$ based on a pentagon or a hexagon is commensurable to a right-angled Coxeter group which has a pentagon or a hexagon as a full subcomplex \cite{dj00}. It follows that the Coxeter group contains a hyperbolic surface subgroup, hence so does the $\RAAG$, and hence so does the Artin group. 
\end{proof}

\begin{corollary}
The Artin group of type $H_3$ contains a hyperbolic surface subgroup.
\end{corollary}

Gordon-Long-Reid also considered Artin groups of Euclidean type, and asked specifically if the Artin group of type $\tilde B_2$ or type $\tilde G_2$ contain a hyperbolic surface subgroup. 
The nerves of these Artin groups are triangles with labels $(2,4,4)$ and $(2,3,6)$ respectively. In unpublished work, Sang-hyun Kim can show the existence of such a subgroup in the $\tilde B_2$ case, but we believe the $\tilde G_2$ case was still open. 
In either case, these groups are locally reducible. Therefore, the answer to this question is also yes, and follows as above from the Generalized Tits Conjecture. 
\begin{corollary} The Artin groups of types $\tilde B_2$ and $\tilde G_2$ contain hyperbolic surface subgroups.
\end{corollary}

At this point, we do not have a complete characterization of which Artin groups contain hyperbolic surface subgroups. This is still open in the right-angled case. On the other hand, these should be easier to construct outside of the right-angled case. For example, if the nerve $L$ is a $4$-cycle, then the $\RAAG$ is $F_2 \times F_2$, which does not contain such a subgroup. On the other hand, if any edge has a label $> 2$ then the Artin group is locally reducible and the associated $\RAAG$ subgroup has nerve an $n$-gon with $n \ge 5$.

\subsection{Subgroups of type $F_n$ and not $F_{n+1}$}

Spherical Artin groups are generally too high-dimensional for coherence to be an interesting question. A more interesting question is when a spherical Artin group contains a subgroup which is type $F_n$ but not type $F_{n+1}$ where $n+2$ is the cohomological dimension (recall a group is type $F_n$ is it admits a classifying space with finite $n$-skeleton). For example, the type $H_3$ Artin group has cohomological dimension $3$, so coherence is an interesting question. It follows again from the generalized Tits Conjecture that Artin groups of type $H_4$ or $F_4$ have subgroups which are $F_2$ but not $F_3$. 

In \cite{zaremsky}, Zaremsky showed that the pure braid groups $PB_n$ contained subgroups $N$ so that $N$ was type $F_{m-3}$ but not $F_{m-2}$ for $3 \le m \le n$. The existence of these subgroups again follows from the fact that the generalized Tits conjecture holds for the braid groups (on the other hand Zaremsky's examples are normal and ours are not). In this case, the nerve of the $\RAAG$ subgroup is the cone on a flag triangulation of $S^{n-1}$. The Bestvina-Brady subgroup of this $\RAAG$ is type $F_{n-3}$ but not type $F_{n-2}$. To get subgroups that are $F_{m-3}$ but not $F_{m-2}$ for $3 \le m < n$, one can instead map $\RA$ to $\zz$ by sending some generators to $0$.

Since the generalized Tits conjecture holds for type $D_n$, we have an analogous theorem with the same proof as above. 

\begin{Theorem}
The Artin group of type $D_n$ contains subgroups $N$ so that $N$ is type $F_{m-3}$ but not $F_{m-2}$ for $3 \le m \le n$.

\end{Theorem}

Of course, the Artin group of type $D_n$ contains the braid group $A_{n-1}$, so the only improvement on what Zaremsky's theorem provides is $n = m$.

\subsection{Subgroup separability}
A subgroup $H$ of a group $G$ is \emph{separable} if $H$ is closed in the profinite topology of $G$, or equivalently if every $H$ is equal to the intersection of all the subgroups of finite index of $G$ containing $H$.
 A group $G$ is \emph{subgroup separable} if every finitely generated subgroup of $G$ is separable in $G$.

In \cite{al} Almeida-Lima used our result to formulate a criterion for subgroup separability of Artin group, generalizing the criterion for RAAG's due to Metaftsis-Raptis \cite{mr}. 
They prove that an Artin group $A$ is subgroup separable if and only if $A$ is obtained from Artin groups of rank $2$ via a sequence of two operations: taking free products and taking direct product with $\mathbb Z$.

\section{Questions}

We end the paper with some open questions.  

\begin{Question}
Does the Generalized Tits Conjecture hold for spherical Artin groups of type $E_n$?
\end{Question}

If one can show the conjecture holds for all spherical Artin groups, a natural next step is the Artin groups of FC type.

\begin{Question}
Does the conjecture hold for all Artin groups of FC type?
\end{Question}

It would be very interesting to know some geometric properties of these subgroups:

\begin{Question}
Each standard free abelian subgroup quasi-isometrically embeds into the Artin group $A$. When the generalized Tits conjecture is true, are the entire $\RAAG$ subgroups quasi-isometrically embedded? Does a quasi-isometry between two Artin groups coarsely preserve these $\RAAG$ subgroups?
\end{Question}

\end{document}